\documentclass[oneside,english]{amsart}
\usepackage[T1]{fontenc}
\usepackage[latin9]{inputenc}
\usepackage{geometry}
\usepackage{babel}
\usepackage{verbatim}
\usepackage{amsthm}
\usepackage{amssymb}
\usepackage{graphicx}
\usepackage[unicode=true,pdfusetitle,
 bookmarks=true,bookmarksnumbered=false,bookmarksopen=false,
 breaklinks=false,pdfborder={0 0 1},backref=false,colorlinks=false]
 {hyperref}

\makeatletter

\providecommand{\tabularnewline}{\\}

\numberwithin{equation}{section}
\numberwithin{figure}{section}
\theoremstyle{plain}
\newtheorem{thm}{\protect\theoremname}
\theoremstyle{plain}
\newtheorem{cor}{\protect\corollaryname}
\theoremstyle{remark}
\newtheorem{rem}{\protect\remarkname}
\theoremstyle{plain}
\newtheorem{lem}{\protect\lemmaname}
\theoremstyle{remark}
\newtheorem{claim}{\protect\claimname}

\usepackage{xifthen}

\newcommand{\manuallabel}[2]{\def\@currentlabel{#2}\label{#1}}

\makeatother

\providecommand{\claimname}{Claim}
\providecommand{\corollaryname}{Corollary}
\providecommand{\lemmaname}{Lemma}
\providecommand{\remarkname}{Remark}
\providecommand{\theoremname}{Theorem}

\begin{document}
\global\long\def\dd{\mathrm{d}}

\global\long\def\bbc{\mathbb{C}}

\global\long\def\bbe{\mathbb{E}}

\global\long\def\bbn{\mathbb{N}}

\global\long\def\bbp{\mathbb{P}}

\global\long\def\bbr{\mathbb{R}}

\global\long\def\calC{\mathcal{C}}

\global\long\def\calD{\mathcal{D}}

\global\long\def\calF{\mathcal{F}}

\global\long\def\calG{\mathcal{G}}

\global\long\def\calH{\mathcal{H}}

\global\long\def\calL{\mathcal{L}}

\global\long\def\calM{\mathcal{M}}

\global\long\def\calR{\mathcal{R}}

\global\long\def\calZ{\mathcal{Z}}

\global\long\def\frakD{\mathfrak{D}}

\global\long\def\ind#1#2{\mathbf{1}_{#1}\left(#2\right)}

\global\long\def\veps{\varepsilon}

\global\long\def\defeq{\overset{\operatorname{def}}{=}}

\global\long\def\area#1{\operatorname{Area}\left(#1\right)}

\global\long\def\pr#1{\bbp\left[#1\right]}

\global\long\def\condpr#1#2{\bbp_{#2}\left[#1\right]}

\global\long\def\ex#1{\bbe\left[#1\right]}

\global\long\def\condex#1#2{\bbe_{#2}\left[#1\right]}

\global\long\def\exabs#1{\bbe\left|#1\right|}

\global\long\def\var#1{{\rm Var}\left[#1\right]}

\global\long\def\Dnorm#1{\frakD\left(#1\right)}

\global\long\def\gef{F_{\bbc}}

\global\long\def\lgef{F_{L}}

\global\long\def\binmax#1#2{#1\vee#2}

\begin{comment}
\global\long\def\binmin#1#2{#1\wedge#2}
\end{comment}

\global\long\def\cmplxgaus#1{N_{\bbc}\left(0,#1\right)}

\global\long\def\stdgaus{\cmplxgaus 1}

\global\long\def\truncpoly{P_{N}}

\global\long\def\seriestail{T_{N}}

\global\long\def\truncscaledpoly{P_{N,L}}

\global\long\def\vec#1{\underline{#1}}

\global\long\def\lvec#1#2{\left(#1_{1},\dots,#1_{#2}\right)}

\global\long\def\lvecl#1#2#3{\left(#1_{#2},\dots,#1_{#3}\right)}

\global\long\def\normalconst{A_{L}^{N}}

\global\long\def\refmeas{\mu_{L}}

\global\long\def\lebmeas{m}

\global\long\def\subslebmeas#1{m_{\left\{  #1\right\}  }}

\global\long\def\eqmeas#1{\mu_{\mbox{{\rm eq}}}^{#1}}

\global\long\def\minmeas#1{\mu_{Z_{#1}}}

\global\long\def\minmeassup#1#2{\mu_{Z_{#1}}^{#2}}

\global\long\def\minmeasaltsup#1#2{\widetilde{\mu}_{Z_{#1}}^{#2}}

\global\long\def\minmeasglob#1{\mu_{Z_{#1}}^{\bbc}}

\global\long\def\condmeas{\mu_{\mathrm{cond}}^{\bbc}}

\global\long\def\optmeas{\mu_{\mathrm{min}}}

\global\long\def\probmeas{\calM_{1}\left(\bbc\right)}

\global\long\def\Radonmeas{\calR\left(\bbc\right)}

\global\long\def\probmeasfiniteenerg{M_{1}^{\Sigma}\left(\bbc\right)}

\global\long\def\logpot#1#2{U_{#1}\left(#2\right)}

\global\long\def\logenerg#1{\Sigma\left(#1\right)}

\global\long\def\setd#1#2{\left\{  #1\,:\,#2\right\}  }

\global\long\def\disc#1#2{D\left(#1,#2\right)}

\global\long\def\funcdiscname{I^{\star}}

\global\long\def\funcdisc#1{\funcdiscname\left(#1\right)}

\global\long\def\funccont#1{I\left(#1\right)}

\global\long\def\funccontparamname#1{I_{#1}}

\global\long\def\funccontparam#1#2{\funccontparamname{#1}\left(#2\right)}

\global\long\def\funcenerg#1{J\left(#1\right)}

\global\long\def\linstats#1#2#3{n_{#1}\left(#2;#3\right)}

\global\long\def\regevent{E_{\mathrm{reg}}}

\global\long\def\kregevent#1{E_{\mathrm{reg}}^{#1}}

\global\long\def\holeevent#1{H_{#1}}

\global\long\def\GEFzeroset{\calZ}

\global\long\def\zeroset#1{\calZ\left(#1\right)}

\global\long\def\condzeroset#1{\GEFzeroset_{#1}}

\global\long\def\cntmeas#1{\left[#1\right]}

\author{Subhroshekhar Ghosh}
\address[S.~Ghosh]{\\} \email{subhrowork@gmail.com}

\author{Alon Nishry}
\address[A.~Nishry]{\\}
\email{alonish@tauex.tau.ac.il}

\title{Gaussian complex zeros on the hole event:\\
the emergence of a forbidden region }
\begin{abstract}
Consider the Gaussian Entire Function
\[
\gef\left(z\right)=\sum_{k=0}^{\infty}\xi_{k}\frac{z^{k}}{\sqrt{k!}},\quad z\in\bbc,
\]
where $\left\{ \xi_{k}\right\} $ is a sequence of independent standard
complex Gaussians. This random Taylor series is distinguished by the
invariance of its zero set with respect to the isometries of the plane
$\bbc$. It has been of considerable interest to study the statistical
properties of the zero set, particularly in comparison to other planar
point processes.

We show that the law of the zero set, conditioned on the function
$\gef$ having no zeros in a disk of radius $r$, and normalized appropriately,
converges to an explicit limiting Radon measure on $\bbc$, as $r\to\infty$.
A remarkable feature of this limiting measure is the existence of
a large ``forbidden region'' between a singular part supported on
the boundary of the (scaled) hole and the equilibrium measure far
from the hole. In particular, this answers a question posed by Nazarov
and Sodin, and is in stark contrast to the corresponding result of
Jancovici, Lebowitz, and Manificat in the random matrix setting: there
is no such region for the Ginibre ensemble.
\end{abstract}

\maketitle

\section{\label{sec:intro}Introduction}

\markright{GAUSSIAN COMPLEX ZEROS ON THE HOLE EVENT}In recent years,
particle systems (also known as point processes) involving local repulsion
have attracted a lot of attention (\cite{bogomolny1996quantum,lebowitz1983charge,nazarov2010random,peres2005zeros,soshnikov2000determinantal,wigner1934interaction,wigner1958distribution},
to provide a partial list). Two of the most significant mathematical
models of translation invariant planar point processes embodying local
repulsion are the Ginibre ensemble and the zeros of the standard Gaussian
Entire Function (GEF). Both of these processes originate in physics.
The Ginibre ensemble was introduced by J. Ginibre (\cite{ginibre1965statistical}),
as a non-Hermitian Gaussian matrix model; it also turns out to be
the 2D Coulomb gas at a specific temperature. The GEF was introduced
by E. Bogomolny, O. Bohigas, and P. Leboeuf (\cite{bogomolny1992distribution,bogomolny1996quantum}),
in the form of Weyl polynomials. These two ensembles share many similar
properties. For instance, their correlations decay as $\exp\left(-c\cdot\mbox{distance}^{2}\right)$
(see \cite{hough2009zeros,nazarov2010gaussian,nazarov2010random}).

For a point process, one quantity of interest is the decay rate of
the \textit{hole probability}, that is, the probability that a disk
of radius $r$ contains no points, as $r\to\infty$. One can consider
this quantity as a rough measure of the mutual repulsion (or ``rigidity'')
in the process (see \cite[Sec. 7.2]{hough2009zeros}). Both for the
Ginibre ensemble (\cite{jancovici1993large,shirai2006large}), and
for the GEF zero process (\cite{nishry2010asymptotics,sodin2005random})
the hole probability decays like $\exp\left(-cr^{4}\left(1+o\left(1\right)\right)\right)$
(for the Poisson point process, which exhibits no rigidity, the decay
rate is $\exp\left(-cr^{2}\right)$). A natural problem that arises
is how to describe the behavior of the point process conditioned to
have such a large hole. Progress on this problem will allow us to
describe the typical configurations that produce this rare event.

Among the main results of this paper is a description of these configurations
for the zero set of the GEF, conditioned upon the hole event. We show
that beyond a singular component on the boundary of the hole, there
is a second ``forbidden region'' $\left\{ r<\left|z\right|<\sqrt{e}r\right\} $,
in which the density of the zeros is negligibly small. This phenomenon
is rather surprising, and to the best of our knowledge, this is, in
fact, the first example where such a forbidden region in particle
systems has been rigorously established, or even heuristically understood
(the appearance of some type of a forbidden region or a gap was suspected
by Nazarov and Sodin, see also \cite[Fig. 2]{hough2005large}).

The work of Jancovici, Lebowitz, and Manificat \cite{jancovici1993large}
treats in particular the case of the Ginibre ensemble. It shows that
conditioning on the hole event, also leads to the formation of a singular
component on the boundary of the hole. However, in this case there
are no macroscopic restrictions outside the hole (see also \cite{majumdar2011many},
and Section \ref{sec:discussion} for a short discussion of the one-dimensional
case). Figures \ref{fig:GEFcondzeros} and \ref{fig:GinibreCond}
present a simulation of the hole event (with $r=13$) for the GEF
and the Ginibre ensemble, respectively. For more details about this
simulation, see Section \ref{sec:discussion}.

\medskip{}

\begin{tabular}{ccc}
\includegraphics[scale=0.32]{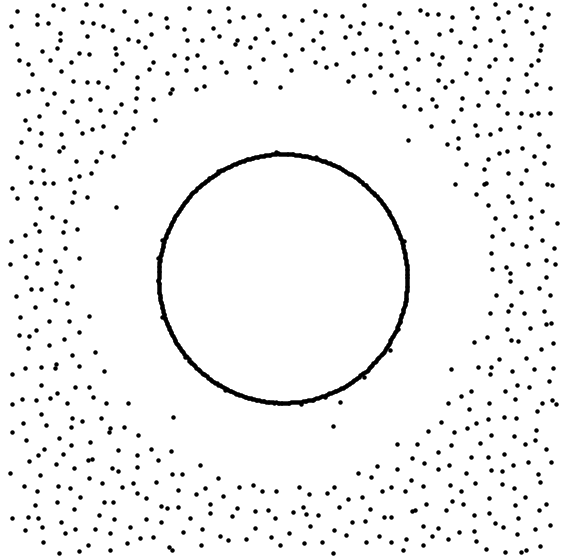} &  & \includegraphics[scale=0.32]{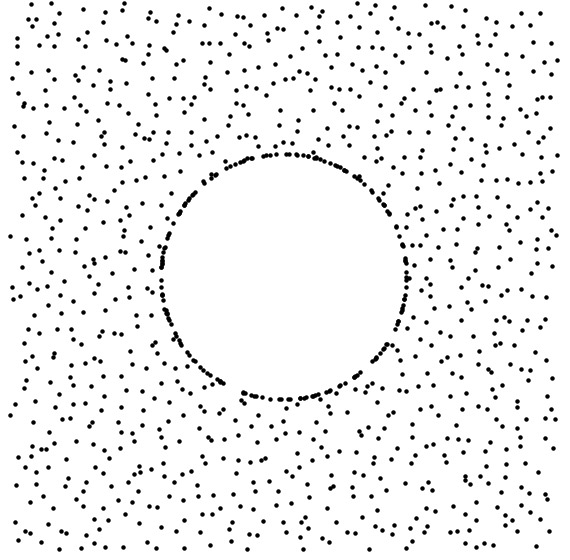}\tabularnewline
\begin{tabular}{c}
Figure 1.1\manuallabel{fig:GEFcondzeros}{1.1} - Zeros of GEF on hole
event\tabularnewline
The black ``circle'' are zeros of the GEF\tabularnewline
\end{tabular} &  & %
\begin{tabular}{c}
Figure 1.2\manuallabel{fig:GinibreCond}{1.2} - Ginibre ensemble on
hole event\tabularnewline
See Section \ref{sec:discussion} for details\tabularnewline
\end{tabular}\tabularnewline
\end{tabular}

\medskip{}
Our result for the hole is proved in a more general setting. Given
$p\ge0$, $p\ne1$, we will condition on the event that the number
of zeros in the disk $\left\{ \left|z\right|<r\right\} $ equals $\left\lfloor pr^{2}\right\rfloor $,
and study in details the conditional distribution of the zeros as
$r\to\infty$. The case $p=0$ corresponds to the hole, $p<1$ corresponds
to a ``deficit'' of zeros, while $p>1$ corresponds to an ``abundance''
of zeros (so called ``overcrowding''). To avoid unnecessary long
preliminaries, here we will only bring a special case of our results
pertaining to the case $p=0$.

For a compactly supported test-function $\varphi$, we put
\[
\linstats{\gef}{\varphi}r=\sum_{z\in\GEFzeroset}\varphi\left(\frac{z}{r}\right),
\]
where $\calZ$ is the random zero set of the GEF. The random variables
$\linstats{\gef}{\varphi}r$ are called \textit{linear statistics.}
In the special case where $\varphi$ is the indicator function of
the unit disk, the corresponding linear statistics is the radial zero
counting function. The classical Edelman-Kostlan formula (\cite[Sec. 2.4]{hough2009zeros})
gives the expected value
\[
\ex{\linstats{\gef}{\varphi}r}=\int_{\bbc}\varphi\left(\frac{w}{r}\right)\,\frac{\dd\lebmeas\left(w\right)}{\pi}=r^{2}\cdot\int_{\bbc}\varphi\left(w\right)\,\frac{\dd\lebmeas\left(w\right)}{\pi},
\]
where $\lebmeas$ is the Lebesgue measure on $\bbc$. We also put
\[
\dd\minmeasglob 0\left(z\right)=e\cdot\dd\subslebmeas{\left|z\right|=1}+\ind{\left\{ \left|w\right|\ge\sqrt{e}\right\} }z\cdot\frac{\dd\lebmeas\left(z\right)}{\pi},
\]
where $\subslebmeas{\left|z\right|=1}$ is the Lebesgue measure on
the unit circle normalized to be a probability measure. It is noteworthy
that the equilibrium measure of the zeros $\tfrac{1}{\pi}\dd\lebmeas$
is preserved outside the disk $\left\{ \left|z\right|<\sqrt{e}\right\} $,
while the total mass of the singular component equals the equilibrium
mass of this disk.

Let $\holeevent r$ denote the hole event, when there are no zeros
of $\gef$ in the disk $\left\{ \left|z\right|<r\right\} $. By $\condex{\cdot}{\holeevent r}$
(resp. $\condpr{\cdot}{\holeevent r}$) we denote the conditional
expectation (resp. probability) on $\holeevent r$. Our main result
is the following
\begin{thm}
\label{thm:conv_emp_meas_for_hole} Fix $\varphi\in C_{c}^{2}\left(\bbc\right)$
a twice continuously differentiable test function with compact support.
As $r\to\infty$,
\[
\condex{\linstats{\gef}{\varphi}r}{\holeevent r}=r^{2}\int_{\bbc}\varphi\left(w\right)\,\dd\minmeasglob 0\left(w\right)+O\left(r\log^{2}r\right).
\]
\end{thm}
Let us write $\cntmeas{\GEFzeroset}$ for the (random) counting measure
of $\calZ$. In addition, by $\condzeroset r$ we denote the zero
set conditioned on $\holeevent r$, and write $\cntmeas{\condzeroset r}$
for the corresponding counting measure. Recall that the space $\Radonmeas$
of Radon (positive, locally finite) measures on $\bbc$ can be endowed
with the vague topology (such that the mapping $\left(\nu,\phi\right)\mapsto\int\phi\,\dd\nu$
is continuous). The following corollary of Theorem \ref{thm:conv_emp_meas_for_hole}
describes the limiting distribution of the zeros (with appropriate
scaling), conditioned on the hole event for large $r$.
\begin{cor}
As $r\to\infty$, the scaled conditional zero counting measure $\frac{1}{r^{2}}\cntmeas{\condzeroset r}\left(\frac{\cdot}{r}\right)\to\minmeasglob 0$
in distribution, where the convergence is in the vague topology.
\end{cor}
Let $n_{\gef}\left(G\right)=\cntmeas{\GEFzeroset}\left(G\right)$
be the number of zeros of the GEF inside a domain $G\subset\bbc$.
Our second result gives a quantitative upper bound for the actual
number of zeros in the forbidden annulus, where the limiting conditional
expectation $\dd\minmeasglob 0$ vanishes.
\begin{thm}
\label{thm:few_zeros_in_the_gap}Suppose $r$ is sufficiently large,
$\varepsilon\in\left(r^{-2},1\right)$, that $\gamma\in\left(1+\frac{1}{2}\log\frac{1}{\varepsilon}\left(\log r\right)^{-1},2\right]$,
and consider the annulus
\[
A_{r,\varepsilon}=\setd{z\in\bbc}{r\left(1+\varepsilon\right)\le\left|z\right|\le\sqrt{e}r\left(1-\varepsilon\right)}.
\]
We have

\[
\condpr{n_{\gef}\left(A_{r,\varepsilon}\right)\ge r^{\gamma}}{\holeevent r}\le\exp\left(-C\varepsilon r^{2\gamma}\right),
\]
where $C>0$ is some numerical constant.
\end{thm}
In fact we can prove deviation bounds of similar nature for general
smooth linear statistics $\varphi\in C_{0}^{2}\left(\bbc\right)$,
see Theorem \ref{thm:devs_in_emp_meas_hole_case}, Section \ref{sec:proof_idea}
for the special case of conditioning on the hole event, and Theorem
\ref{thm:devs_in_emp_meas}, Section \ref{sec:conv_emp_meas} for
the general setting.
\begin{rem}
In particular, Theorem \ref{thm:few_zeros_in_the_gap} implies that
for every fixed $\varepsilon,\delta>0$ we have 
\[
\condex{n_{\gef}\left(A_{r,\varepsilon}\right)}{\holeevent r}=O\left(r^{1+\delta}\right).
\]
This should be compared with the (unconditional) expected number of
zeros
\[
\ex{n_{\gef}\left(A_{r,\varepsilon}\right)}=\left(e-1\right)r^{2}+O\left(\varepsilon r^{2}\right).
\]
\end{rem}
Our approach is based on precise estimates for the zero set of the
GEF via polynomial approximations. We obtain effective deviation bounds
for linear statistics of the zeros, which are inspired by a large
deviation principle (LDP) for zeros of Gaussian random polynomials
due to Zeitouni and Zelditch (\cite{zeitouni2010large}, similar LDPs
were previously obtained for eigenvalues of random matrices in \cite{BenArous1997,arous1998large,hiai1998maximizing}).
This approach enables us to reduce a problem on the distribution of
the zeros to the solution of a constrained optimization problem in
the space of probability measures. We develop a more precise version
of the LDP, which allows us to make the transition from polynomials
to entire functions (using results from complex analysis about the
variation of the zeros of an analytic function under analytic perturbations).
In addition, this allows us to control the error terms, leading in
particular to Theorem \ref{thm:few_zeros_in_the_gap}. A key difficulty
that arises in this program is that the constrained optimization problem
involves a non-standard, non-differentiable functional, which requires
the application of potential theoretic techniques.

As mentioned, the techniques of this paper can be effectively used
to study other properties of the GEF. We defer the statements of these
results to Section \ref{sec:upp_bnd_for_larg_flucs} and the discussion
to Section \ref{sec:discussion}.

\subsection*{Acknowledgements}

We learned about the question concerning the possible appearance of
a gap conditioned on the hole event from Fedor Nazarov and Mikhail
Sodin. We thank Amir Dembo and Ofer Zeitouni for suggesting that solutions
of constrained optimization problems on the space of probability measures,
are a possible way to approach the problems in this paper, for emphasizing
the relevance of the paper \cite{zeitouni2010large}, and for helpful
discussions. We thank Zakhar Kabluchko, Fedor Nazarov, and Mikhail
Sodin for very helpful discussions. We thank Zemer Kosloff, Joel Lebowitz,
and Ramon van Handel for helpful suggestions that led to an improved
presentation of the results. We thank the anonymous referee for their
meticulous reading of the manuscript and numerous corrections and
suggestions. The work of S. G. was supported in part by the ARO grant
W911NF-14-1-0094.

\subsection*{Notation and general remarks}

The letters $C$ and $c$ denote positive numerical constants, that
do not depend on $r$ and on $p$. The values of these constants are
not essential to the proof, and their value may vary from line to
line, or even within the same line. We denote by $A,B,C_{1},C_{2}$,
etc. constants that we keep fixed throughout the proof in which they
appear.

We write $\bbn=\left\{ 0,1,2,\dots\right\} $ and $\bbn^{+}=\left\{ 1,2,3,\dots\right\} $.
For a finite set $J$, we sometimes use $\#J$ to denote the size
of set (number of elements). We denote by $\disc ar$ the open disk
$\setd{z\in\bbc}{\left|z-a\right|<r}$. The letter $D$ stands for
the unit disk $\disc 01$. $\left\lfloor x\right\rfloor $ is the
integer part of a number $x\in\bbr$. For $a,b\in\bbr$, we write
$\binmax ab$ %
\begin{comment}
($\binmin ab$)
\end{comment}
{} for the maximum %
\begin{comment}
(minimum)
\end{comment}
{} of the two.

Let $f\left(r\right),g\left(r\right)$ be positive functions. The
notation $g\left(r\right)=O\left(f\left(r\right)\right)$ means there
is a constant $C=C\left(p\right)$ such that $g\left(r\right)\le Cf\left(r\right)$
for $r$ sufficiently large (possibly depending on $p$ and other
parameters). The notation $g\left(r\right)=o\left(f\left(r\right)\right)$
means $\frac{g\left(r\right)}{f\left(r\right)}\xrightarrow[r\to\infty]{}0$.

A function $\varphi:\bbc\mapsto\bbr$ admits $\omega\left(t\right):\left[0,\infty\right)\mapsto\left[0,\infty\right)$
as its modulus of continuity, if for all $z,w\in\bbc$,
\[
\left|\varphi\left(z\right)-\varphi\left(w\right)\right|\le\omega\left(\left|z-w\right|\right).
\]
By writing $\omega\left(\varphi;t\right)$ we mean a function which
is admitted as the modulus of continuity of $\varphi$. For example,
$\varphi$ is Hölder continuous if $\omega\left(\varphi;t\right)=C_{\varphi}t^{\alpha}$,
for some $C_{\varphi}>0$, $\alpha\in\left(0,1\right]$.

We write $\lebmeas$ for the Lebesgue measure on $\bbc$, while $\lebmeas_{\left|z-w\right|=t}$
is the Lebesgue measure on the circle $\left|z-w\right|=t$, normalized
to have mass $1$.

\subsubsection*{Analytic functions.}

Let $f$ be an entire function. We use the standard notation
\[
M_{f}\left(r\right)=\max\setd{\left|f\left(z\right)\right|}{\left|z\right|\le r},\quad r\ge0.
\]
We write $\zeroset f=\setd{z\in\bbc}{f\left(z\right)=0}$ for the
collection of zeros of $f$ (zero set). In principle, multiple zeros
appear as many times as their multiplicity (but in this paper all
zeros are simple). We denote by $\cntmeas{\zeroset f}$ the counting
measure of the zeros of $f$, that is for a domain $G$: 
\[
\cntmeas{\zeroset f}\left(G\right)=\#\setd{z\in G}{f\left(z\right)=0}.
\]
We write $n_{f}\left(r\right)$ for the number of zeros of the function
$f$ inside the closed disk $\overline{\disc 0r}=\left\{ \left|z\right|\le r\right\} $.
For the GEF $\gef$ we usually just write $M\left(r\right)=M_{\gef}\left(r\right)$
and $n\left(r\right)=n_{\gef}\left(r\right)$ (these are random variables).
Let $\varphi$ be a test function with compact support, the linear
statistics of $f$ with respect to (w.r.t.) $\varphi$ is given by
\[
\linstats f{\varphi}r=\sum_{z\in\zeroset f}\varphi\left(\frac{z}{r}\right),\quad r>0.
\]

\subsubsection*{Probability and negligible events.}

We denote events by $E,F,$ etc., by $E^{c}$ the complement of the
event $E$, and by $\biguplus E_{k}$ the disjoint union of the events
$E_{k}$. We write $\pr E$ for the probability of the event $E$
(the probability space will always be clear from the context). An
event $E=E\left(r\right)$ will be called \textit{negligible} with
respect to $F\left(p;r\right)=\left\{ n\left(r\right)\le pr^{2}\right\} $
($M\left(p;r\right)=\left\{ n\left(r\right)\ge pr^{2}\right\} $)
if $\pr E=o\left(\pr{F\left(p;r\right)}\right)$ (respectively, $\pr E=o\left(\pr{M\left(p;r\right)}\right)$).
In that case we have $\pr{F\left(p;r\right)}\le\pr{F\left(p;r\right)\cap E^{c}}+\pr E=\left(1+o\left(1\right)\right)\pr{F\left(p;r\right)\cap E^{c}}.$

If $X$ is a random variable, then $\ex X$ is its mean (expected
value), and $\var X$ is its variance (if they exist). We write $\left.X\right|_{F}$
to denote the random variable $X$ conditioned on the event $F$.
We have $\condex XF=\left(\pr F\right)^{-1}\cdot\ex{X\cdot\ind F{\cdot}}$,
where $\ind F{\cdot}$ is the indicator random variable of the event
$F$.

\subsubsection*{Measures.}

We consider mainly probability measures on the complex plane, which
we denote by $\probmeas$. Sometimes we consider Radon (locally finite)
measures. All the measures we work with are assumed to be Borel measures.
All sets are assumed to be Borel measurable. We denote collections
(or sets) of measures by $\calC,\calD$, etc.

\subsubsection*{Potential theory.}

Let $\mu\in\probmeas$. We write
\[
\logpot{\mu}z=\int_{\bbc}\log\left|z-w\right|\,\dd\mu\left(w\right),\quad\logenerg{\mu}=\int_{\bbc}\int_{\bbc}\log\left|z-w\right|\,\dd\mu\left(z\right)\dd\mu\left(w\right)=\int_{\bbc}\logpot{\mu}z\,\dd\mu\left(z\right),
\]
for the logarithmic potential and the logarithmic energy of the measure,
respectively. A measure is said to have finite logarithmic energy
if $\left|\logenerg{\mu}\right|<\infty$. Sometime we use the same
notation for the logarithmic potential and energy of signed measures
(with finite total variation). For more details, see Appendix \ref{sec:poten_theory}.

\section{\label{sec:proof_idea}Idea of the proof}

Recall the GEF is the Gaussian Entire function, given by the Taylor
series
\begin{equation}
\gef\left(z\right)=\sum_{k=0}^{\infty}\xi_{k}\frac{z^{k}}{\sqrt{k!}},\quad z\in\bbc,\label{eq:gef_Taylor_series}
\end{equation}
where $\left\{ \xi_{k}\right\} $ is a sequence of independent standard
complex Gaussians (i.e., the density of $\xi_{k}$ with respect to
$\lebmeas$, the Lebesgue measure on the complex plane $\bbc$, is
$\frac{1}{\pi}e^{-\left|z\right|^{2}}$). The hole event at radius
$r$, denoted $\holeevent r$, is the event where $\gef$ has no zeros
inside the disk $\disc 0r=\left\{ \left|z\right|<r\right\} $. Suppose
$\varphi\in C_{0}^{2}\left(\bbc\right)$ is a test function, which
is twice continuously differentiable and with compact support, and
define
\[
\Dnorm{\varphi}=\left\Vert \nabla\varphi\right\Vert _{L^{2}\left(\lebmeas\right)}^{2}=\int_{\bbc}\left(\varphi_{x}^{2}+\varphi_{y}^{2}\right)\,\dd m\left(z\right).
\]
Let $\linstats{\gef}{\varphi}r$ be the linear statistics associated
with $\varphi$, we would like to consider an event where the linear
statistics is far from the conditional limiting measure $\minmeasglob 0$.
Theorems \ref{thm:conv_emp_meas_for_hole} and \ref{thm:few_zeros_in_the_gap}
will be deduced from the following (conditional) deviation inequality.
\begin{thm}
\label{thm:devs_in_emp_meas_hole_case}Suppose $C^{\prime}>0$ is
fixed, and that $r$ is sufficiently large. For $\lambda\in\left(0,C^{\prime}r^{2}\right)$
we have, 
\[
\condpr{\left|\linstats{\gef}{\varphi}r-r^{2}\int_{\bbc}\varphi\left(w\right)\,\dd\minmeasglob 0\left(w\right)\right|\ge\lambda}{H_{r}}\le\exp\left(-\frac{C}{\Dnorm{\varphi}}\cdot\lambda^{2}+C_{\varphi}r^{2}\log^{2}r\right),
\]
where $C>0$ is a numerical constant and $C_{\varphi}>0$ is a constant
depending only on $\varphi$.
\end{thm}
A key ingredient in our proof is the fact that the zero set of the
random polynomial, obtained by truncating the Taylor series (\ref{eq:gef_Taylor_series})
at a large degree $N$ (which depends on $r$), serves as a good approximation
for the zero set of $\gef$. The advantage of working with the polynomial
is that one can write down a closed form expression for the joint
density of its zeros. We can then identify any given instance of such
a random zero set with the corresponding empirical measure (i.e.,
the probability measure with equal weights on the zeros). Another
key ingredient is that at the exponential scale, this joint density
can be further approximated by a certain functional that acts on the
empirical measure. For the problem at hand, we can focus our attention
on the behavior of this functional on an appropriate subset of empirical
measures, namely those that put no mass on the ``hole''.

\subsection{Truncation of the Taylor series}

Suppose $r>0$ is large. Let $\varphi$ be a test function supported
on the disk $\disc 0B$, where $B\ge1$ is fixed. Let $N\in\bbn$
be a large parameter (depending on $r$). We also introduce the large
parameter $L=r+O\left(\frac{1}{r}\right)$. We will work with the
scaled polynomials
\begin{equation}
\truncscaledpoly\left(z\right)=\sum_{k=0}^{N}\xi_{k}\frac{\left(Lz\right)^{k}}{\sqrt{k!}},\quad z\in\bbc.\label{eq:idea_trunc_poly}
\end{equation}
Roughly speaking, $L$ would correspond to the size of the hole, and
$N$ is the degree of polynomial truncation of $\gef$. As a matter
of fact, we will choose $N$ such that $\alpha\defeq Nr^{-2}$ is
of order $\log r$. In what follows, we denote by $\regevent$ an
event for which the (scaled) zeros of the polynomial $\truncscaledpoly$
serve as a good approximation for the zeros of $\gef$ inside the
disk $\disc 0{Br}$. We can choose $\regevent$ so that $\pr{\kregevent c}\le\exp\left(-Ar^{4}\right)$,
for some large constant $A>0$. For the details see Subsection \ref{subsec:upp_bnd_GEF_trunc}.

\subsection{The joint distribution of the zeros of $\protect\truncscaledpoly$}

Put $\dd\refmeas\left(w\right)=\frac{L^{2}}{\pi}e^{-L^{2}\left|w\right|^{2}}\,\dd m\left(w\right)$,
where $m$ is Lebesgue measure on $\bbc$ ($\refmeas$ is a probability
measure). Let us denote by $z_{1},\dots,z_{N}$ the zeros of the polynomial
$\truncscaledpoly$, and in addition write $\vec z=\left(z_{1},\dots,z_{N}\right)$.
Lemma \ref{lem:poly_change_of_vars} (in Appendix \ref{sec:joint_density})
shows that the joint probability density of the zeros (in uniform
random order), with respect to Lebesgue measure on $\bbc^{N}$, is
given by
\[
f\left(\vec z\right)=f\left(z_{1},\dots,z_{N}\right)=\normalconst\left|\Delta\left(\vec z\right)\right|^{2}\left(\int_{\bbc}\prod_{j=1}^{N}\left|q_{\vec z}\left(w\right)\right|^{2}\,\dd\refmeas\left(w\right)\right)^{-\left(N+1\right)},
\]
where 
\[
q_{\vec z}\left(w\right)=\prod_{j=1}^{N}\left(w-z_{j}\right),
\]
 $\left|\Delta\left(\vec z\right)\right|^{2}=\prod_{j\ne k}\left|z_{j}-z_{k}\right|$,
and $\normalconst$ is the normalization constant.

For a probability measure $\mu\in\probmeas$ we denote by $\logpot{\mu}z,$
$\logenerg{\mu}$ its logarithmic potential and logarithmic energy,
respectively (for the definitions we refer to the notation section
of the introduction). Let $\mu_{\vec z}=\frac{1}{N}\sum_{j=1}^{N}\delta_{z_{j}}$
be the empirical probability measure of the zeros. Instead of working
with the squared Vandermonde $\left|\Delta\left(\vec z\right)\right|^{2}$,
we would like to work with the logarithmic energy functional $\logenerg{\mu}$
. However, the latter is not well-defined for discrete measures. Thus,
it will be required to introduce the smoothed empirical measure $\mu_{\vec z}^{t}=\mu_{\vec z}\star\subslebmeas{\left|z\right|=t}$,
where $t=t\left(r\right)>0$ is a small parameter.

Consider the functional $\funccontparamname{\alpha}:\probmeas\to\left[0,\infty\right]$
given by
\[
\funccontparam{\alpha}{\nu}=2\sup_{w\in\bbc}\left\{ \logpot{\nu}w-\frac{\left|w\right|^{2}}{2\alpha}\right\} -\logenerg{\nu}.
\]
The uniform probability measure on the disk $\disc 0{\sqrt{\alpha}}$,
denoted by $\mu_{\alpha}$, is known to be the unique global minimizer
of $\funccontparamname{\alpha}$. In Section \ref{subsec:est_joint_dist_P_L}
we show that one can bound $f\left(\vec z\right)$ by
\[
\exp\left(-N^{2}\left[\funccontparam{\alpha}{\mu_{\vec z}^{t}}-\funccontparam{\alpha}{\mu_{\alpha}}+o\left(1\right)\right]\right).
\]
Let $Z\subset\bbc^{N}$ be a nice subset of possible `configurations'
of the zeros. Roughly speaking, for such $Z$, we show that
\begin{align}
\pr{\left\{ \vec z\in Z\right\} \cap\regevent} & =\int_{Z\cap\regevent}f\left(\vec z\right)\,\dd m\left(\vec z\right)\le\exp\left(-N^{2}\left[\inf_{\vec z\in Z}\funccontparam{\alpha}{\mu_{\vec z}^{t}}-\funccontparam{\alpha}{\mu_{\alpha}}+o\left(1\right)\right]\right).\label{eq:idea_upp_bnd_prob}
\end{align}
For the reader who is acquainted with the theory of large deviations,
this upper bound is similar in spirit to the large deviations upper
bound for empirical measures of random polynomials obtained in \cite{zeitouni2010large}.

\subsection{Conditioning on the hole event and a constrained optimization problem}

On the hole event $H_{r}$, there are no zeros of $\truncscaledpoly$
inside the disk $\left\{ \left|z\right|\le\left(1-\delta\right)\frac{r}{L}\right\} $
(for a small $\delta>0$, depending on $r$). The factor $1-\delta$
appears as a side-effect of the truncation of $\gef$. The factor
$L$ is the result of the scaling of the zeros. Choosing the parameter
$L$ slightly smaller than $r$, we see that $\mu_{\vec z}^{t}\left(D\right)=0$
for $t>0$ sufficiently small ($D$ is the unit disk). The upper bound
(\ref{eq:idea_upp_bnd_prob}) suggests that in order to bound the
probability of the hole event we should find the minimizer of $\funccontparamname{\alpha}$
over the set 
\[
\calH=\setd{\nu\in\probmeas}{\nu\left(D\right)=0}.
\]
Since the functional $\funccontparamname{\alpha}$ is lower semi-continuous
and strictly convex, this minimizer exists and is unique. If the value
of this minimizer happens to agree with the lower bound for the hole
probability, we can consider this minimizing measure to be (an appropriately
scaled limit of) the empirical measure of the most likely configuration
of zeros that gives rise to the hole event $H_{r}$.

To be somewhat more precise, we set $t=r^{-C_{2}}$, for some constant
$C_{2}\ge4$. Let us first obtain an upper bound for the probability
of the hole event. As suggested by the discussion above, we want to
consider the configurations in
\[
Z=\setd{\vec z\in\bbc^{N}}{\mu_{\vec z}^{t}\left(D\right)=0}.
\]
Since $\setd{\mu_{\vec z}^{t}}{\vec z\in Z}\subset\calH$, using (\ref{eq:idea_upp_bnd_prob})
we obtain the bound
\[
\pr{H_{r}\cap\regevent}\le\exp\left(-N^{2}\left[\inf_{\vec z\in\calH}\funccontparam{\alpha}{\mu_{\vec z}^{t}}-\funccontparam{\alpha}{\mu_{\alpha}}+o\left(1\right)\right]\right).
\]
In Section \ref{sec:sol_of_energy_prob} we find that the minimizer
of $\funccontparam{\alpha}{\nu}$ over the set $\calH$ is given by
\[
\dd\minmeassup 0{\alpha}\left(z\right)=\frac{e}{\alpha}\dd\subslebmeas{\left|z\right|=1}+\frac{1}{\alpha}\ind{\left\{ \sqrt{e}\le\left|z\right|\le\sqrt{\alpha}\right\} }z\cdot\frac{\dd\lebmeas\left(z\right)}{\pi}.
\]
A short calculation shows that
\begin{equation}
\exp\left(-N^{2}\left[\inf_{\nu\in\calH}\funccontparam{\alpha}{\nu}-\funccontparam{\alpha}{\mu_{\alpha}}\right]\right)=\exp\left(-N^{2}\left[\funccontparam{\alpha}{\minmeassup 0{\alpha}}-\funccontparam{\alpha}{\mu_{\alpha}}\right]\right)=\exp\left(-\frac{e^{2}}{4}r^{4}\right).\label{eq:min_val_of_const_prob}
\end{equation}
Since $\pr{\kregevent c}\le\exp\left(-Ar^{4}\right)$, for some large
constant $A>0$, we obtain the required bound for $\pr{H_{r}}$. See
Section \ref{sec:upp_bnd_for_larg_flucs} for a proof (for general
$p\ge0$, $p\ne1$).

\subsection{Large deviations for linear statistics}

Now we would like to consider configurations where the empirical measure
is `far' from the minimizing measure $\minmeassup 0{\alpha}$. Recall
$N=\alpha r^{2}$ and that $\alpha$ is a large parameter. It follows
that $\minmeasglob 0\left(\disc 0{\sqrt{\alpha}}\right)=\alpha$,
and thus
\begin{equation}
r^{2}\int_{\bbc}\varphi\left(w\right)\,\dd\minmeasglob 0\left(w\right)=N\cdot\int_{\bbc}\varphi\left(w\right)\,\dd\minmeassup 0{\alpha}\left(w\right).\label{eq:eq_for_Radon_and_trunc_meas}
\end{equation}
In addition, on the event $\regevent$, we have
\begin{eqnarray}
\int_{\bbc}\varphi\left(w\right)\,\dd\mu_{\vec z}^{t}\left(w\right) & = & N\cdot n_{\gef}\left(\varphi;r\right)\left(1+o\left(1\right)\right).\label{eq:idea_mu_z_t_lin_stats}
\end{eqnarray}
Consider now the event
\begin{eqnarray*}
L\left(0,\varphi,\lambda;r\right) & = & \left\{ \left|\linstats{\gef}{\varphi}r-r^{2}\int_{\bbc}\varphi\left(w\right)\,\dd\minmeasglob 0\left(w\right)\right|\ge\lambda\right\} .
\end{eqnarray*}
By (\ref{eq:eq_for_Radon_and_trunc_meas}) and (\ref{eq:idea_mu_z_t_lin_stats}),
on the event $L\left(0,\varphi,\lambda;r\right)\cap H_{r}\cap\regevent$
we have $\mu_{\vec z}^{t}\left(D\right)=0$, and
\begin{align}
\left|\int_{\bbc}\varphi\left(w\right)\,\dd\mu_{\vec z}^{t}\left(w\right)-\int_{\bbc}\varphi\left(w\right)\,\dd\minmeassup 0{\alpha}\left(w\right)\right| & \ge\frac{\lambda}{N}-\mbox{error terms}.\label{eq:idea_mu_z_t_prop_2}
\end{align}
Therefore, we now consider the configurations in
\[
Z^{\prime}=\setd{\vec z\in\bbc^{N}}{\mu_{\vec z}^{t}\left(D\right)=0\mbox{ and }\left|\int_{\bbc}\varphi\left(w\right)\,\dd\mu_{\vec z}^{t}\left(w\right)-\int_{\bbc}\varphi\left(w\right)\,\dd\minmeassup 0{\alpha}\left(w\right)\right|\ge\frac{\lambda}{N}\left(1+o\left(1\right)\right)}.
\]
In order to obtain a large deviation bound for the linear statistics,
we need some estimates for the convexity of the functional $\funccontparamname{\alpha}$.
By Claim \ref{claim:low_bnd_func_for_far_from_min_meas} we have for
any measure $\nu\in\calH$ which satisfies 
\[
\left|\int_{\bbc}\varphi\left(w\right)\,\dd\nu\left(w\right)-\int_{\bbc}\varphi\left(w\right)\,\dd\minmeassup 0{\alpha}\left(w\right)\right|\ge x,
\]
the following bound
\begin{equation}
\funccontparam{\alpha}{\nu}\ge\funccontparam{\alpha}{\minmeassup 0{\alpha}}+\frac{2\pi}{\Dnorm{\varphi}}\cdot x^{2}.\label{eq:idea_low_bnd_func_meas_far_from_opt}
\end{equation}
This implies that for $\vec z\in Z^{\prime}$ we have
\[
\funccontparam{\alpha}{\mu_{\vec z}^{t}}\ge\funccontparam{\alpha}{\minmeassup 0{\alpha}}+\frac{2\pi}{\Dnorm{\varphi}}\cdot\left(\frac{\lambda}{N}\right)^{2}-\mbox{error terms}.
\]
Finally, the bound (\ref{eq:idea_upp_bnd_prob}) (together with the
lower bound for $\pr{H_{r}}$, see Section \ref{sec:low_bnd_for_large_flucs})
gives
\begin{eqnarray*}
\pr{L\left(0,\varphi,\lambda;r\right)\cap H_{r}\cap\regevent} & \le & \pr{Z^{\prime}\cap\regevent}\le\exp\left(-N^{2}\left[\inf_{\vec z\in Z^{\prime}}\funccontparam{\alpha}{\mu_{\vec z}^{t}}-\funccontparam{\alpha}{\mu_{\alpha}}+o\left(1\right)\right]\right)\\
 & \le & \exp\left(-N^{2}\left[\frac{C}{\Dnorm{\varphi}}\cdot\left(\frac{\lambda}{N}\right)^{2}+\funccontparam{\alpha}{\minmeassup 0{\alpha}}-\funccontparam{\alpha}{\mu_{\alpha}}\right]+\mbox{error terms}\right)\\
 & \le & \pr{H_{r}}\exp\left(-\frac{C}{\Dnorm{\varphi}}\cdot\lambda^{2}+\mbox{error terms}\right),
\end{eqnarray*}
hence we obtain Theorem \ref{thm:devs_in_emp_meas_hole_case}. See
Subsection \ref{subsec:large_fluct_in_lin_stats} for the details.

\section{Preliminaries}

Throughout the paper we will use the following standard bounds for
the factorial
\begin{equation}
\left(\frac{k}{e}\right)^{k}\le k!\le3\sqrt{k}\left(\frac{k}{e}\right)^{k},\quad k\ge1.\label{eq:bounds_for_factorial}
\end{equation}
The lower bound follows immediately from the series expansion for
$e^{x}$, and the upper bound by induction and using the inequality
$\left(1+\frac{1}{k}\right)^{k+\frac{1}{2}}>e$. Recall that for a
standard Gaussian random variable $\xi$, we have $\left|\xi\right|^{2}\sim\exp\left(1\right)$,
i.e. $\pr{\left|\xi\right|\ge\lambda}=\exp\left(-\lambda^{2}\right)$.
In what follows, we frequently use the estimate
\[
\pr{\left|\xi\right|\le\lambda}\in\left[\frac{\lambda^{2}}{2},\lambda^{2}\right],\quad\lambda<1.
\]

\subsection{Estimates for the GEF}

When studying the distribution of the zeros of the GEF, it is usually
easier to work with a truncation of its Taylor series. We use simple
estimates for the `tail' of the series, to control the error that
is introduced by the truncation.

Let $\left\{ \xi_{k}\right\} _{k=0}^{\infty}$ be a sequence of i.i.d.
standard complex Gaussians. We need the following simple estimates.
\begin{lem}
Let $r>2$. We have
\[
\pr{\forall k\,\left|\xi_{k}\right|\le\sqrt{r^{6}+k}}\ge1-Ce^{-r^{6}}.
\]
\end{lem}
\begin{proof}
For all $k\ge0$, we have
\[
\pr{\left|\xi_{k}\right|\le\sqrt{r^{6}+k}}=1-\exp\left(-\left(r^{6}+k\right)\right).
\]
Therefore, using the independence of the $\xi_{k}$s
\begin{eqnarray*}
\pr{\forall k\,\left|\xi_{k}\right|\le\sqrt{r^{6}+k}} & = & \prod_{k=0}^{\infty}\left(1-\exp\left(-\left(r^{6}+k\right)\right)\right)=\exp\left(\sum_{k=0}^{\infty}\log\left[1-\exp\left(-\left(r^{6}+k\right)\right)\right]\right)\\
 & \ge & \exp\left(-2\cdot\sum_{k=0}^{\infty}\exp\left(-\left(r^{6}+k\right)\right)\right)\ge\exp\left(-4\cdot e^{-r^{6}}\right)\ge1-4\cdot e^{-r^{6}}.
\end{eqnarray*}
\end{proof}
\begin{lem}
\label{lem:up_bnd_sum_mod_xi_sq}Let $x\ge10$. For $N\in\bbn$,
\[
\pr{\sum_{k=0}^{N}\left|\xi_{k}\right|^{2}>x\left(N+1\right)}\le\exp\left(-\frac{Nx}{2}\right).
\]
\end{lem}
\begin{proof}
Let $t\in\left(0,1\right)$. Using Markov's inequality and the independence
of the $\xi_{k}$s, we get
\begin{eqnarray*}
\pr{\sum_{k=0}^{N}\left|\xi_{k}\right|^{2}>x\left(N+1\right)} & = & \pr{\exp\left(t\cdot\sum_{k=0}^{N}\left|\xi_{k}\right|^{2}\right)>e^{tx\left(N+1\right)}}\le e^{-tx\left(N+1\right)}\ex{\exp\left(t\cdot\sum_{k=0}^{N}\left|\xi_{k}\right|^{2}\right)}\\
 & = & e^{-tx\left(N+1\right)}\left(\ex{e^{t\left|\xi\right|^{2}}}\right)^{N+1},
\end{eqnarray*}
where $\xi$ is a standard complex Gaussian. Using the fact $\ex{e^{t\left|\xi\right|^{2}}}=\frac{1}{1-t}$,
and taking $t=\frac{x-1}{x}$, we then have
\[
\pr{\sum_{k=0}^{N}\left|\xi_{k}\right|^{2}>x\left(N+1\right)}\le\left(xe^{1-x}\right)^{N+1}\le\exp\left(-\frac{1}{2}Nx\right),
\]
since $x\ge10$.
\end{proof}
For $N\in\bbn$ define the tail of the GEF to be the series
\[
\seriestail\left(z\right)=\sum_{k=N+1}^{\infty}\xi_{k}\frac{z^{k}}{\sqrt{k!}},\quad z\in\bbc.
\]

\begin{lem}
\label{lem:tail_bound}Let $r>0$ be sufficiently large, $\lambda>4$,
and $B\in\left[0,\frac{\sqrt{\lambda}}{2}\right]$. Outside an exceptional
event of probability at most $\exp\left(-Cr^{6}\right)$, we have
for any $N\in\bbn$, such that $N\ge\lambda r^{2}$, 
\[
\left|\seriestail\left(z\right)\right|\le\exp\left(\frac{N}{2}\log\left(\frac{4B^{2}}{\lambda}\right)\right),\quad\left|z\right|\le Br.
\]
\end{lem}
\begin{proof}
By the previous lemma, after discarding an exceptional event, we may
assume $\left|\xi_{k}\right|\le\sqrt{r^{6}+k}$ for all $k\in\bbn$.
Let us write
\[
d_{k}=\sqrt{r^{6}+k}\cdot\frac{r^{k}}{\sqrt{k!}},\quad k\in\bbn.
\]
Then, for $k\ge N$,
\[
\frac{d_{k+1}}{d_{k}}=\sqrt{1+\frac{1}{r^{6}+k}}\cdot\frac{r}{\sqrt{k+1}}\le\sqrt{\frac{2}{\lambda}}.
\]
Therefore, for $\left|z\right|\le Br$, we have
\[
\left|\seriestail\left(z\right)\right|\le\sum_{k=N+1}^{\infty}\left|\xi_{k}\right|\frac{B^{k}r^{k}}{\sqrt{k!}}=\sum_{k=N+1}^{\infty}B^{k}\cdot d_{k}\le B^{N+1}d_{N+1}\sum_{k=0}^{\infty}\left(B\sqrt{\frac{2}{\lambda}}\right)^{k}\le C\cdot B^{N+1}\cdot d_{N+1}.
\]
Using (\ref{eq:bounds_for_factorial}),
\[
d_{N+1}=\sqrt{r^{6}+N+1}\cdot\frac{r^{N+1}}{\sqrt{\left(N+1\right)!}}\le C\sqrt{r^{6}+N+1}\left(\frac{er^{2}}{N+1}\right)^{\frac{N+1}{2}}\le CN^{2}\left(\frac{3}{\lambda}\right)^{\frac{N+1}{2}}\le\left(\frac{7}{2\lambda}\right)^{\frac{N}{2}},
\]
since $CN^{2}\le\left(\frac{7}{6}\right)^{\frac{N}{2}}$, for $N$
sufficiently large. Similarly $CB\le N\le\left(\frac{8}{7}\right)^{N/2}$
for $N$ sufficiently large, hence the required estimate is obtained.
\end{proof}
For the upper bound estimate we need an `a priori' bound for the number
of zeros of the GEF inside a disk. The following estimate follows
immediately from \cite[Theorem 3]{krishnapur2006overcrowding}.
\begin{cor}
\label{cor:n(r)_apriori_bound}Let $r>0$ be large enough. We have
\[
\pr{n_{F}\left(r\right)\ge r^{3}}\le\exp\left(-r^{6}\right).
\]
\end{cor}
Let $M_{\gef}\left(\rho\right)=\max\setd{\left|\gef\left(z\right)\right|}{\left|z\right|\le\rho}$.
We want to avoid the event where the GEF is very small (in absolute
value) inside a large disk. We use the following simple estimate (cf.
with the more accurate \cite[Lemma 7]{nishry2010asymptotics}).
\begin{lem}
\label{lem:max_gef_low_bnd} Let $x>0$. For $\rho>1$ we have
\[
\pr{M_{\gef}\left(\rho\right)\le\exp\left(-x\right)}\le\exp\left(-2x\rho^{2}\right).
\]
\end{lem}
\begin{proof}
Assuming $M_{\gef}\left(r\right)\le\exp\left(-x\right)<1$ and using
Cauchy's estimate for the coefficients of $\gef$ we find that
\[
\left|\xi_{k}\right|\frac{\rho^{k}}{\sqrt{k!}}\le M_{\gef}\left(\rho\right)\le\exp\left(-x\right),\qquad\forall k\in\bbn.
\]
Notice that the sequence $\frac{\rho^{k}}{\sqrt{k!}}$ is increasing
from $k=0$ to $k=\left\lfloor \rho^{2}\right\rfloor $. Therefore,
we get
\[
\left|\xi_{k}\right|\le\exp\left(-x\right),\quad k\in\left\{ 0,\dots,\left\lfloor \rho^{2}\right\rfloor \right\} ,
\]
and the probability of this event is at most $\exp\left(-2x\left(\left\lfloor \rho^{2}\right\rfloor +1\right)\right)\le\exp\left(-2x\rho^{2}\right)$.
\end{proof}
We also want to control the probability the GEF is too large inside
a large disk (this is a very rare event).
\begin{lem}[{See \cite[Lemma 1]{sodin2005random}}]
\label{lem:max_gef_upp_bnd}  For $\rho>0$ large enough, we have
\[
\pr{M_{\gef}\left(\rho\right)\ge\exp\left(\rho^{2}\right)}\le\exp\left(-\exp\left(\rho^{2}\right)\right).
\]
\end{lem}

\subsection{Perturbation of Zeros of Analytic Functions}

Let $f$ be an entire function and denote by $w_{1},\dots,w_{m}$
the zeros of $f$ in $\disc 0r$ (including multiplicities). For $0<\gamma\le\tfrac{1}{4}$,
set
\[
C_{\gamma}\left(r\right)=\bigcup_{k=1}^{m}\disc{w_{k}}{\gamma},\quad E_{\gamma}\left(r\right)=\disc 0r\backslash C_{\gamma},\mbox{ and}\quad m_{f}\left(r;\gamma\right)=\min_{z\in E_{\gamma}}\left|f\left(z\right)\right|.
\]
The following theorem is a restatement of a theorem of Rosenbloom
(\cite[Theorem 4]{rosenbloom1969perturbation}) for the unit disk.
It gives an effective lower bound for the modulus of an analytic function,
outside a neighborhood of its zeros. 
\begin{thm}
\label{thm:Rosenbloom}Let $f$ be an entire function, and $\gamma\in\left(0,\tfrac{1}{4}\right]$.
Suppose that $\left|f\left(z_{0}\right)\right|\ge1$ for some $z_{0}\in\bbc$
with $\left|z_{0}\right|=\rho>0$. Let $0<r\le\frac{\rho}{2}$, and
suppose that $E_{\gamma}\left(r\right)\ne\emptyset$, then
\[
m_{f}\left(r;\gamma\right)\ge\exp\left(-C\log M_{f}\left(3\rho\right)\log\frac{1}{\gamma}\right).
\]
\end{thm}
We can use the previous theorem to control the perturbation of the
zeros of analytic functions, when we add an `error term' of small
modulus.
\begin{lem}
\label{lem:zero_perturb}Let $f,g$ be entire functions, and $B,\rho\ge1$.
Suppose that $f$ has at most $M>0$ zeros in the disk $\disc 0{2B\rho}$
and let $0<\gamma<\frac{\rho}{2M}$. In addition, assume $M_{g}\left(2B\rho\right)<m_{f}\left(2B\rho;\gamma\right)$.
Then,
\[
n_{f+g}\left(\rho^{\prime}-2M\gamma\right)\le n_{f}\left(\rho^{\prime}\right)\le n_{f+g}\left(\rho^{\prime}+2M\gamma\right),\quad\forall\rho^{\prime}\in\left(2M\gamma,2B\rho-2M\gamma\right).
\]
Furthermore, if $\varphi$ is a test function supported on $\disc 0B$,
with modulus of continuity $\omega\left(\varphi;t\right)$, then
\[
\left|n_{f}\left(\varphi;\rho\right)-n_{f+g}\left(\varphi;\rho\right)\right|\le M\cdot\omega\left(\varphi;\frac{2M\gamma}{\rho}\right).
\]
\end{lem}
\begin{proof}
Let $C_{\gamma}=C_{\gamma}\left(2B\rho\right)$. We can write $C_{\gamma}=\bigcup C_{j}$
where $C_{j}$ are the connected components of $C_{\gamma}$ (tangent
disks are not connected). Notice that the diameter of each component
is at most $2\gamma\cdot M$, and that by Rouché's theorem the number
of zeros of $f$ and $f+g$ is the same in each component. Therefore,
we find that
\begin{eqnarray*}
n_{f}\left(\rho^{\prime}\right) & \le & \#\left\{ \mbox{zeros of \ensuremath{f} in components intersecting \ensuremath{\left|z\right|\le\rho^{\prime}}}\right\} \\
 & = & \#\left\{ \mbox{zeros of \ensuremath{f+g} in components intersecting \ensuremath{\left|z\right|\le\rho^{\prime}}}\right\} \\
 & \le & n_{f+g}\left(\rho^{\prime}+2M\gamma\right).
\end{eqnarray*}
The lower bound is obtained in the same way. Similarly, if $w,w^{\prime}\in C_{j}$
are two points in the same component $C_{j}$, then
\[
\left|\varphi\left(\frac{w}{\rho}\right)-\varphi\left(\frac{w^{\prime}}{\rho}\right)\right|\le\omega\left(\varphi;\frac{2M\gamma}{\rho}\right).
\]
We conclude that
\[
\left|n_{f}\left(\varphi;\rho\right)-n_{f+g}\left(\varphi;\rho\right)\right|=\left|\sum_{z\in\zeroset f}\varphi\left(\frac{z}{\rho}\right)-\sum_{z^{\prime}\in\zeroset{f+g}}\varphi\left(\frac{z^{\prime}}{\rho}\right)\right|\le M\cdot\omega\left(\varphi;\frac{2M\gamma}{\rho}\right).
\]
\end{proof}
\begin{rem}
If $f$ has no zeros in the disk $\disc 0{2B\rho}$, then, under the
assumptions of the lemma, $f+g$ also has no zeros there. Thus, the
results of the lemma follow also in this case.
\end{rem}

\subsection{\label{subsec:GEF_trunc}Truncation of the GEF}

We now explain how to truncate the power series $\gef$, such that
we can obtain a polynomial whose zeros (inside a disk $\disc 0{Cr}$)
are very close to the zeros of $\gef$. This introduces some technical
complications. Let $N\in\bbn$. We would like to split the GEF in
the following way,
\[
\gef\left(z\right)=\sum_{k=0}^{N}\xi_{k}\frac{z^{k}}{\sqrt{k!}}+\sum_{k=N+1}^{\infty}\xi_{k}\frac{z^{k}}{\sqrt{k!}}\defeq\truncpoly\left(z\right)+\seriestail\left(z\right),\quad z\in\bbc.
\]
We now define a `regular' event, on which the GEF has desirable properties.
The complement of this event is negligible. An additional technical
issue is the fact we need control over the leading coefficient of
$P_{N}$, that is over $\left|\xi_{N}\right|$ (we want to keep it
not too small). This means that in order to make the exceptional set
small, we have to pick a random value for $N$. Recall that $n_{\gef}\left(r\right)$
is the number of zeros of $\gef$ in the disk $\left\{ \left|z\right|\le r\right\} $,
and that
\[
M_{\gef}\left(r\right)=\sup_{\left|z\right|\le r}\left|\gef\left(z\right)\right|.
\]

\begin{lem}
\label{lem:reg_event}Let $\rho>0$ be sufficiently large, $A\ge1$,
$\lambda>16$, and $B\in\left[1,\frac{\sqrt{\lambda}}{2}\right]$.
There exist an event $\regevent$ with the following properties:

\begin{enumerate}
\item $\pr{\kregevent c}\le\exp\left(-C\cdot AB^{4}\rho^{4}\right)+\exp\left(-C\lambda\rho^{4}\right)$.
\item \label{enu:GEF_reg_props}On the event $\regevent$ we have:

\begin{enumerate}
\item For any $N\in\bbn$ with $N\ge\lambda\rho^{2}$ we have $\left|\seriestail\left(z\right)\right|\le\exp\left(\frac{N}{2}\log\left(\frac{16B^{2}}{\lambda}\right)\right),\quad\forall\left|z\right|\le2B\rho.$
\item \label{enu:prop_zeros_upp_bnd}$n_{\gef}\left(2B\rho\right)\le\left(2B\rho\right)^{3}.$
\item \label{enu:prop_max_upp_bnd}$M_{\gef}\left(6B\rho\right)\le\exp\left(36\cdot B^{2}\rho^{2}\right)$.
\item \label{enu:prop_max_low_bnd}$M_{\gef}\left(4B\rho\right)\ge\exp\left(-AB^{2}\rho^{2}\right)$.
\end{enumerate}
\item Let $N_{0}=\left\lfloor \lambda\rho^{2}\right\rfloor +1$, $N_{1}=\left\lfloor 2\lambda\rho^{2}\right\rfloor +1$.
We have $\regevent=\biguplus_{N=N_{0}}^{N_{1}}\kregevent N$, where
on the event $\kregevent N$ we have $\left|\xi_{N}\right|\ge\exp\left(-\rho^{2}\right)$.
\item For any $N\in\left\{ N_{0},\dots,N_{1}\right\} $ we have $\sum_{k=0}^{N}\left|\xi_{k}\right|^{2}\le C\lambda\rho^{4}$.
\end{enumerate}
\end{lem}
\begin{proof}
The properties in (\ref{enu:GEF_reg_props}) follow by combining the
statements of Lemma \ref{lem:tail_bound}, Corollary \ref{cor:n(r)_apriori_bound},
Lemma \ref{lem:max_gef_low_bnd}, and Lemma \ref{lem:max_gef_upp_bnd}.
The probability of the exceptional event in Lemma \ref{lem:max_gef_low_bnd}
is the largest one.

Since the $\xi_{k}$ are independent, the probability that $\left|\xi_{k}\right|<\exp\left(-\rho^{2}\right)$
for all $k\in\left\{ N_{0},\dots,N_{1}\right\} $ is at most $\exp\left(-C\left(N_{1}-N_{0}+1\right)\rho^{2}\right)\le\exp\left(-C\lambda\rho^{4}\right)$.
To make the events $\kregevent N$ disjoint (this is not essential
for our estimates), we can choose $N$ to be smallest value of $k\in\left\{ N_{0},\dots,N_{1}\right\} $
such that $\left|\xi_{k}\right|\ge\exp\left(-\rho^{2}\right)$.

Finally, by Lemma \ref{lem:up_bnd_sum_mod_xi_sq}, we have
\[
\pr{\sum_{k=0}^{N_{1}}\left|\xi_{k}\right|>3\lambda\rho^{4}}\le\pr{\sum_{k=0}^{N_{1}}\left|\xi_{k}\right|>\rho^{2}\left(N_{1}+1\right)}\le\exp\left(-\frac{N_{1}}{2}\cdot\rho^{2}\right)\le\exp\left(-C\lambda\rho^{4}\right).
\]
\end{proof}
We would now like to apply Lemma \ref{lem:zero_perturb} to show that
we can approximate the zeros of the GEF.
\begin{lem}
\label{lem:reg_event_conclusion}Let $\rho,A,\lambda,B$ be as above.
In addition, let $\gamma>0$ and $M_{0}=8B^{3}\rho^{3}$. Suppose
$\gamma\le\frac{\rho}{2M_{0}}$, $\lambda=o\left(\rho\right)$ and
that $B^{2}A\log\frac{1}{\gamma}=o\left(\lambda\log\left(\frac{\lambda}{16B^{2}}\right)\right)$
is satisfied. Then, on the event $\regevent$, we have 
\[
m_{\gef}\left(2B\rho;\gamma\right)>M_{\seriestail}\left(2B\rho\right),\quad\forall N\in\left\{ N_{0},\dots,N_{1}\right\} ,
\]
and on $\kregevent N$ we also have
\[
\frac{1}{\left|\xi_{N}\right|}\sum_{k=0}^{N}\left|\xi_{k}\right|^{2}\le\exp\left(C\rho^{2}\right).
\]
In particular, this implies, 
\[
n_{\truncpoly}\left(\rho-2M_{0}\gamma\right)\le n_{\gef}\left(\rho\right)\le n_{\truncpoly}\left(\rho+2M_{0}\gamma\right).
\]
Furthermore, if $\varphi$ is a test function supported on $\disc 0B$,
with modulus of continuity $\omega\left(\varphi;t\right)$, then
\[
\left|n_{\gef}\left(\varphi;\rho\right)-n_{\truncpoly}\left(\varphi;\rho\right)\right|\le CM_{0}\cdot\omega\left(\varphi;\frac{2M_{0}\gamma}{\rho}\right).
\]
\end{lem}
\begin{proof}
The bound for $\frac{1}{\left|\xi_{N}\right|}\sum_{k=0}^{N}\left|\xi_{k}\right|^{2}$
follows immediately from the previous lemma. Let $N\in\left\{ N_{0},\dots,N_{1}\right\} $.
On $\regevent$ we have the bound
\[
\left|\seriestail\left(z\right)\right|\le\exp\left(\frac{N}{2}\log\left(\frac{16B^{2}}{\lambda}\right)\right),\quad\forall\left|z\right|\le2B\rho.
\]
According to Property (\ref{enu:prop_max_low_bnd}) in the previous
lemma, and the maximum modulus principle, there exists a point $z_{0}$,
with $\left|z_{0}\right|=4B\rho$ and $\left|\gef\left(z_{0}\right)\right|\ge\exp\left(-AB^{2}\rho^{2}\right)$.
We now set
\[
\widetilde{F}\left(z\right)=\frac{\gef\left(z\right)}{\gef\left(z_{0}\right)},\quad\widetilde{M}\left(t\right)=\sup\setd{\left|\widetilde{F}\left(z\right)\right|}{\left|z\right|\le t}.
\]
Thus, using Property (\ref{enu:prop_max_upp_bnd}), we find that $\log\widetilde{M}\left(6B\rho\right)\le CB^{2}\rho^{2}+AB^{2}\rho^{2}$.
Applying Theorem \ref{thm:Rosenbloom} to the function $\widetilde{F}$,
we find that
\[
m_{\gef}\left(2B\rho;\gamma\right)\ge\exp\left(-\left(CB^{2}\rho^{2}+AB^{2}\rho^{2}\right)\log\frac{1}{\gamma}\right)\ge\exp\left(-CB^{2}\rho^{2}\cdot A\log\frac{1}{\gamma}\right),
\]
where we used the fact that $\left|\gef\left(z\right)\right|\ge\left|\widetilde{F}\left(z\right)\right|\exp\left(-AB^{2}\rho^{2}\right)$.
In order to obtain $m_{\gef}\left(2B\rho;\gamma\right)>\left|\seriestail\left(z\right)\right|$
for all $z\in\disc 0{2B\rho}$, we should have
\[
\frac{N}{2}\log\left(\frac{16B^{2}}{\lambda}\right)<-CB^{2}\rho^{2}\cdot A\log\frac{1}{\gamma},
\]
which is satisfied by our requirements on $\lambda$ and $\gamma$
(recall $N$ is of order $\lambda\rho^{2}$). Finally, by Lemma \ref{lem:reg_event},
Property \ref{enu:prop_zeros_upp_bnd}, we have that $M\defeq n\left(2B\rho\right)\le8B^{3}\rho^{3}=M_{0}$
(w.l.o.g. we may assume that $M>0$, see the remark after Lemma \ref{lem:zero_perturb}).
Lemma \ref{lem:zero_perturb}, applied to the functions $\gef$ and
$-\seriestail$, then implies,
\[
n_{\truncpoly}\left(\rho-2M\gamma\right)\le n_{\gef}\left(\rho\right)\le n_{\truncpoly}\left(\rho+2M\gamma\right),
\]
and
\[
\left|n_{\gef}\left(\varphi;\rho\right)-n_{\truncpoly}\left(\varphi;\rho\right)\right|\le CM\cdot\omega\left(\varphi;\frac{2M\gamma}{\rho}\right).
\]
\end{proof}
\begin{rem}
When applying the previous lemma, the parameters $A,B$ will be arbitrary,
but fixed (not depending on $\rho$). In addition, $\lambda=\log\rho$,
and $\gamma=\rho^{-C}$, with some constant $C\ge4$.
\end{rem}

\subsection{Logarithmic potential and linear statistics}

The following result is known as Jensen's formula.
\begin{thm}
\label{thm:Jensen_for}Let $\nu\in\probmeas$ and $r>0$. Then, 
\[
\logpot{\nu}0+\int_{0}^{r}\frac{\nu\left(\overline{\disc 0t}\right)}{t}\,\dd t=\frac{1}{2\pi}\int_{0}^{2\pi}\logpot{\nu}{re^{i\theta}}\,\dd\theta.
\]
In case $\nu$ is a radial measure, we have
\[
\int_{0}^{r}\frac{\nu\left(\overline{\disc 0t}\right)}{t}\,\dd t=\logpot{\nu}r-\logpot{\nu}0.
\]
\end{thm}
\begin{proof}
This follows from the Poisson-Jensen formula in the disk $\disc 0r$
(\cite[Theorem 4.10]{saff2013logarithmic}), applied to the subharmonic
function $\logpot{\nu}z$, and integration by parts.
\end{proof}
For a function $\phi:\bbc\mapsto\bbr$, with $\phi_{x},\phi_{y}\in L^{2}\left(\bbc\right)$,
we recall 
\[
\Dnorm{\phi}=\int_{\bbc}\left(\phi_{x}^{2}+\phi_{y}^{2}\right)\,\dd m\left(z\right)=\int_{\bbc}\left|\nabla\phi\left(z\right)\right|^{2}\,\dd m\left(z\right).
\]
Let $\nu,\mu\in\probmeas$ be probability measures with compact support
and finite logarithmic energy, and let $\sigma=\nu-\mu$ be a signed
measure (with $\sigma\left(\bbc\right)=0$). It is known (\cite[Theorem 1.20]{landkof1972foundations},
see also \cite[Proof of Lemma I.1.8]{saff2013logarithmic}) that
\[
\Dnorm{\logpot{\sigma}w}=\int_{\bbc}\left|\nabla\logpot{\sigma}w\right|^{2}\,\dd m\left(w\right)=-2\pi\cdot\logenerg{\nu-\mu}<\infty,
\]
and in particular that $\left|\nabla\logpot{\sigma}w\right|\in L^{2}\left(\bbc\right)$
. We also mention that (as to be expected) $\logenerg{\nu-\mu}\le0$,
with equality if and only if $\nu=\mu$ (\cite[Theorem 1.16]{landkof1972foundations},
\cite[Lemma I.1.8]{saff2013logarithmic}).

The following result allows us to get a lower bound for the `distance'
between two measures, in terms of linear statistics (cf. \cite[Eq. (3.10)]{pritsker2011equidistribution}).
\begin{lem}
\label{lem:lin_stats_for_meas}Suppose $\varphi\in C_{0}^{2}\left(\bbc\right)$
is a compactly supported test function, which is twice continuously
differentiable. Let $\nu,\mu\in\probmeas$ be probability measures
with compact support and finite logarithmic energy. Then
\[
\left|\int_{\bbc}\varphi\left(w\right)\,\dd\nu\left(w\right)-\int_{\bbc}\varphi\left(w\right)\,\dd\mu\left(w\right)\right|\le\frac{1}{\sqrt{2\pi}}\sqrt{\Dnorm{\varphi}}\sqrt{-\logenerg{\nu-\mu}}.
\]
\end{lem}
\begin{proof}
Let us write $\sigma=\nu-\mu$, and recall that $\dd\sigma\left(z\right)=\frac{1}{2\pi}\Delta\logpot{\sigma}z\,\dd m\left(z\right)$
in the sense of distributions. Integrating by parts, and using the
Cauchy\textendash Schwarz inequality, we get
\begin{eqnarray*}
\left|\int_{\bbc}\varphi\left(w\right)\,\dd\sigma\left(w\right)\right| & = & \frac{1}{2\pi}\left|\int_{\bbc}\Delta\varphi\left(w\right)\logpot{\sigma}w\,\dd m\left(w\right)\right|=\frac{1}{2\pi}\left|\int_{\bbc}\nabla\varphi\left(w\right)\cdot\nabla U_{\sigma}\,\dd m\left(w\right)\right|\\
 & \le & \frac{1}{2\pi}\sqrt{\int_{\bbc}\left|\nabla\varphi\left(w\right)\right|^{2}\,\dd m\left(w\right)}\sqrt{\int_{\bbc}\left|\nabla\logpot{\sigma}w\right|^{2}\,\dd m\left(w\right)}\\
 & = & \frac{1}{2\pi}\sqrt{\Dnorm{\varphi}}\sqrt{\Dnorm{\logpot{\sigma}w}}=\frac{1}{\sqrt{2\pi}}\sqrt{\Dnorm{\varphi}}\sqrt{-\logenerg{\nu-\mu}}.
\end{eqnarray*}
\end{proof}

\section{\label{sec:upp_bnd_for_larg_flucs}Probability of large fluctuations
in the number of zeros - Upper bound}

Given $p\ge0$, $p\ne1$, we find in this section an asymptotic upper
bound for the probability of the event $\pr{n\left(r\right)=\left\lfloor pr^{2}\right\rfloor }$,
as $r\to\infty$, where $n\left(r\right)=n_{\gef}\left(r\right)$
is the number of zeros of the GEF inside the disk $\left\{ \left|z\right|\le r\right\} $.
Recall the GEF is given by the random Taylor series
\[
\gef\left(z\right)=\sum_{k=0}^{\infty}\xi_{k}\frac{z^{k}}{\sqrt{k!}},\quad z\in\bbc,
\]
where $\left\{ \xi_{k}\right\} $ is a sequence of independent standard
complex Gaussians. Almost surely all of the zeros of $\gef$ are simple,
so we can ignore multiplicities in this paper. The well-known Edelman-Kostlan
formula (\cite[Sec. 2.4]{hough2009zeros}) implies that the mean number
of zeros per unit area is $\frac{1}{\pi}$, and in particular $\ex{n\left(r\right)}=r^{2}.$
The asymptotic behavior of the variance $\var{n\left(r\right)}$ was
originally computed by Forrester and Honner in \cite{forrester1999exact}:

\[
\var{n\left(r\right)}=\kappa_{1}r+o\left(r\right),\quad r\to\infty,
\]
with an explicit constant $\kappa_{1}$. It was shown in the paper
\cite{nazarov2012correlation} that the normalized random variables
$\frac{n\left(r\right)-r^{2}}{\sqrt{\var{n\left(r\right)}}}$ converge
in distribution to a standard Gaussian random variable (asymptotic
normality).

We wish to find the precise logarithmic asymptotics of the probability
of the event where $n\left(r\right)$ is (very) far from its expected
value $r^{2}$. To formulate our theorem, we define a function $q\left(p\right):\left[0,\infty\right)\to\left[0,e\right]$
as follows:
\begin{enumerate}
\item In case $p\in\left(0,1\right)\cup\left(1,e\right)$, take $q\ne p$
to be the solution of $p\left(\log p-1\right)=q\left(\log q-1\right)$. 
\item In the remaining cases, put $q=e$ for $p=0$, $q=1$ for $p=1$,
and $q=0$ for $p\ge e$.
\end{enumerate}
Note that $q$ can be written explicitly in terms of $p$, using the
Lambert function \cite{corless1996lambertw}.
\begin{thm}
\label{thm:very_large_fluct}Fix $p\in\left[0,\infty\right)\backslash\left\{ 1\right\} $.
With $q=q\left(p\right)$ as above, we have as $r\to\infty$
\[
\pr{n\left(r\right)=\left\lfloor pr^{2}\right\rfloor }=\exp\left(-Z_{p}r^{4}+O\left(r^{2}\log^{2}r\right)\right),
\]
where

\[
Z_{p}=\left|\int_{p}^{q\left(p\right)}x\log x\,\dd x\right|.
\]
\end{thm}
\begin{rem}
A simple calculation shows
\[
Z_{p}=\begin{cases}
\left|\frac{1}{4}\left[q^{2}\left(2\log q-1\right)-p^{2}\left(2\log p-1\right)\right]\right| & \,p\in\left(0,e\right)\backslash\left\{ 1\right\} ;\\
\frac{1}{4}p^{2}\left(2\log p-1\right) & \,p\ge e.
\end{cases}
\]
The case $Z_{0}=\frac{e^{2}}{4}$, corresponding to the hole event
$\left\{ n\left(r\right)=0\right\} $, follows from the results in
\cite{nishry2012hole} (with a slightly better error term). See Section
\ref{sec:discussion} for a graph of this function.
\end{rem}
In this section we obtain the upper bound in Theorem \ref{thm:very_large_fluct}.
The proof is slightly more general than necessary for proving the
results of this section, as we are going to use it again in Section
\ref{sec:conv_emp_meas}. In Section \ref{sec:low_bnd_for_large_flucs},
we prove the lower bound in Theorem \ref{thm:very_large_fluct}.
\begin{rem}
\label{rem:choice_of_params}Many parameters appear in the course
of the proof. Ultimately they all depend on $p$ and $r$, that appear
in the statement of Theorem \ref{thm:very_large_fluct}. The parameter
$B\ge1$ is fixed (but can be arbitrarily large), the parameter $A$
depends only on $p$, for the other parameters we have
\[
\gamma=t=r^{-C_{2}},\,\lambda=\log r,\,L=r+O\left(\frac{1}{r}\right),
\]
where $C_{2}\ge4$ is fixed. In addition,
\[
N\in\left\{ \left\lfloor \lambda r^{2}\right\rfloor +1,\left\lfloor 2\lambda r^{2}\right\rfloor +1\right\} ,\,\alpha=Nr^{-2}\le3\log r.
\]
\end{rem}

\subsection{\label{subsec:upp_bnd_GEF_trunc}Truncation of the power series}

Let $B\ge1$ be a fixed constant, and suppose that $r>0$ is large.
Denote by $n\left(r\right)=n_{\gef}\left(r\right)$ the number of
zeros of the GEF inside the disk $\disc 0r$. We wish to approximate
$\gef$ by a polynomial, in such a way that the zeros of the polynomial
are close to the zeros of the GEF inside the larger disk $\disc 0{Br}$.

Suppose $\varphi$ is a continuous test function supported on the
disk $\disc 0B$, where $B\ge1$. Let $A\ge1$, $\lambda>16$, and
$\gamma=r^{-C_{2}}$, with $C_{2}\ge4$. In addition, put $N_{0}=\left\lfloor \lambda r^{2}\right\rfloor +1$,
$N_{1}=\left\lfloor 2\lambda r^{2}\right\rfloor +1$. We now wish
to apply Lemma \ref{lem:reg_event} and Lemma \ref{lem:reg_event_conclusion}
with $\rho=r$. We notice that if $A=O\left(1\right)$, $\lambda=\log r$,
then the conditions of both lemmas are satisfied. We find that there
exist events $\regevent$ and $\kregevent N$, $N\in\left\{ N_{0},\dots,N_{1}\right\} $,
such that
\[
\regevent=\biguplus_{N=N_{0}}^{N_{1}}\kregevent N,\quad\pr{\kregevent c}\le\exp\left(-C\cdot AB^{4}r^{4}\right).
\]
Put $M_{0}=8B^{3}r^{3}$, and $K_{0}=2M_{0}\gamma\le Cr^{3-C_{2}}=O\left(\frac{1}{r}\right)$.
If we write
\[
\truncpoly\left(z\right)=\sum_{k=0}^{N}\xi_{k}\frac{z^{k}}{\sqrt{k!}},\quad z\in\bbc,
\]
then, on the event $\kregevent N$, we have
\begin{eqnarray}
n_{P_{N}}\left(r-K_{0}\right) & \le & n\left(r\right)\le n_{P_{N}}\left(r+K_{0}\right),\quad K_{0}=16B^{3}r^{3}\gamma\le Cr^{3-C_{2}}=O\left(\frac{1}{r}\right),\label{eq:zero_count_poly}\\
\left|n_{\gef}\left(\varphi;r\right)-n_{\truncpoly}\left(\varphi;r\right)\right| & \le & CM_{0}\cdot\omega\left(\varphi;K_{0}r^{-1}\right),\label{eq:lin_stats_poly}
\end{eqnarray}
and 
\begin{equation}
\frac{1}{\left|\xi_{N}\right|}\sum_{k=0}^{N}\left|\xi_{k}\right|^{2}\le\exp\left(Cr^{2}\right).\label{eq:reg_event_bound_for_xi}
\end{equation}

\subsubsection{Choice of the parameter $A$.}

Let $n\left(r\right)$ be the number of zeros of the GEF $\gef\left(z\right)$
in the disk $\left\{ \left|z\right|\le r\right\} $. We assume that
one of the following two events occurs:
\begin{itemize}
\item[Case 1. ]  $n\left(r\right)\le pr^{2}$, where $p\in\left[0,1\right)$.
\item[Case 2. ]  ${\rm I}$: $n\left(r\right)\ge pr^{2}$, where $p\in\left(1,e\right)$.
${\rm II}$: $n\left(r\right)\ge pr^{2}$, where $p\in\left[e,\infty\right)$.
\end{itemize}
We always assume that $r$ is sufficiently large for all the different
asymptotic estimates that we use (this might depend on $p$ in Case
2.II.). We remark that, if $C>0$ is a sufficiently large numerical
constant, then events with probability at most $\exp\left(-Cr^{4}\right)$
would be negligible events in Case 1 and Case 2.I. In Case 2.II. the
same holds with $\exp\left(-Cp^{2}\log p\cdot r^{4}\right)$. Let
$C_{1}>0$ be a sufficiently large numerical constant, we then set
\[
A=C_{1}\left(\binmax 1{p^{2}\log p}\right).
\]
We choose $C_{1}$ such that the event $\kregevent c$ is negligible.

\subsubsection{The parameters $L$ and $\alpha$.}

Let $L>0$ be a large parameter. For $N\in\left\{ N_{0},\dots,N_{1}\right\} $
we set $\alpha=Nr^{-2}$, and remark that $\alpha\le3\log r$. We
will pick the precise value of $L$ later, but in all cases it holds
that $L=r+O\left(\frac{1}{r}\right)$, and therefore $L^{2}=r^{2}+O\left(1\right)$.
In the rest of the section, it will be more convenient to consider
the scaled polynomials
\[
\truncscaledpoly\left(z\right)=\sum_{k=0}^{N}\xi_{k}\frac{\left(Lz\right)^{k}}{\sqrt{k!}}.
\]
Rewriting (\ref{eq:zero_count_poly}) and (\ref{eq:lin_stats_poly})
in terms of $\truncscaledpoly$, we get
\begin{eqnarray}
n_{\truncscaledpoly}\left(\frac{r-K_{0}}{L}\right) & \le & n\left(r\right)\le n_{\truncscaledpoly}\left(\frac{r+K_{0}}{L}\right),\label{eq:zero_count_poly_L}\\
\left|n_{\gef}\left(\varphi;r\right)-n_{\truncscaledpoly}\left(\varphi;\frac{r}{L}\right)\right| & \le & CM_{0}\cdot\omega\left(\varphi;K_{0}r^{-1}\right).\label{eq:lin_stats_poly_L}
\end{eqnarray}

\subsection{\label{subsec:est_joint_dist_P_L}Estimates for the joint distribution
of the zeros of $\protect\truncscaledpoly$}

We denote the zeros of the polynomial $\truncscaledpoly$ by $z_{1},\dots,z_{N}$
(in uniform random order). In many cases it will be convenient to
use the vector notation $\vec z=\lvec zN$. Recall that $N=\alpha L^{2}+O\left(\alpha\right)$.
We will frequently use the notation $\left|\Delta\left(\vec z\right)\right|^{2}=\prod_{j\ne k}\left|z_{j}-z_{k}\right|$
and the (probability) measure
\[
\dd\refmeas\left(w\right)=\frac{L^{2}}{\pi}e^{-L^{2}\left|w\right|^{2}}\,\dd\lebmeas\left(w\right),
\]
where $\lebmeas$ is Lebesgue measure on $\bbc$. Using a change of
variables from the coefficients of the polynomial $\truncscaledpoly$
to the zeros (see Lemma \ref{lem:poly_change_of_vars}, Appendix \ref{sec:joint_density}),
we find that the joint distribution of the zeros, w.r.t. $\lebmeas\left(\vec z\right)$,
the Lebesgue measure on $\bbc^{N}$, is given by 
\begin{equation}
f\left(\vec z\right)=f\left(z_{1},\dots,z_{N}\right)=\normalconst\left|\Delta\left(\vec z\right)\right|^{2}\left(\int_{\bbc}\left|q_{\vec z}\left(w\right)\right|^{2}\,\dd\refmeas\left(w\right)\right)^{-\left(N+1\right)},\label{eq:joint_density}
\end{equation}
where $q_{\vec z}\left(w\right)=\prod_{j=1}^{N}\left(w-z_{j}\right)$
is the monic polynomial corresponding to $\truncscaledpoly$, and
the normalizing constant $\normalconst$ is given by
\begin{eqnarray}
\normalconst & = & \frac{N!\cdot\prod_{j=1}^{N}j!}{\pi^{N}L^{N\left(N+1\right)}}=\exp\left(\frac{1}{2}N^{2}\log\left(\frac{N}{L^{2}}\right)-\frac{3}{4}N^{2}+O\left(N\left(\log N+\log L\right)\right)\right).\nonumber \\
 & = & \exp\left(\frac{1}{2}N^{2}\log\left(\frac{N}{L^{2}}\right)-\frac{3}{4}N^{2}+O\left(L^{2}\log^{2}L\right)\right),\label{eq:normal_const_asymp}
\end{eqnarray}
where we used $\alpha\le3\log r=3\log L+O\left(1\right)$. By Lemma
\ref{lem:bern_markov}, Appendix \ref{sec:joint_density}, we have
\[
S\left(\vec z\right)\defeq\int_{\bbc}\left|q_{\vec z}\left(w\right)\right|^{2}\,\dd\refmeas\left(w\right)\ge\sup_{w\in\bbc}\left\{ \left|q_{\vec z}\left(w\right)\right|^{2}e^{-L^{2}\left|w\right|^{2}}\right\} \defeq A\left(\vec z\right).
\]

\subsubsection{Estimates for $A\left(\protect\vec z\right)$ and $S\left(\protect\vec z\right)$}

In order to bound the density (\ref{eq:joint_density}) from above,
we need a simple lower bound for $A\left(\vec z\right)$. We will
use the identity (\cite[Example 0.5.7]{saff2013logarithmic})
\begin{equation}
\frac{1}{2\pi}\int_{0}^{2\pi}\log\left|te^{i\theta}-z\right|\,\dd\theta=\log\left(\binmax t{\left|z\right|}\right).\label{eq:log_int}
\end{equation}

\begin{claim}
We have
\[
A\left(\vec z\right)\ge\left[\prod_{j=1}^{N}\left(\binmax 1{\left|z_{j}\right|}\right)\right]^{2}\exp\left(-L^{2}\right).
\]
\end{claim}
\begin{proof}
Using (\ref{eq:log_int}) with $t=1$, the inequality follows by replacing
supremum with an average over the unit circle,
\begin{eqnarray*}
\frac{1}{2}\log A\left(\vec z\right) & = & \sup_{w\in\bbc}\left\{ \log\left|q_{\vec z}\left(w\right)\right|-\frac{L^{2}}{2}\left|w\right|^{2}\right\} \ge\frac{1}{2\pi}\int_{0}^{2\pi}\log\left|q_{\vec z}\left(e^{i\theta}\right)\right|\,\dd\theta-\frac{L^{2}}{2}\\
 & = & \sum_{j=1}^{N}\log\left(\binmax 1{\left|z_{j}\right|}\right)-\frac{L^{2}}{2}.
\end{eqnarray*}
\end{proof}
Now we clearly have,
\begin{claim}
\label{claim:converg_fact}For $b>1$,
\[
\int_{\bbc^{N}}A\left(\vec z\right)^{-b}\,\dd\lebmeas\left(\vec z\right)\le\exp\left(bL^{2}\right)\cdot\left(\frac{Cb}{b-1}\right)^{N}.
\]
\end{claim}
\begin{proof}
By the previous claim,
\begin{eqnarray*}
\int_{\bbc^{N}}A\left(\vec z\right)^{-b}\,\dd\lebmeas\left(\vec z\right) & \le & \exp\left(bL^{2}\right)\cdot\int_{\bbc^{N}}\left[\prod_{j=1}^{N}\left(\binmax 1{\left|z_{j}\right|}\right)\right]^{-2b}\,\dd\lebmeas\left(\vec z\right)\\
 & = & \exp\left(bL^{2}\right)\cdot\left[\int_{\bbc}\left(\binmax 1{\left|z\right|}\right)^{-2b}\,\dd\lebmeas\left(z\right)\right]^{N}\\
 & \le & \exp\left(bL^{2}\right)\cdot\left(\frac{Cb}{b-1}\right)^{N}.
\end{eqnarray*}
\end{proof}
We will also need a probabilistic upper bound for $S\left(\vec z\right)$
.
\begin{claim}
\label{claim:bound_for_S(z)}On the event $\kregevent N$ we have
\[
S\left(\vec z\right)\le\exp\left(C\alpha\log\alpha\cdot L^{2}\right).
\]
\end{claim}
\begin{proof}
In Claim \ref{claim:integ_wrt_ref_meas}, Appendix \ref{sec:joint_density}
we show that
\[
\int_{\bbc}\left|q_{\vec z}\left(w\right)\right|^{2}\,\dd\refmeas\left(w\right)=\left(\sum_{k=0}^{N}\left|\xi_{k}\right|^{2}\right)\cdot\left(\left|\xi_{N}\right|^{2}\frac{L^{2N}}{N!}\right)^{-1}.
\]
On the event $\kregevent N$, applying (\ref{eq:reg_event_bound_for_xi})
(and using $N!\le N^{N}$), we have,
\[
\left(\left|\xi_{N}\right|^{-2}\sum_{k=0}^{N}\left|\xi_{k}\right|^{2}\right)\cdot\left(\frac{L^{2N}}{N!}\right)^{-1}\le\exp\left(Cr^{2}\right)\left(\frac{N}{L^{2}}\right)^{N}\le\exp\left(CL^{2}+C\alpha L^{2}\log\alpha\right)\le\exp\left(C\alpha\log\alpha\cdot L^{2}\right).
\]
\end{proof}

\subsubsection{Upper bound for the probability}

Consider a set $Z\subset\bbc^{N}$. We think about $Z$ as a collection
of possible `configurations' of the zeros of $\truncscaledpoly$.
We are interested in bounding the probability of these configurations.
We introduce the functional $\funcdiscname:\bbc^{N}\to\bbr$:
\[
\funcdisc{\vec z}=\begin{cases}
2\sup_{w\in\bbc}\left\{ \frac{1}{N}\cdot\log\left|q_{\vec z}\left(w\right)\right|-\frac{L^{2}}{2N}\left|w\right|^{2}\right\} -\frac{1}{N^{2}}\sum_{j\ne k}\log\left|z_{j}-z_{k}\right| & \,\forall j\ne k,\,z_{j}\ne z_{k};\\
\infty & \,\mbox{otherwise}.
\end{cases}.
\]
Notice that the supremum term above is equal to $\frac{1}{2N}\log A\left(\vec z\right)$.
We will show, that at the exponential scale, the probability of the
configurations we consider is bounded above by the minimum of the
functional $\funcdiscname$ over $Z$.

Rewriting the joint density of the zeros (\ref{eq:joint_density}),
we have
\[
f\left(\vec z\right)=\normalconst\left|\Delta\left(\vec z\right)\right|^{2}S\left(\vec z\right)^{-\left(N+1\right)}=\normalconst\exp\left(\sum_{j\ne k}\log\left|z_{j}-z_{k}\right|\right)S\left(\vec z\right)^{-\left(N+1\right)}.
\]
On the event $\kregevent N$ (and using $S\left(\vec z\right)\ge A\left(\vec z\right)$)
we get
\begin{eqnarray*}
S\left(\vec z\right){}^{-\left(N+1\right)} & \le & A\left(\vec z\right){}^{-\left(1+\frac{1}{N}\right)}S\left(\vec z\right)^{\frac{1}{N}}A\left(\vec z\right){}^{-N}\\
 & \le & A\left(\vec z\right){}^{-\left(1+\frac{1}{N}\right)}\exp\left(C\log\alpha\right)\cdot\exp\left(-N\log A\left(\vec z\right)\right).
\end{eqnarray*}
Thus, if we introduce the set
\[
E=Z\cap\setd{\vec z\in\bbc^{N}}{S\left(\vec z\right)\le\exp\left(C\alpha\log\alpha\cdot L^{2}\right)},
\]
then by combining Claim \ref{claim:converg_fact} (with $b=1+\frac{1}{N}$),
Claim \ref{claim:bound_for_S(z)}, and (\ref{eq:normal_const_asymp}),
we obtain
\begin{eqnarray*}
\pr{Z\cap\kregevent N} & \le & \normalconst\cdot\exp\left(C\log\alpha\right)\cdot\int_{E}\exp\left(\sum_{j\ne k}\log\left|z_{j}-z_{k}\right|-N\log A\left(\vec z\right)\right)A\left(\vec z\right)^{-\left(1+\frac{1}{N}\right)}\,\dd\lebmeas\left(\vec z\right)\\
 & \le & \normalconst\cdot\exp\left(C\log\alpha\right)\cdot\left(CN\right)^{N}\cdot\exp\left(-N^{2}\cdot\inf_{\vec z\in E}\funcdisc{\vec z}\right)\\
 & \le & \exp\left(-N^{2}\cdot\inf_{\vec z\in Z}\funcdisc{\vec z}+\frac{1}{2}N^{2}\log\left(\frac{N}{L^{2}}\right)-\frac{3}{4}N^{2}+O\left(L^{2}\log^{2}L\right)\right)\cdot
\end{eqnarray*}
Let $\nu\in\probmeas$ be a probability measure. We now introduce
the functional
\[
\funccontparam{\alpha}{\nu}=2\sup_{w\in\bbc}\left\{ \logpot{\nu}w-\frac{\left|w\right|^{2}}{2\alpha}\right\} -\logenerg{\nu},
\]
where $\logpot{\nu}w$ and $\logenerg{\nu}$ are the logarithmic potential
and the logarithmic energy of the measure $\nu$, respectively (see
notation in Section \ref{sec:intro}). We discuss this functional
in more details in Section \ref{sec:sol_of_energy_prob}.

Define the empirical probability measure of the zeros by
\[
\mu_{\vec z}=\frac{1}{N}\sum_{j=1}^{N}\delta_{z_{j}},
\]
where $\delta_{z}$ is a Dirac delta measure at the point $z\in\bbc$.
A technical issue is the fact that the logarithmic energy of the empirical
measure is not defined. To resolve it, we smoothen the empirical measure
by defining
\[
\mu_{\vec z}^{t}=\mu_{\vec z}\star\lebmeas_{\left|z\right|=t}=\frac{1}{N}\sum_{j=1}^{N}\lebmeas_{\left|z-z_{j}\right|=t},
\]
where $\lebmeas_{\left|z-w\right|=t}$ is the (normalized) Lebesgue
measure on the circle $\left|z-w\right|=t$. We now wish to compare
$\funcdisc{\vec z}$ and $\funccontparam{\alpha}{\mu_{\vec z}^{t}}$.

\subsubsection{\label{subsec:estimate_for_functionals}Comparing $\protect\funcdisc{\protect\vec z}$
and $\protect\funccontparam{\alpha}{\mu_{\protect\vec z}^{t}}$}

We will now show that for $t>0$ sufficiently small we have
\[
\funcdisc{\vec z}\ge\funccontparam{\alpha}{\mu_{\vec z}^{t}}-C\left(\frac{1}{N}\log\frac{1}{t}+\frac{t}{\sqrt{\alpha}}+\frac{1}{L^{2}}\right).
\]
The proof consists of two simple claims.
\begin{claim}
We have
\[
\frac{1}{N^{2}}\sum_{j\ne k}\log\left|z_{j}-z_{k}\right|\le\logenerg{\mu_{\vec z}^{t}}+\frac{C\log\frac{1}{t}}{N}.
\]
\end{claim}
\begin{proof}
Note that
\[
\int\log\left|z-w\right|\,\dd\lebmeas_{\left|w-a\right|=t}=\frac{1}{2\pi}\int_{0}^{2\pi}\log\left|z-a+te^{i\theta}\right|\,\dd\theta=\log\left(\binmax{\left|z-a\right|}t\right).
\]
Since the logarithm is a subharmonic function, we have
\begin{eqnarray*}
\frac{1}{N^{2}}\sum_{j\ne k}\log\left|z_{j}-z_{k}\right| & \le & \frac{1}{N^{2}}\sum_{j\ne k}\int\log\left|z-w\right|\,\dd\lebmeas_{\left|z-z_{j}\right|=t}\dd\lebmeas_{\left|w-z_{k}\right|=t}\\
 & \le & \frac{1}{N^{2}}\sum_{j,k=1}^{N}\int\log\left|z-w\right|\,\dd\lebmeas_{\left|z-z_{j}\right|=t}\dd\lebmeas_{\left|w-z_{k}\right|=t}+\frac{C\log\frac{1}{t}}{N}\\
 & = & \int\log\left|z-w\right|\,\dd\mu_{\vec z}^{t}\left(z\right)\dd\mu_{\vec z}^{t}\left(w\right)+\frac{C\log\frac{1}{t}}{N}.
\end{eqnarray*}
\end{proof}
Let us write $B_{\alpha}\left(\nu\right)=\funccontparam{\alpha}{\nu}+\logenerg{\nu}=2\cdot\sup\setd{\logpot{\nu}w-\frac{\left|w\right|^{2}}{2\alpha}}{w\in\bbc}.$
By Claim \ref{claim:sup_log_pot} in Appendix \ref{sec:poten_theory},
we have
\[
B_{\alpha}\left(\nu\right)=2\sup_{\left|w\right|\le\sqrt{\alpha}}\left\{ \logpot{\nu}w-\frac{\left|w\right|^{2}}{2\alpha}\right\} .
\]

\begin{claim}
For $t>0$ sufficiently small, we have
\[
\frac{1}{N}\log A\left(\vec z\right)=2\sup_{w\in\bbc}\left\{ \frac{1}{N}\log\left|q_{\vec z}\left(w\right)\right|-\frac{L^{2}}{2N}\left|w\right|^{2}\right\} \ge B_{\alpha}\left(\mu_{\vec z}^{t}\right)-C\left(\frac{1}{L^{2}}+\frac{t}{\sqrt{\alpha}}\right).
\]
\end{claim}
\begin{proof}
Using the fact $\mu_{\vec z}^{t}$ is the convolution of $\mu_{\vec z}$
and $\lebmeas_{\left|z\right|=t}$ (and the linearity of $\mu\mapsto\logpot{\mu}w$),
we can write
\[
B_{\alpha}\left(\mu_{\vec z}^{t}\right)=2\sup_{\left|w\right|\le\sqrt{\alpha}}\left\{ \logpot{\mu_{\vec z}^{t}}w-\frac{\left|w\right|^{2}}{2\alpha}\right\} =2\sup_{\left|w\right|\le\sqrt{\alpha}}\left\{ \frac{1}{2\pi}\int_{0}^{2\pi}\logpot{\mu_{\vec z}}{w+te^{i\theta}}\,\dd\theta-\frac{\left|w\right|^{2}}{2\alpha}\right\} .
\]
By the definition of $A\left(\vec z\right)$, for any $\theta\in\left[0,2\pi\right]$,
\begin{eqnarray*}
\logpot{\mu_{\vec z}}{w+te^{i\theta}}-\frac{L^{2}}{2N}\left|w+te^{i\theta}\right|^{2} & = & \frac{1}{N}\sum_{j=1}^{N}\log\left|w+te^{i\theta}-z_{j}\right|-\frac{L^{2}}{2N}\left|w+te^{i\theta}\right|^{2}\\
 & = & \frac{1}{N}\log\left|q_{\vec z}\left(w+te^{i\theta}\right)\right|-\frac{L^{2}}{2N}\left|w+te^{i\theta}\right|^{2}\le\frac{1}{2N}\log A\left(\vec z\right).
\end{eqnarray*}
Recall that $\frac{N}{L^{2}}=\frac{\alpha r^{2}}{L^{2}}=\alpha\left(1+O\left(L^{-2}\right)\right)$.
Notice that for $\left|w\right|\le\sqrt{\alpha}$ and for $t\le1$,
we have
\[
\frac{L^{2}}{2N}\left|w+te^{i\theta}\right|^{2}\le\frac{L^{2}}{2N}\left|w\right|^{2}+\frac{Ct}{\sqrt{\alpha}}\le\frac{\left|w\right|^{2}}{2\alpha}\left(\frac{\alpha L^{2}}{N}\right)+\frac{Ct}{\sqrt{\alpha}}\le\frac{\left|w\right|^{2}}{2\alpha}+\frac{C}{L^{2}}+\frac{Ct}{\sqrt{\alpha}}.
\]
Hence,
\begin{align*}
\frac{1}{2\pi}\int_{0}^{2\pi}\logpot{\mu_{\vec z}}{w+te^{i\theta}}\,\dd\theta-\frac{\left|w\right|^{2}}{2\alpha} & \le\frac{1}{2\pi}\int_{0}^{2\pi}\left[\logpot{\mu_{\vec z}}{w+te^{i\theta}}-\frac{L^{2}}{2N}\left|w+te^{i\theta}\right|^{2}\right]\,\dd\theta+\frac{C}{L^{2}}+\frac{Ct}{\sqrt{\alpha}}\\
 & \le\frac{1}{N}\sup_{w\in\bbc}\left\{ \log\left|q_{\vec z}\left(w\right)\right|-\frac{1}{2}L^{2}\left|w\right|^{2}\right\} +\frac{C}{L^{2}}+\frac{Ct}{\sqrt{\alpha}}.
\end{align*}
\end{proof}

\subsection{\label{subsec:reduc_mod_energ_problem}Reduction to a modified weighted
energy problem}

For a set $Z_{N}\subset\bbc^{N}$, after combining the estimates above
(and using $N=\alpha L^{2}+O\left(\alpha\right)$, $\alpha\le C\log L$)
, we find that
\begin{align}
\pr{Z_{N}\cap\kregevent N} & \le\label{eq:prob_upp_bnd_est}\\
 & \exp\left(-N^{2}\left[\inf_{\vec z\in Z_{N}}\funccontparam{\alpha}{\mu_{\vec z}^{t}}-\frac{1}{2}\log\left(\frac{N}{L^{2}}\right)+\frac{3}{4}\right]+L^{2}\log L\cdot O\left(\log L+\log\frac{1}{t}+tL^{2}\sqrt{\log L}\right)\right)\cdot\nonumber 
\end{align}

\subsubsection{Choosing the parameters $t$ and $L$}

We now choose $t=r^{-C_{2}}$ (with $C_{2}\ge4$) and recall that
$\gamma=r^{-C_{2}}$, $K_{0}=16B^{3}r^{3}\gamma\le Cr^{3-C_{2}}=O\left(\frac{1}{r}\right)$.
Consider again the two cases we described at the beginning of the
section. In Case 1, we set $L=\left(1+t\right)^{-1}\left(r-K_{0}\right)$.
Since $L=r+O\left(\frac{1}{r}\right)$, and using (\ref{eq:zero_count_poly_L})
we find
\begin{equation}
n_{\truncscaledpoly}\left(1+t\right)=n_{\truncpoly}\left(r-K_{0}\right)\le n\left(r\right)\le pr^{2}=\frac{pN}{\alpha}\implies\mu_{\vec z}^{t}\left(D\right)\le\frac{p}{\alpha}.\label{eq:few_zeros}
\end{equation}
Similarly, in Case 2, we set $L=\left(1-t\right)^{-1}\left(r+K_{0}\right)$,
and get
\begin{equation}
n_{\truncscaledpoly}\left(1-t\right)=n_{\truncpoly}\left(r+K_{0}\right)\ge n\left(r\right)\ge pr^{2}=\frac{pN}{\alpha}\implies\mu_{\vec z}^{t}\left(\overline{D}\right)\ge\frac{p}{\alpha}.\label{eq:many_zeros}
\end{equation}
Define the set
\[
L_{\varphi,\tau,\lambda}^{N}=L_{\varphi,\tau,\lambda}^{N}\left(t\right)=\setd{\vec z}{\left|\frac{1}{N}\sum_{j=1}^{N}\varphi\left(z_{j}\right)-\tau\right|\ge\lambda+\omega\left(\varphi;t\right)}.
\]
Notice that $\left|\varphi\left(z_{j}\right)-\int\varphi\,\dd\lebmeas_{\left|z-z_{j}\right|=t}\right|\le\omega\left(\varphi;t\right)$,
hence, if $\vec z\in L_{\varphi,\tau,,\lambda}^{N}$, then
\[
\left|\int_{\bbc}\varphi\left(w\right)\,\dd\mu_{\vec z}^{t}\left(w\right)-\tau\right|\ge\lambda.
\]

\subsubsection{Completing the reduction}

For $x\ge0$, we define the following sets of measures:
\begin{eqnarray*}
\calF{}_{x} & = & \setd{\nu\in\probmeas}{\nu\left(D\right)\le\frac{x}{\alpha}},\\
\calM_{x} & = & \setd{\nu\in\probmeas}{\nu\left(\overline{D}\right)\ge\frac{x}{\alpha}},\\
\calL_{\varphi,\tau,x} & = & \setd{\nu\in\probmeas}{\left|\int_{\bbc}\varphi\left(w\right)\,\dd\nu\left(w\right)-\tau\right|\ge x}.
\end{eqnarray*}
Clearly we have,
\begin{equation}
\setd{\mu_{\vec z}^{t}}{\mu_{\vec z}^{t}\left(D\right)\le\frac{p}{\alpha}}\subset\calF_{p},\,\setd{\mu_{\vec z}^{t}}{\mu_{\vec z}^{t}\left(\overline{D}\right)\ge\frac{p}{\alpha}}\subset\calM_{p},\,\setd{\mu_{\vec z}^{t}}{\vec z\in L_{\varphi,\tau,\lambda}^{N}}\subset\calL_{\varphi,\tau,\lambda}.\label{eq:inclusion_of_meas}
\end{equation}
We thus reduced an estimate for the probability, to a minimization
problem for a functional acting on (general) probability measures.

\subsection{The upper bound in Theorem \ref{thm:very_large_fluct}}

We remind the two cases of the theorem: 
\begin{itemize}
\item[Case 1. ]  $n\left(r\right)\le pr^{2}$, where $p\in\left[0,1\right)$.
\item[Case 2. ]  ${\rm I}$: $n\left(r\right)\ge pr^{2}$, where $p\in\left(1,e\right)$.
${\rm II}$: $n\left(r\right)\ge pr^{2}$, where $p\in\left[e,\infty\right)$.
\end{itemize}
Recall $\alpha\le3\log r$, $t=r^{-C_{2}}$ ($C_{2}\ge4$), $L=r+O\left(\frac{1}{r}\right)$,
and $N=\alpha r^{2}=\alpha L^{2}+O\left(\alpha\right)$. In Case 1,
using (\ref{eq:prob_upp_bnd_est}), (\ref{eq:few_zeros}) and the
first inclusion (\ref{eq:inclusion_of_meas}), we get
\begin{multline*}
\log\pr{\left\{ n\left(r\right)\le pr^{2}\right\} \cap\kregevent N}\le\log\pr{\left\{ n_{\truncscaledpoly}\left(1+t\right)\le\frac{pN}{\alpha}\right\} \cap\kregevent N}\\
\quad\,\quad\le-N^{2}\left[\inf_{\nu\in\calF_{p}}\funccontparam{\alpha}{\nu}-\frac{1}{2}\log\left(\frac{N}{L^{2}}\right)+\frac{3}{4}\right]+L^{2}\log L\cdot O\left(\log L+\log\frac{1}{t}+tL^{2}\sqrt{\log L}\right)\\
=-N^{2}\left[\inf_{\nu\in\calF_{p}}\funccontparam{\alpha}{\nu}-\frac{1}{2}\log\alpha+\frac{3}{4}\right]+O\left(\frac{N^{2}}{L^{2}}+r^{2}\log^{2}r\right)\quad\quad\quad\quad\quad\quad\quad\quad\quad\\
=-N^{2}\left[\inf_{\nu\in\calF_{p}}\funccontparam{\alpha}{\nu}-\frac{1}{2}\log\alpha+\frac{3}{4}\right]+O\left(r^{2}\log^{2}r\right).\quad\quad\quad\quad\quad\quad\quad\quad\quad\quad\quad\quad\quad\quad\quad\quad
\end{multline*}
By \ref{claim:props_min_meas_F_p}, Subsection \ref{subsec:min_meas_p_less_1},
the minimal value of the functional $\funccontparam{\alpha}{\nu}$
over the set $\calF_{p}=\setd{\nu\in\probmeas}{\nu\left(D\right)\le\frac{p}{\alpha}}$
is given by
\begin{eqnarray*}
\inf_{\nu\in\calF_{p}}\funccontparam{\alpha}{\nu} & = & \frac{1}{2}\log\alpha-\frac{3}{4}+\frac{q^{2}\left(2\log q-1\right)}{4\alpha^{2}}-\frac{p^{2}\left(2\log p-1\right)}{4\alpha^{2}},
\end{eqnarray*}
where $q=q\left(p\right)>1$ is the solution of the equation $q\left(\log q-1\right)=p\left(\log p-1\right)$.

Summing up over $N\in\left\{ N_{0},\dots,N_{1}\right\} $, we get
\begin{eqnarray}
\pr{\left\{ n\left(r\right)\le pr^{2}\right\} \cap\regevent} & \le & \sum_{N=N_{0}}^{N_{1}}\pr{\left\{ n_{\truncscaledpoly}\left(1+t\right)\le\frac{pN}{\alpha}\right\} \cap\kregevent N}\nonumber \\
 & \le & \sum_{N=N_{0}}^{N_{1}}\exp\left(-\frac{N^{2}}{4\alpha^{2}}\left[q^{2}\left(2\log q-1\right)-p^{2}\left(2\log p-1\right)\right]+O\left(r^{2}\log^{2}r\right)\right)\nonumber \\
 & = & \exp\left(-\frac{r^{4}}{4}\left[q^{2}\left(2\log q-1\right)-p^{2}\left(2\log p-1\right)\right]+O\left(r^{2}\log^{2}r+\log r\right)\right)\nonumber \\
 & \le & \exp\left(-\frac{1}{4}\left[q^{2}\left(2\log q-1\right)-p^{2}\left(2\log p-1\right)\right]r^{4}+O\left(r^{2}\log^{2}r\right)\right).\label{eq:few_zeros_prob_upp_bnd}
\end{eqnarray}
Notice that $\kregevent c$ is a negligible event, and therefore we
can use the simple bound $\pr{\left\{ n\left(r\right)\le pr^{2}\right\} }\le\pr{\left\{ n\left(r\right)\le pr^{2}\right\} \cap\regevent}+\pr{\kregevent c}$.
The upper bounds in Case 2.I. and Case 2.II. are obtained in a similar
way. We leave the details for the reader. This completes the proof
of the upper bound in Theorem \ref{thm:very_large_fluct}.

\section{\label{sec:sol_of_energy_prob}Modified energy problems}

Let $\nu\in\probmeas$ be a probability measure and $\alpha>0$. In
this section we consider the problem of minimizing the functional

\[
\funccontparam{\alpha}{\nu}=2\sup_{w\in\bbc}\left\{ \logpot{\nu}w-\frac{\left|w\right|^{2}}{2\alpha}\right\} -\logenerg{\nu}\defeq B_{\alpha}\left(\nu\right)-\logenerg{\nu},
\]
over probability measures with compact support, restricted to certain
(closed and convex) subsets of $\probmeas$. We mention that this
functional is lower semi-continuous and strictly convex. It is known
that $\logenerg{\nu}$ is an upper semi-continuous and strictly concave
functional (\cite[Proposition 2.2]{hiai1998maximizing}). Note that
$B_{\alpha}\left(\nu\right)$ is lower semi-continuous and convex
as the supremum of affine functionals. This implies that a unique
minimizer of $\funccontparam{\alpha}{\nu}$ exists over closed and
convex subsets of $\probmeas$ . In a more general setting, it is
proved in \cite[Lemma 29]{zeitouni2010large} that the global minimizer
of $\funccontparam{\alpha}{\nu}$ is the uniform probability measure
on the disk $\disc 0{\sqrt{\alpha}}$.

It will be useful to make the following definition
\[
g_{\nu}\left(z\right)=\logpot{\nu}z-\frac{\left|z\right|^{2}}{2\alpha}-\frac{B_{\alpha}\left(\nu\right)}{2}.
\]
Notice that $g_{\nu}\left(z\right)\le0$ for all $z\in\bbc$.
\begin{rem}
Since in our case the minimizers turn out to be compactly supported,
it is enough to identify the minimizing measure among measures with
compact support (since we can approximate an arbitrary measure using
measures with compact support). In general, some energy problems can
have solutions which are not compactly supported, even in case the
global minimizer has compact support (e.g., \cite[Theorem 1.3]{ghosh2015large}).
\end{rem}

\subsection{The principle of domination}

The key tool that we need from potential theory is a special case
of the principle of domination (\cite[pg. 43]{saff2013logarithmic}).
\begin{thm}
Let $\nu$and $\mu$ be probability measures with compact support
and finite logarithmic energy. If
\[
\logpot{\nu}z\le\logpot{\mu}z+C,\quad\mu-\mbox{a.e. in }z,
\]
then
\[
\logpot{\nu}z\le\logpot{\mu}z+C,\quad\forall z\in\bbc.
\]
\end{thm}
We apply it using the next claim (cf. \cite[Lemma 29]{zeitouni2010large}).
\begin{claim}
\label{claim:meas_compr}Let $\nu$ and $\mu$ be probability measures
with compact support and finite logarithmic energy. If $\mu-\mbox{a.e. in }z$
we have $g_{\nu}\left(z\right)\le g_{\mu}\left(z\right)$, then $\funccontparam{\alpha}{\mu}\le\funccontparam{\alpha}{\nu}$.
\end{claim}
\begin{proof}
By the assumption and using the previous theorem, we find
\[
\logpot{\nu}z-\frac{B_{\alpha}\left(\nu\right)}{2}\le\logpot{\mu}z-\frac{B_{\alpha}\left(\mu\right)}{2},\quad z\in\bbc.
\]
Now,
\begin{eqnarray*}
\logenerg{\nu} & = & \int_{\bbc}\logpot{\nu}z\,\dd\nu\left(z\right)\le\int_{\bbc}\logpot{\mu}z\,\dd\nu\left(z\right)+\frac{B_{\alpha}\left(\nu\right)-B_{\alpha}\left(\mu\right)}{2}=\int_{\bbc}\logpot{\nu}z\,\dd\mu\left(z\right)+\frac{B_{\alpha}\left(\nu\right)-B_{\alpha}\left(\mu\right)}{2}\\
 & \le & \int_{\bbc}\logpot{\mu}z\,\dd\mu\left(z\right)+B_{\alpha}\left(\nu\right)-B_{\alpha}\left(\mu\right)=\logenerg{\mu}+B_{\alpha}\left(\nu\right)-B_{\alpha}\left(\mu\right).
\end{eqnarray*}
\end{proof}
Therefore, in order to establish that a certain measure $\mu$ minimizes
the value of the functional $\funccontparam{\alpha}{\nu}$ (maybe
over some subset of $\probmeas$), it is sufficient to show that for
any other measure $\nu$, the inequality $g_{\nu}\left(z\right)\le g_{\mu}\left(z\right)$
is satisfied on the support of $\mu$ (this is mainly useful for problems
where the minimizer is a radially symmetric measure).

\subsection{Identifying the minimizing measures}

We consider here the solution of rotation symmetric minimization problems
for the functional $\funccontparam{\alpha}{\nu}$ (that is, a rotation
of a measure that satisfies the constraint continues to satisfy it).
If $\mu$ is a measure that satisfies a rotation symmetric constraint,
then its radial symmetrization $\mu_{{\rm rad}}$ also satisfies this
constraint (we define $\mu_{{\rm rad}}$ as the normalized integral
over all the rotations of $\mu$ w.r.t. the origin). By the radial
symmetry and the convexity of the functional $\funccontparam{\alpha}{\nu}$,
we have that $\funccontparam{\alpha}{\mu_{{\rm rad}}}\le\funccontparam{\alpha}{\mu}$.
This implies the solution is a radial measure, and we will consider
these minimization problems over the set of radially symmetric measures.

\subsubsection{\label{subsec:min_meas_p_less_1}The case $p\in\left[0,1\right)$}

We now consider the minimization problem over the set of measures
\[
\calF_{p}=\calF_{p,\alpha}=\setd{\nu\in\probmeas}{\nu\left(D\right)\le\frac{p}{\alpha}}.
\]
Notice that this is a closed set of measures, since $D$ is an open
set. It is also clearly convex.

We define $q=q\left(p\right)>1$ as the solution of the equation $q\left(\log q-1\right)=p\left(\log p-1\right)$.
We now show that the minimizing measure is given by
\[
\minmeas p\left(z\right)=\minmeassup p{\alpha}\left(z\right)=\frac{1}{\alpha}\left[\ind{\left\{ \left|z\right|\le\sqrt{p}\right\} }z+\ind{\left\{ \sqrt{q}\le\left|z\right|\le\sqrt{\alpha}\right\} }z\right]\cdot\frac{\lebmeas\left(z\right)}{\pi}+\frac{q-p}{\alpha}\subslebmeas{\left|z\right|=1},
\]
where $\subslebmeas{\left|z\right|=1}$ is the normalized Lebesgue
measure on $\left\{ \left|z\right|=1\right\} $, and $\lebmeas$ is
Lebesgue measure on $\bbc$. After a straightforward computation one
obtains the following
\begin{claim}
\label{claim:props_min_meas_F_p} The logarithmic potential of the
measure $\minmeas p$ given above (on its support) is:
\[
\logpot{\minmeas p}z=\frac{\log\alpha}{2}-\frac{1}{2}+\begin{cases}
\frac{\left|z\right|^{2}}{2\alpha} & \,0\le\left|z\right|\le\sqrt{p};\\
\frac{q\left(1-\log q\right)}{2\alpha} & \,\left|z\right|=1;\\
\frac{\left|z\right|^{2}}{2\alpha} & \,\sqrt{q}\le\left|z\right|\le\sqrt{\alpha}.
\end{cases}
\]
In addition, we have
\[
B\left(\minmeas p\right)=\log\alpha-1,\quad\funccontparam{\alpha}{\minmeas p}=\frac{\log\alpha}{2}-\frac{3}{4}+\frac{q^{2}\left(2\log q-1\right)}{4\alpha^{2}}-\frac{p^{2}\left(2\log p-1\right)}{4\alpha^{2}}.
\]
\end{claim}
The results above imply that $g_{\minmeas p}\left(z\right)=0$ on
the support of $\minmeas p$ except for $\left|z\right|=1$, and there
we have $g_{\minmeas p}\left(z\right)=\frac{q\left(1-\log q\right)}{2\alpha}-\frac{1}{2\alpha}$.
It remains to prove the following simple result.
\begin{claim}
Let $\nu$ be a radially symmetric measure with compact support. If
$\nu\left(D\right)\le\frac{p}{\alpha}$, then $g_{\nu}\left(z\right)\le\frac{q\left(1-\log q\right)}{2\alpha}-\frac{1}{2\alpha}$
for all $z$ with $\left|z\right|=1$.
\end{claim}
\begin{proof}
Since $\nu$ is radial we have by Jensen's formula (\ref{thm:Jensen_for}),
\[
\logpot{\nu}1-\logpot{\nu}{\sqrt{p}}=\int_{\sqrt{p}}^{1}\frac{\nu\left(\left|z\right|\le t\right)}{t}\,\dd t\le\frac{p}{\alpha}\int_{\sqrt{p}}^{1}\frac{1}{t}\,\dd t=-\frac{p\log p}{2\alpha}.
\]
Now,
\begin{eqnarray*}
g_{\nu}\left(z\right) & = & g_{\nu}\left(1\right)=\logpot{\nu}1-\frac{1}{2\alpha}-\frac{B_{\alpha}\left(\nu\right)}{2}\le\logpot{\nu}1-\frac{1}{2\alpha}-\left[\logpot{\nu}{\sqrt{p}}-\frac{p}{2\alpha}\right]\\
 & \le & \frac{p\left(1-\log p\right)}{2\alpha}-\frac{1}{2\alpha}=\frac{q\left(1-\log q\right)}{2\alpha}-\frac{1}{2\alpha}.
\end{eqnarray*}
\end{proof}
Since for any radially symmetric measure $\nu$ with compact support
we have $g_{\nu}\left(z\right)\le g_{\minmeas p}\left(z\right)$ on
the support of $\minmeas p$, Claim \ref{claim:meas_compr} implies
that $\minmeas p$ is the minimizing measure over the set $\calF_{p}$.

\subsubsection{The case $p>1$}

In a similar way to the previous problem, we now consider the minimization
problem over the convex set of measures
\[
\calM_{p}=\calM_{p,\alpha}=\setd{\nu\in\probmeas}{\nu\left(\overline{D}\right)\ge\frac{p}{\alpha}},
\]
where $p\in\left(1,\alpha\right)$. Notice that in this case the set
$\calM_{p}$ is closed, since $\overline{D}$ is a closed set.

Here there are two cases. In case $p\in\left(1,e\right)$ the minimizing
measure is given by
\[
\minmeas p\left(z\right)=\minmeassup p{\alpha}\left(z\right)=\frac{1}{\alpha}\left[\ind{\left\{ \left|z\right|\le\sqrt{q}\right\} }z+\ind{\left\{ \sqrt{p}\le\left|z\right|\le\sqrt{\alpha}\right\} }z\right]\cdot\frac{\lebmeas\left(z\right)}{\pi}+\frac{p-q}{\alpha}\subslebmeas{\left|z\right|=1},
\]
where $q=q\left(p\right)<1$ is defined as the solution of the equation
$q\left(\log q-1\right)=p\left(\log p-1\right)$. It is not difficult
to check that $g_{\minmeas p}\left(z\right)$ vanishes on the support
of $\minmeas p$, except for $\left|z\right|=1$. That is, the supremum
of $\logpot{\nu}w-\frac{\left|w\right|^{2}}{2\alpha}$ is attained
on the (non-singular) support of $\minmeas p$.

Notice that as $p$ tends to $e$, $q\left(p\right)$ tends to $0$.
In case $p\in\left[e,\alpha\right)$ the measure is given by
\[
\minmeas p\left(z\right)=\frac{p}{\alpha}\subslebmeas{\left|z\right|=1}+\frac{1}{\alpha}\ind{\left\{ \sqrt{p}\le\left|z\right|\le\sqrt{\alpha}\right\} }z\cdot\frac{\lebmeas\left(z\right)}{\pi},
\]
and we see that the $\minmeas p$-measure of the interior of the unit
disk is zero. Again $g_{\minmeas p}\left(z\right)$ vanishes on the
support of $\minmeas p$, except for $\left|z\right|=1$.

The proofs that the measures above are the minimizers of the functional
$\funccontparam{\alpha}{\nu}$ over the set $\calM_{p}$ are very
similar to the case $p<1$, and are left to the reader. A straightforward
computation gives
\[
\funccontparam{\alpha}{\minmeas p}=\begin{cases}
\frac{\log\alpha}{2}-\frac{3}{4}+\frac{q^{2}\left(2\log q-1\right)}{4\alpha^{2}}-\frac{p^{2}\left(2\log p-1\right)}{4\alpha^{2}} & \,p\in\left(1,e\right);\\
\frac{\log\alpha}{2}-\frac{3}{4}-\frac{p^{2}\left(2\log p-1\right)}{4\alpha^{2}} & \,p\in\left[e,\alpha\right).
\end{cases}
\]

\subsection{`Variational' characterization of the minimizers}

The following simple results will be of use in Section \ref{sec:conv_emp_meas}.
\begin{claim}
Let $\mu,\nu\in\probmeas$ be probability measure with finite logarithmic
energy and let $t\in\left[0,1\right]$ be small. Then
\[
\logenerg{t\nu+\left(1-t\right)\mu}=\logenerg{\mu}-2t\left[\logenerg{\mu}-\int_{\bbc}\logpot{\nu}w\,\dd\mu\left(w\right)\right]+O\left(t^{2}\right).
\]
\end{claim}
\begin{proof}
From the definition of the logarithmic energy we have
\begin{eqnarray*}
\logenerg{t\nu+\left(1-t\right)\mu} & = & t^{2}\logenerg{\nu}+\left(1-t\right)^{2}\logenerg{\mu}+2t\left(1-t\right)\int_{\bbc}\logpot{\nu}w\,\dd\mu\left(w\right)\\
 & = & \logenerg{\mu}-2t\left[\logenerg{\mu}-\int_{\bbc}\logpot{\nu}w\,\dd\mu\left(w\right)\right]+O\left(t^{2}\right).
\end{eqnarray*}
\end{proof}
\begin{claim}
Let $\mu,\nu\in\probmeas$ and let $t\in\left[0,1\right]$. Then
\[
B_{\alpha}\left(t\nu+\left(1-t\right)\mu\right)\le B_{\alpha}\left(\mu\right)+t\left[B_{\alpha}\left(\nu\right)-B_{\alpha}\left(\mu\right)\right].
\]
\end{claim}
\begin{proof}
By the definition of $B\left(\nu\right)$, and the linear properties
of the logarithmic potential,
\begin{eqnarray*}
B_{\alpha}\left(t\nu+\left(1-t\right)\mu\right) & = & 2\sup_{w\in\bbc}\left\{ t\logpot{\nu}w+\left(1-t\right)\logpot{\mu}w-\frac{\left|w\right|^{2}}{2\alpha}\right\} \\
 & \le & t\cdot2\sup_{w\in\bbc}\left\{ \logpot{\nu}w-\frac{\left|w\right|^{2}}{2\alpha}\right\} +\left(1-t\right)\cdot2\sup_{w\in\bbc}\left\{ \logpot{\mu}w-\frac{\left|w\right|^{2}}{2\alpha}\right\} \\
 & = & tB_{\alpha}\left(\nu\right)+\left(1-t\right)B_{\alpha}\left(\mu\right).
\end{eqnarray*}
\end{proof}
In the following lemma we derive a characterization of the minimizers
of the functional $\funccontparam{\alpha}{\nu}$.
\begin{lem}
\label{lem:variational_char}Let $\calC\subset\probmeas$ be a closed
and convex set of measures. Suppose $\mu_{\mathrm{min}}\in\calC$
is the unique minimizer of $\funccontparam{\alpha}{\nu}$ over the
set $\calC$. For $\mu\in\calC$, we have that
\[
\int_{\bbc}\logpot{\nu}w\,\dd\mu\left(w\right)-\frac{B_{\alpha}\left(\nu\right)}{2}\le\int_{\bbc}\logpot{\mu}w\,\dd\mu\left(w\right)-\frac{B_{\alpha}\left(\mu\right)}{2},\quad\forall\nu\in\calC,
\]
if and only if $\mu=\mu_{\mathrm{min}}$. 
\end{lem}
\begin{proof}
Suppose $\mu=\mu_{\mathrm{min}}$ and let $t>0$ be small. For $\nu\in\calC$,
we have that $\mu_{t}=t\nu+\left(1-t\right)\mu_{\mathrm{min}}\in\calC$.
Therefore, using the previous claims
\begin{eqnarray*}
\funccontparam{\alpha}{\mu_{\mathrm{min}}} & \le & \funccontparam{\alpha}{\mu_{t}}\\
 & = & B_{\alpha}\left(t\nu+\left(1-t\right)\mu_{\mathrm{min}}\right)-\logenerg{t\nu+\left(1-t\right)\mu_{\mathrm{min}}}\\
 & \le & B_{\alpha}\left(\mu_{\mathrm{min}}\right)+t\left[B_{\alpha}\left(\nu\right)-B_{\alpha}\left(\mu_{\mathrm{min}}\right)\right]\\
 &  & \quad\quad-\logenerg{\mu_{\mathrm{min}}}+2t\left[\logenerg{\mu_{\mathrm{min}}}-\int_{\bbc}\logpot{\nu}w\,\dd\mu_{\mathrm{min}}\left(w\right)\right]+O\left(t^{2}\right)\\
 & = & \funccontparam{\alpha}{\mu_{\mathrm{min}}}+t\left[B_{\alpha}\left(\nu\right)-B_{\alpha}\left(\mu_{\mathrm{min}}\right)+2\left\{ \logenerg{\mu_{\mathrm{min}}}-\int_{\bbc}\logpot{\nu}w\,\dd\mu_{\mathrm{min}}\left(w\right)\right\} \right]+O\left(t^{2}\right).
\end{eqnarray*}
Since $t>0$ can be arbitrarily small, this implies
\[
B_{\alpha}\left(\nu\right)-B_{\alpha}\left(\mu_{\mathrm{min}}\right)+2\left\{ \logenerg{\mu_{\mathrm{min}}}-\int_{\bbc}\logpot{\nu}w\,\dd\mu_{\mathrm{min}}\left(w\right)\right\} \ge0.
\]
In the other direction, assume
\[
\int_{\bbc}\logpot{\nu}w\,\dd\mu\left(w\right)-\frac{B_{\alpha}\left(\nu\right)}{2}>\int_{\bbc}\logpot{\mu}w\,\dd\mu\left(w\right)-\frac{B_{\alpha}\left(\mu\right)}{2},
\]
for some $\nu\in\calC$. Using the argument above, we can find another
measure $\mu_{t}$ such that $\funccontparam{\alpha}{\mu_{t}}<\funccontparam{\alpha}{\mu}$,
thus $\mu$ is not the minimizer.
\end{proof}
Let $\calC\subset\probmeas$ be a closed and convex set of measures,
and let $\optmeas$ be the measure that minimizes $\funccontparam{\alpha}{\nu}$
over all $\nu\in\calC$. We will need the following simple bound.
\begin{claim}
\label{claim:log_energy_low_bnd}For any $\nu\in\calC$ we have
\[
-\logenerg{\nu-\optmeas}\le\funccontparam{\alpha}{\nu}-\funccontparam{\alpha}{\optmeas}.
\]
\end{claim}
\begin{proof}
By the above lemma
\begin{eqnarray*}
\int_{\bbc}\logpot{\nu}w\,\dd\optmeas\left(w\right) & \le & \int_{\bbc}\logpot{\optmeas}w\,\dd\optmeas\left(w\right)+\frac{B_{\alpha}\left(\nu\right)}{2}-\frac{B_{\alpha}\left(\optmeas\right)}{2}\\
 & = & \logenerg{\optmeas}+\frac{B_{\alpha}\left(\nu\right)}{2}-\frac{B_{\alpha}\left(\optmeas\right)}{2},
\end{eqnarray*}
which implies
\begin{eqnarray*}
-\logenerg{\nu-\optmeas} & = & -\logenerg{\nu}+2\int_{\bbc}\logpot{\nu}w\,\dd\optmeas\left(w\right)-\logenerg{\optmeas}\\
 & \le & B_{\alpha}\left(\nu\right)-\logenerg{\nu}-\left[B_{\alpha}\left(\optmeas\right)-\logenerg{\optmeas}\right]=\funccontparam{\alpha}{\nu}-\funccontparam{\alpha}{\optmeas}.
\end{eqnarray*}
\end{proof}

\section{\label{sec:low_bnd_for_large_flucs}Probability of large fluctuations
in the number of zeros - Lower bound}

The goal of this section is to obtain the lower bound in Theorem \ref{thm:very_large_fluct},
that is, we are looking for a lower bound for the probability $\pr{n\left(r\right)=\left\lfloor pr^{2}\right\rfloor }$,
where $p\ge0$, $p\ne1$ is fixed. Let us write
\begin{equation}
k_{0}=k_{0}\left(r,p\right)=\left\lfloor pr^{2}\right\rfloor .\label{eq:k0_def}
\end{equation}
The main idea is to use Rouché's theorem. More precisely, we explicitly
construct an event where the term $\left|\xi_{k_{0}}\frac{z^{k_{0}}}{\sqrt{k_{0}!}}\right|$
in the Taylor series of the GEF dominates the sum over all the other
terms (on the circle $\left\{ \left|z\right|=r\right\} $). This simple
but effective method originally appeared in the paper \cite{sodin2005random},
and was later used in many other problems of this type.

\subsection{Outline of the proof}

We use the notation
\[
b_{k}=b_{k}\left(r\right)=\frac{r^{k}}{\sqrt{k!}},\quad k\in\bbn,
\]
and, using (\ref{eq:bounds_for_factorial}), we have the following
bounds
\begin{equation}
\frac{1}{2k^{\frac{1}{4}}}\left(\frac{er^{2}}{k}\right)^{\frac{k}{2}}\le b_{k}\le\left(\frac{er^{2}}{k}\right)^{\frac{k}{2}},\quad k\ge1.\label{eq:b_k_estimates}
\end{equation}
Let us consider the event $\left\{ n\left(r\right)=k_{0}\right\} $
(with $k_{0}$ given by (\ref{eq:k0_def})). Rouché's theorem implies
that
\[
E_{p}=E_{p}\left(r\right)\defeq\left\{ \left|\xi_{k_{0}}\right|b_{k_{0}}>\sum_{k\ne k_{0}}\left|\xi_{k}\right|b_{k}\right\} \subset\left\{ \left|\xi_{k_{0}}\right|\frac{r^{k_{0}}}{\sqrt{k_{0}!}}>\left|\sum_{k\ne k_{0}}\xi_{k}\frac{z^{k}}{\sqrt{k!}}\right|\right\} \subset\left\{ n\left(r\right)=k_{0}\right\} .
\]
We will construct an event that is contained in $E_{p}$, and thus
obtain a lower bound for the probability of the event $\left\{ n\left(r\right)=k_{0}\right\} $.
Depending on $p$, we define an interval $I_{p}\subset\bbr^{+}$.
We consider two main cases:
\begin{itemize}
\item[Case 1.]  $0\le p<e$.

\begin{itemize}
\item In this case we define $q=q\left(p\right)\ne p$ to be the non-trivial
solution of $q\left(\log q-1\right)=p\left(\log p-1\right)$.
\item We set $I_{p}=\left[p,q\right]$ in case $p<1$, and $I_{p}=\left[q,p\right]$
in case $p>1$.
\end{itemize}
\item[Case 2. ]  $p\ge e$.

\begin{itemize}
\item In this case we set $I_{p}=\left[0,p\right]$.
\end{itemize}
\end{itemize}
In general, our strategy is to `suppress' the terms $b_{k}$ for which
$\frac{k}{r^{2}}\in I_{p}$ (by choosing $\left|\xi_{k}\right|$ to
be small), except for the main term $b_{k_{0}}$. We will also assume
$\left|\xi_{k_{0}}\right|\ge1$, which happens with a constant probability.
By (\ref{eq:b_k_estimates}), this implies
\begin{align}
\left|\xi_{k_{0}}\right|b_{k_{0}} & \ge\exp\left(\frac{p}{2}\log\left(\frac{e}{p}\right)r^{2}-C\log\left(pr^{2}\right)\right)\nonumber \\
 & \ge\exp\left(\frac{p}{2}\log\left(\frac{e}{p}\right)r^{2}-C\log r\right),\label{eq:central_term_low_bnd}
\end{align}
for $r$ sufficiently large (or $\left|\xi_{k_{0}}\right|b_{k_{0}}\ge1$
in case $p=k_{0}=0$).

\subsubsection{Sketch of the proof in case $p<1$}

We now explain the idea of the proof in Case 1 (the proof of the other
case is similar). Using bounds for the factorial, we find
\begin{equation}
\frac{b_{k_{0}}}{b_{k}}=\exp\left(\frac{p}{2}\log\left(\frac{e}{p}\right)r^{2}-\frac{k}{2}\log\left(\frac{er^{2}}{k}\right)+\mbox{error terms}\right).\label{eq:idea_coeff_ratio}
\end{equation}
Put $k_{1}=\left\lfloor qr^{2}\right\rfloor +1$, where $q=q\left(p\right)$
as defined above. The estimate (\ref{eq:idea_coeff_ratio}) implies
that $\frac{b_{k_{0}}}{b_{k}}$ is small for $k$ not in the range
$\left\{ k_{0},\dots,k_{1}\right\} $. That is, the tail
\[
\left|\sum_{k\notin\left\{ k_{0},\dots,k_{1}\right\} }\xi_{k}\frac{z^{k}}{\sqrt{k!}}\right|\le\sum_{k\notin\left\{ k_{0},\dots,k_{1}\right\} }\left|\xi_{k}\right|b_{k},
\]
is small compared to $\left|\xi_{k_{0}}\right|b_{k_{0}}$, with sufficiently
large probability. To make the sum over $k\in\left\{ k_{0}+1,\dots,k_{1}\right\} $
small, we consider the event where a $\left|\xi_{k}\right|$ is at
most $\left(\frac{b_{k_{0}}}{b_{k}}\right)^{-1}$, for $k$ in this
range. The probability of this event is
\[
\exp\left(-2\cdot\sum_{k\in\left\{ k_{0}+1,\dots,k_{1}\right\} }\log\left(\frac{b_{k_{0}}}{b_{k}}\right)+\mbox{error terms}\right).
\]
We obtain the lower bound in Theorem \ref{thm:very_large_fluct},
after verifying
\[
2\cdot\sum_{k\in\left\{ k_{0}+1,\dots,k_{1}\right\} }\log\left(\frac{b_{k_{0}}}{b_{k}}\right)=-Z_{p}r^{4}+\mbox{error terms},
\]
where $Z_{p}=\int_{p}^{q}x\log x\,\dd x$.

\subsection{\label{subsec:main_terms_bounds} The main terms}

Consider now $p\ge0$, $p\ne1$. We define the set of main terms by
\begin{equation}
M=\setd{k\in\bbn}{\frac{k}{r^{2}}\in I_{p},\,k\ne k_{0}}.\label{eq:def_set_M}
\end{equation}
For $k\in\bbn^{+}$, let us define $A_{p,k}=\frac{b_{k_{0}}}{b_{k}}$.
From (\ref{eq:b_k_estimates}), we find the following bounds
\[
-C\log\left(k_{0}+1\right)\le\log A_{p,k}-\left[\frac{p}{2}\log\left(\frac{e}{p}\right)r^{2}-\frac{k}{2}\log\left(\frac{er^{2}}{k}\right)\right]\le C\log\left(k+1\right),
\]
thus, for $r$ sufficiently large,
\begin{equation}
\left|\log A_{p,k}-\left[\frac{p}{2}\log\left(\frac{e}{p}\right)r^{2}-\frac{k}{2}\log\left(\frac{er^{2}}{k}\right)\right]\right|\le C_{1}\log\left(k+1\right),\label{eq:ratio_to_main_term}
\end{equation}
for some numerical constant $C_{1}>0$. Notice that if $k=\alpha r^{2}$
for some $\alpha\ge0$, then
\[
\frac{p}{2}\log\left(\frac{e}{p}\right)r^{2}-\frac{k}{2}\log\left(\frac{er^{2}}{k}\right)=\left[p\log\left(\frac{e}{p}\right)-\alpha\log\left(\frac{e}{\alpha}\right)\right]\frac{r^{2}}{2}=\left[p\left(1-\log p\right)-\alpha\left(1-\log\alpha\right)\right]\frac{r^{2}}{2}.
\]
This means that $M$ contains the terms for which the expression above
is non-positive. Since $A_{p,k}\cdot\exp\left(-2C_{1}\log\left(k+1\right)\right)\le1$
for $k\in M$, we have

\begin{eqnarray*}
\pr{\left|\xi_{k}\right|\le\frac{1}{6r^{2}\cdot\left(k+1\right)^{2C_{1}}}\cdot A_{p,k}} & \ge & \frac{C}{r^{4}\left(k+1\right)^{4C_{1}}}A_{p,k}^{2}\\
 & \ge & \exp\left(p\log\left(\frac{e}{p}\right)r^{2}-k\log\left(\frac{er^{2}}{k}\right)-C\log\left(\left(p+1\right)r\right)\right)\\
 & \ge & \exp\left(p\log\left(\frac{e}{p}\right)r^{2}-k\log\left(\frac{er^{2}}{k}\right)-C\log r\right),
\end{eqnarray*}
for $r$ sufficiently large. We notice that in Case 1 there at most
$3r^{2}$ elements in $M$. We introduce the following event:
\[
E_{M}^{1}=\left\{ \left|\xi_{k}\right|\le\frac{1}{6r^{2}\left(k+1\right)^{2C_{1}}}\cdot A_{p,k},\quad\mbox{for all \ensuremath{k\in M}}\right\} .
\]
Clearly on the event $\left\{ \left|\xi_{k_{0}}\right|\ge1\right\} \cap E_{M}^{1}$,
we have
\[
\sum_{k\in M}\left|\xi_{k}\right|b_{k}\le\sum_{k\in M}\frac{A_{p,k}}{6r^{2}}\cdot b_{k}\le\frac{1}{2}b_{k_{0}}\le\frac{1}{2}\left|\xi_{k_{0}}\right|b_{k_{0}}.
\]
On the other hand,
\begin{eqnarray*}
\pr{E_{M}^{1}} & = & \prod_{k\in M}\pr{\left|\xi_{k}\right|\le\frac{1}{6r^{2}\cdot\left(k+1\right)^{2C_{1}}}\cdot A_{p,k}}\\
 & \ge & \exp\left(\sum_{k\in M}\left[p\log\left(\frac{e}{p}\right)r^{2}-k\log\left(\frac{er^{2}}{k}\right)-C\log r\right]\right)\\
 & \ge & \exp\left(\sum_{k\in M}\left[p\log\left(\frac{e}{p}\right)r^{2}-k\log\left(\frac{er^{2}}{k}\right)\right]+O\left(r^{2}\log r\right)\right).
\end{eqnarray*}
By Corollary \ref{cor:asymp_prob_main_terms}, Section \ref{subsec:low_bnd_finish},
we have
\[
\sum_{k\in M}\left[p\log\left(\frac{e}{p}\right)r^{2}-k\log\left(\frac{er^{2}}{k}\right)\right]=-Z_{p}\cdot r^{4}+O\left(r^{2}\log^{2}r\right).
\]
Case 2 is similar, and now there are at most $\left\lfloor 2pr^{2}\right\rfloor $
elements in $M$. Consider the event:
\[
E_{M}^{2}=\left\{ \left|\xi_{k}\right|\le\frac{1}{4pr^{2}\left(k+1\right)^{2C_{1}}}\cdot A_{p,k},\quad\mbox{for all \ensuremath{k\in M}}\right\} .
\]
On the event $\left\{ \left|\xi_{k_{0}}\right|\ge1\right\} \cap E_{M}^{2}$,
we have
\[
\sum_{k\in M}\left|\xi_{k}\right|b_{k}\le\sum_{k\in M}\frac{A_{p,k}}{4pr^{2}\left(k+1\right)^{2C_{1}}}\cdot b_{k}\le\frac{1}{2}b_{k_{0}}\le\frac{1}{2}\left|\xi_{k_{0}}\right|b_{k_{0}}.
\]
We also have, for $r$ sufficiently large,
\begin{eqnarray*}
\pr{E_{M}^{2}} & = & \prod_{k\in M}\pr{\left|\xi_{k}\right|\le\frac{1}{4pr^{2}\left(k+1\right)^{2C_{1}}}\cdot A_{p,k}}\\
 & \ge & \exp\left(\sum_{k\in M}\left[p\log\left(\frac{e}{p}\right)r^{2}-k\log\left(\frac{er^{2}}{k}\right)-C\log r\right]\right)\\
 & \ge & \exp\left(\sum_{k\in M}\left[p\log\left(\frac{e}{p}\right)r^{2}-k\log\left(\frac{er^{2}}{k}\right)\right]+O\left(r^{2}\log^{2}r\right)\right).
\end{eqnarray*}

\subsection{\label{subsec:tail_bounds} Tail bounds}

For technical reasons we consider two parts of the `tail' separately,
the far tail and the close tail.

\subsubsection{The far tail}

Let $r>0$ be sufficiently large, $\alpha>10$, and $N=\left\lfloor \alpha r^{2}\right\rfloor +1$.
Recall the tail of the GEF is given by 
\[
\seriestail\left(z\right)=\sum_{k=N+1}^{\infty}\xi_{k}\frac{z^{k}}{\sqrt{k!}},\quad z\in\bbc.
\]
By Lemma \ref{lem:tail_bound} we have, outside an exceptional event
$E_{T}$ of probability at most $\exp\left(-Cr^{6}\right)$,
\[
\left|\seriestail\left(z\right)\right|\le\exp\left(\frac{N}{2}\log\left(\frac{4}{\alpha}\right)\right),\quad\left|z\right|\le r.
\]
We will take
\[
\alpha=\alpha\left(p\right)=\begin{cases}
16 & \,0\le p\le11;\\
5+p & \,p>11.
\end{cases}
\]
If $p\le11$, then $\frac{p}{2}\log\left(\frac{e}{p}\right)>-8$,
and therefore by (\ref{eq:central_term_low_bnd}) we have $\left|\seriestail\left(z\right)\right|\le\exp\left(-8r^{2}\right)<\frac{1}{4}\left|\xi_{k_{0}}\right|b_{k_{0}}$,
when $r$ is sufficiently large. Similarly for $p>11$, we have $\frac{p}{2}\log\left(\frac{e}{p}\right)>\frac{5+p}{2}\log\left(\frac{4}{5+p}\right)$,
and
\[
\left|\seriestail\left(z\right)\right|\le\exp\left(\frac{5+p}{2}\log\left(\frac{4}{5+p}\right)\cdot r^{2}\right)<\frac{1}{4}\left|\xi_{k_{0}}\right|b_{k_{0}}.
\]

\subsubsection{The close tail}

We now consider all the terms such that $\frac{k}{r^{2}}\notin I_{p}$,
but $k\le N$. For these terms we have $\frac{k}{2}\log\left(\frac{er^{2}}{k}\right)\le\frac{p}{2}\log\left(\frac{e}{p}\right)r^{2}$,
and the lower bound of (\ref{eq:ratio_to_main_term}) implies $b_{k}\le b_{k_{0}}\left(k+1\right)^{C_{1}}$
. In Case 1, we notice there are at most $\left\lfloor 17r^{2}\right\rfloor $
elements in the close tail, let us denote their indices by $M^{\prime}$.
Introduce the following event:
\[
E_{M^{\prime}}^{1}=\left\{ \left|\xi_{k}\right|\le\frac{1}{70r^{2}\left(k+1\right)^{C_{1}}},\quad\mbox{for all \ensuremath{k\in M^{\prime}}}\right\} .
\]
On the event $\left\{ \left|\xi_{k_{0}}\right|\ge1\right\} \cap E_{M^{\prime}}^{1}$
we have
\[
\sum_{k\in M^{\prime}}\left|\xi_{k}\right|b_{k}\le\sum_{k\in M^{\prime}}\frac{1}{70r^{2}\left(k+1\right)^{C_{1}}}\cdot b_{k}<\frac{1}{4}b_{k_{0}}\le\frac{1}{4}\left|\xi_{k_{0}}\right|b_{k_{0}}.
\]
In addition,
\[
\pr{E_{M^{\prime}}^{1}}\ge\prod_{k=0}^{\left\lfloor 17r^{2}\right\rfloor }\frac{1}{2}\cdot\left(\frac{1}{70r^{2}\left(k+1\right)^{C_{1}}}\right)^{2}\ge\left(Cr^{4+2C_{1}}\right)^{-Cr^{2}}\ge\exp\left(-Cr^{2}\log r\right).
\]
Similarly, in Case 2, there are at most $\left\lfloor 6r^{2}\right\rfloor $
elements in the close tail. Let us again denote their indices by $M^{\prime}$.
Consider the event:
\[
E_{M^{\prime}}^{2}=\left\{ \left|\xi_{k}\right|\le\frac{1}{24r^{2}\left(k+1\right)^{C_{1}}},\quad\mbox{for all \ensuremath{k\in M^{\prime}}}\right\} .
\]
On the event $\left\{ \left|\xi_{k_{0}}\right|\ge1\right\} \cap E_{M^{\prime}}^{2}$
we have
\[
\sum_{k\in M^{\prime}}\left|\xi_{k}\right|b_{k}\le\sum_{k\in M^{\prime}}\frac{1}{24r^{2}\left(k+1\right)^{C_{1}}}\cdot b_{k}<\frac{1}{4}b_{k_{0}}\le\frac{1}{4}\left|\xi_{k_{0}}\right|b_{k_{0}}.
\]
Now, for $k\in M^{\prime}$ we have $k\le r^{3}$ (for $r$ sufficiently
large), and therefore
\begin{eqnarray*}
\pr{E_{M^{\prime}}^{2}} & \ge & \prod_{k\in M^{\prime}}\frac{1}{2}\cdot\left(\frac{1}{24r^{2}\left(k+1\right)^{C_{1}}}\right)^{2}\ge\prod_{k=0}^{\left\lfloor 6r^{2}\right\rfloor }Cr^{-2\left(2+3C_{1}\right)}\\
 & \ge & \exp\left(-Cr^{2}\log r\right).
\end{eqnarray*}

\subsection{\label{subsec:low_bnd_finish}Combining the estimates and finishing
the proof}

We start with a simple computation. Recall that
\[
Z_{p}=\left|\int_{p}^{q\left(p\right)}x\log x\,\dd x\right|,
\]
and consider the function
\[
l\left(x\right)=p\log\left(\frac{e}{p}\right)-x\log\left(\frac{e}{x}\right),\quad x\ge0.
\]

\begin{claim}
With $q\left(p\right)$ and $Z_{p}$ defined as before, we have
\[
Z_{p}=\left|\int_{p}^{q\left(p\right)}l\left(x\right)\,\dd x\right|.
\]
\end{claim}
\begin{proof}
Let us consider Case 1, for $p<1$. Set $t\left(x\right)=x\log\left(\frac{e}{x}\right)$
and $q=q\left(p\right)$. We have
\begin{align*}
\int_{p}^{q}\left[p\log\left(\frac{e}{p}\right)-x\log\left(\frac{e}{x}\right)\right]\,\dd x & =\left(q-p\right)p\log\left(\frac{e}{p}\right)-\int_{p}^{q}t\left(x\right)\,\dd x\\
 & =\left(q-p\right)p\log\left(\frac{e}{p}\right)-\left.xt\left(x\right)\right|_{x=p}^{q}+\int_{p}^{q}xt^{\prime}\left(x\right)\,\dd x\\
 & =\left(q-p\right)p\log\left(\frac{e}{p}\right)-qt\left(q\right)+pt\left(p\right)-\int_{p}^{q}x\log x\,\dd x\\
 & =q\left(p\log\left(\frac{e}{p}\right)-q\log\left(\frac{e}{q}\right)\right)-\int_{p}^{q}x\log x\,\dd x\\
 & =-\int_{p}^{q}x\log x\,\dd x,
\end{align*}
where in the last line we used the definition of $q=q\left(p\right)$.
The other cases are proved in a similar way.
\end{proof}
By the definition of the set $M$ (see (\ref{eq:def_set_M})), we
get
\begin{cor}
\label{cor:asymp_prob_main_terms} For $r$ sufficiently large
\[
\sum_{k\in M}\left[p\log\left(\frac{e}{p}\right)r^{2}-k\log\left(\frac{er^{2}}{k}\right)\right]=-Z_{p}\cdot r^{4}+O\left(C\left(p\right)r^{2}\right),
\]
where $C\left(p\right)=\max\left\{ p\log\left(\frac{p}{e}\right),1\right\} $.
\end{cor}
\begin{proof}
The function $l\left(x\right)$ has a single minimum at $x=1$, thus
\[
\sum_{k\in M}\left[p\log\left(\frac{e}{p}\right)r^{2}-k\log\left(\frac{er^{2}}{k}\right)\right]=r^{2}\sum_{k\in M}l\left(\frac{k}{r^{2}}\right)=-\left|\int_{p}^{q\left(p\right)}l\left(x\right)\,\dd x\right|r^{4}+O\left(C\left(p\right)r^{2}\right),
\]
with $C\left(p\right)$ as above (in Case 2 we take $q\left(p\right)=0$).
\end{proof}
Finally, notice that the events $\left\{ \left|\xi_{k_{0}}\right|\ge1\right\} $,
$E_{T}$, $E_{M}^{j}$, and $E_{M^{\prime}}^{j}$ are all independent
($j\in\left\{ 1,2\right\} $). Therefore, combining our estimates
from Section \ref{subsec:main_terms_bounds} and Section \ref{subsec:tail_bounds},
we find that the probability of the event $\left\{ \left|\xi_{k_{0}}b_{k_{0}}\right|>\sum_{k\ne k_{0}}\left|\xi_{k}\right|b_{k}\right\} $
is at least
\[
\exp\left(\sum_{k\in M}\left[p\log\left(\frac{e}{p}\right)r^{2}-k\log\left(\frac{er^{2}}{k}\right)\right]+O\left(r^{2}\log^{2}r\right)\right).
\]
Hence, by \ref{cor:asymp_prob_main_terms}, for $r$ sufficiently
large we have the bound,
\[
\pr{n\left(r\right)=\left\lfloor pr^{2}\right\rfloor }\ge\exp\left(-Z_{p}r^{4}+O\left(r^{2}\log^{2}r\right)\right),
\]
thus proving the lower bound in Theorem \ref{thm:very_large_fluct}.

\section{\label{sec:conv_emp_meas} The conditional distribution of the zeros}

In this section we describe the distribution of the zeros of the GEF,
conditioned on a prescribed number of zeros inside the the disk $\left\{ \left|z\right|\le r\right\} $
(recall that we denote this number by $n\left(r\right)$). We will
again consider two main cases:
\begin{itemize}
\item[Case 1. ]  Deficiency in the number of zeros: $F\left(p;r\right)=\left\{ n\left(r\right)\le pr^{2}\right\} $,
where $p\in\left[0,1\right)$.
\begin{itemize}
\item In particular, this implies the case $p=0$ of Theorem \ref{thm:conv_emp_meas_for_hole}.
\end{itemize}
\item[Case 2. ]  Overcrowding of zeros: $M\left(p;r\right)=\left\{ n\left(r\right)\ge pr^{2}\right\} $,
where $p>1$.

\begin{itemize}
\item We consider separately the range $p\in\left(1,e\right)$, and the
range $p\in\left[e,\infty\right)$.
\end{itemize}
\end{itemize}
For each of the cases we define the limiting conditional distribution,
by the following Radon measures:
\begin{align*}
\dd\minmeasglob p\left(z\right) & =\begin{cases}
{\displaystyle \left[\left\{ \ind{\left\{ 0\le\left|w\right|\le\sqrt{p}\right\} }z+\ind{\left\{ \sqrt{q}\le\left|w\right|\right\} }z\right\} \right]\cdot\frac{\dd\lebmeas\left(z\right)}{\pi}+\left(q-p\right)\dd\subslebmeas{\left|z\right|=1}} & \,p\in\left[0,1\right);\\
{\displaystyle \left[\ind{\left\{ 0\le\left|w\right|\le\sqrt{q}\right\} }z+\ind{\left\{ \sqrt{p}\le\left|w\right|\right\} }z\right]\cdot\frac{\dd\lebmeas\left(z\right)}{\pi}+\left(p-q\right)\dd\subslebmeas{\left|z\right|=1}} & \,p\in\left(1,e\right);\\
\ind{\left\{ \sqrt{p}\le\left|w\right|\right\} }z\cdot\frac{\dd\lebmeas\left(z\right)}{\pi}+p\cdot\dd\subslebmeas{\left|z\right|=1} & \,p\ge e.
\end{cases}
\end{align*}
Here $\subslebmeas{\left|z\right|=1}$ is Lebesgue measure on the
circle $\left|z\right|=1$, normalized to be a probability measure
and $\ind Az$ is the indicator of the set $A$. We recall that $q=q\left(p\right)$
was defined before the statement of Theorem \ref{thm:very_large_fluct}.

Suppose now $\varphi\in C_{0}^{2}\left(\bbc\right)$ is a test function,
which is twice continuously differentiable and with compact support,
and recall
\[
\Dnorm{\varphi}=\left\Vert \nabla\varphi\right\Vert _{L^{2}\left(\lebmeas\right)}^{2}=\int_{\bbc}\left(\varphi_{x}^{2}+\varphi_{y}^{2}\right)\,\dd m\left(z\right).
\]
Let $\linstats{\gef}{\varphi}r$ be the linear statistics associated
with $\varphi$, recall that $\ex{\linstats{\gef}{\varphi}r}=r^{2}\cdot\frac{1}{\pi}\int_{\bbc}\varphi\left(w\right)\,\dd\lebmeas\left(w\right)$,
and consider the event
\begin{eqnarray*}
L\left(p,\varphi,\lambda\right)=L\left(p,\varphi,\lambda;r\right) & = & \left\{ \left|\linstats{\gef}{\varphi}r-r^{2}\int_{\bbc}\varphi\left(w\right)\,\dd\minmeasglob p\left(w\right)\right|\ge\lambda\right\} .
\end{eqnarray*}
The following theorem shows that conditioned on the event $F\left(p;r\right)$
(or $M\left(p;r\right)$), the value of $\linstats{\gef}{\varphi}r$
is unlikely to be far from $r^{2}\int_{\bbc}\varphi\left(w\right)\,\dd\minmeasglob p\left(w\right)$.
Recall that by $\condex{\cdot}{F\left(p;r\right)}$ (resp. $\condpr{\cdot}{F\left(p;r\right)}$)
we denote the conditional expectation (resp. probability) on the event
$F\left(p;r\right)$.
\begin{thm}
\label{thm:devs_in_emp_meas}Suppose $C^{\prime}>0$ is fixed. For
$\lambda\in\left(0,C^{\prime}r^{2}\right)$ and $r$ sufficiently
large we have, 
\[
\condpr{L\left(p,\varphi,\lambda;r\right)}{F\left(p;r\right)},\,\condpr{L\left(p,\varphi,\lambda;r\right)}{M\left(p;r\right)}\le\exp\left(-\frac{C_{p}}{\Dnorm{\varphi}}\cdot\lambda^{2}+C_{\varphi}r^{2}\log^{2}r\right),
\]
where $C_{\varphi}>0$ is some constant that depends on $\omega\left(\varphi,t\right)$
- the modulus of continuity of the function $\varphi$, and $C_{p}>0$
is a constant depending only on $p$ (and which can be replaced by
an absolute constant for $p\le e$).
\end{thm}
\begin{rem}
It will be clear from the proof, that we can take the test function
$\varphi$ depending on $r$, such that its modulus of continuity
satisfies $\omega\left(\varphi,t\right)=O\left(r^{C_{3}}t^{C_{4}}\right)$,
for some numerical constants $C_{3},C_{4}>0$. In that case, the constant
$C_{\varphi}$ will depend only on $C_{3}$ and $C_{4}$.
\end{rem}
This theorem implies the convergence in distribution of the zero counting
measure, conditioned on the event $F\left(p\right)$ (or $M\left(p\right)$).
We denote by $\condzeroset r^{p}$ the zero set of $\gef$ conditioned
on the occurrence of the event $F\left(p;r\right)$, and write $\cntmeas{\condzeroset r^{p}}$
for the corresponding counting measure (similar definitions can be
made for the event $M\left(p\right)$).
\begin{thm}
\label{thm:conv_emp_meas} Let $\varphi\in C_{0}^{2}\left(\bbc\right)$
be a fixed test function. As $r\to\infty$,
\[
\condex{\linstats{\gef}{\varphi}r}{F\left(p;r\right)},\,\condex{\linstats{\gef}{\varphi}r}{M\left(p;r\right)}=r^{2}\int_{\bbc}\varphi\left(w\right)\,\dd\minmeasglob p\left(w\right)+O\left(r\log^{2}r\right).
\]
In addition, as $r\to\infty$, the scaled zero counting measure $\frac{1}{r^{2}}\cntmeas{\condzeroset r^{p}}\left(\frac{\cdot}{r}\right)\to\minmeasglob p$
in distribution, where the convergence is in the vague topology. That
is, for any continuous test function $\phi$ with compact support,
we have
\[
\frac{1}{r^{2}}\int_{\bbc}\phi\,\dd\cntmeas{\condzeroset r^{p}}\left(\frac{\cdot}{r}\right)=\frac{1}{r^{2}}\sum_{z\in\condzeroset r^{p}}\phi\left(\frac{z}{r}\right)\xrightarrow[r\to\infty]{d}\int_{\bbc}\phi\left(w\right)\,\dd\minmeasglob p\left(w\right).
\]
An analogous result holds for the event $M\left(p;r\right)$.
\end{thm}

\subsection{Preliminaries}

Notice that for any $p\ge0$, if $\alpha$ is sufficiently large,
then $\minmeasglob p\left(\disc 0{\sqrt{\alpha}}\right)=\alpha$.
We are going to work with the following probability measures, which
are the normalized truncations of $\minmeasglob p$:
\[
\dd\minmeassup p{\alpha}\left(z\right)=\begin{cases}
{\displaystyle \frac{1}{\alpha}\left[\ind{\left\{ \left|z\right|\le\sqrt{p}\right\} }z+\ind{\left\{ \sqrt{q}\le\left|z\right|\le\sqrt{\alpha}\right\} }z\right]\cdot\frac{\dd\lebmeas\left(z\right)}{\pi}+\frac{q-p}{\alpha}\dd\subslebmeas{\left|z\right|=1}} & \,p\in\left[0,1\right);\\
{\displaystyle \frac{1}{\alpha}\left[\ind{\left\{ \left|z\right|\le\sqrt{q}\right\} }z+\ind{\left\{ \sqrt{p}\le\left|z\right|\le\sqrt{\alpha}\right\} }z\right]\cdot\frac{\dd\lebmeas\left(z\right)}{\pi}+\frac{p-q}{\alpha}\dd\subslebmeas{\left|z\right|=1}} & \,p\in\left(1,e\right);\\
\frac{1}{\alpha}\ind{\left\{ \sqrt{p}\le\left|z\right|\le\sqrt{\alpha}\right\} }z\cdot\frac{\dd\lebmeas\left(z\right)}{\pi}+\frac{p}{\alpha}\dd\subslebmeas{\left|z\right|=1} & \,p\in\left[e,\alpha\right).
\end{cases}
\]
In Section \ref{sec:sol_of_energy_prob}, we showed that these measures
are the minimizers of the functional $\funccont{\nu}=\funccontparam{\alpha}{\nu}$
over the sets $\calF_{p}=\setd{\nu\in\probmeas}{\nu\left(D\right)\le\frac{p}{\alpha}}$,
where $p<1$, and $\calM_{p}=\setd{\nu\in\probmeas}{\nu\left(\overline{D}\right)\ge\frac{p}{\alpha}}$,
where $p>1$.

We introduce the following notation
\[
\calL_{\varphi,\tau,\lambda}\defeq\setd{\nu\in\probmeas}{\left|\int_{\bbc}\varphi\left(w\right)\,\dd\nu\left(w\right)-\tau\right|\ge\lambda}.
\]
The key tool that we will use is the next claim, which can be seen
as an effective form of the fact that $\funccont{\nu}$ is strictly
convex. It shows that if a measure $\nu\in\calF_{p}$ (or $\calM_{p}$)
is far from the minimizer $\minmeassup p{\alpha}$ (with respect to
the test function $\varphi$), then $\funccont{\nu}$ is relatively
large.
\begin{claim}
\label{claim:low_bnd_func_for_far_from_min_meas}Let $\tau=\int_{\bbc}\varphi\left(w\right)\,\dd\minmeassup p{\alpha}\left(w\right)$.
For any compactly supported measure $\nu\in\calF_{p}\cap\calL_{\varphi,\tau,\lambda}$,
we have
\[
I\left(\nu\right)-I\left(\minmeassup p{\alpha}\right)\ge\frac{2\pi}{\Dnorm{\varphi}}\cdot\lambda^{2}.
\]
The same result holds if we replace $\calF_{p}$ by $\calM_{p}$.
\end{claim}
\begin{proof}
In Section \ref{sec:sol_of_energy_prob}, we prove that the measure
$\minmeassup p{\alpha}$ minimizes $\funccont{\nu}$ over the set
$\calF_{p}$. Now, combining Lemma \ref{lem:lin_stats_for_meas} and
Claim \ref{claim:log_energy_low_bnd} we have for $\nu\in\calF_{p}$
\begin{eqnarray*}
\lambda & \le & \left|\int_{\bbc}\varphi\left(w\right)\,\dd\nu\left(w\right)-\int_{\bbc}\varphi\left(w\right)\,\dd\minmeassup p{\alpha}\left(w\right)\right|\le\frac{1}{\sqrt{2\pi}}\sqrt{\Dnorm{\varphi}}\sqrt{-\logenerg{\nu-\minmeassup p{\alpha}}},\\
 & \le & \frac{1}{\sqrt{2\pi}}\sqrt{\Dnorm{\varphi}}\sqrt{I\left(\nu\right)-I\left(\minmeassup p{\alpha}\right)}.
\end{eqnarray*}
The same proof applies for $\calM_{p}$ as well.
\end{proof}

\subsection{Truncation of the power series and estimates for the joint distribution
of the zeros}

We start by recalling some of the results we proved in Section \ref{sec:upp_bnd_for_larg_flucs}.
Let $r>0$ be sufficiently large. Put $\alpha=Nr^{-2}$, $\lambda=\log r$,
$t=\gamma=r^{-C_{2}}$, with $C_{2}\ge4$, and $N_{0}=\left\lfloor \lambda r^{2}\right\rfloor +1$,
$N_{1}=\left\lfloor 2\lambda r^{2}\right\rfloor +1$. Let $\varphi$
be a test function supported on the disk $\disc 0B$, with fixed $B\ge1$.
We found that there exist events $\regevent$ and $\kregevent N$,
$N\in\left\{ N_{0},\dots,N_{1}\right\} $, such that
\[
\regevent=\biguplus_{N=N_{0}}^{N_{1}}\kregevent N,\quad\kregevent c\mbox{ is negligible.}
\]
If we introduce the scaled polynomial
\[
\truncscaledpoly\left(z\right)=\sum_{k=0}^{N}\xi_{k}\frac{\left(Lz\right)^{k}}{\sqrt{k!}},\quad z\in\bbc,
\]
then, on the event $\kregevent N$, we have

\begin{eqnarray}
n_{\truncscaledpoly}\left(\frac{r-K_{0}}{L}\right) & \le & n\left(r\right)\,\,\le n_{\truncscaledpoly}\left(\frac{r+K_{0}}{L}\right),\label{eq:zero_count_poly_L_lin_stats}\\
\left|n_{\gef}\left(\varphi;r\right)-n_{\truncscaledpoly}\left(\varphi;\frac{r}{L}\right)\right| & \le & CM_{0}\cdot\omega\left(\varphi;K_{0}r^{-1}\right),\label{eq:lin_stats_poly_L_lin_stats}
\end{eqnarray}
where $M_{0}=8B^{3}r^{3}$, and $K_{0}=2M_{0}\gamma\le Cr^{3-C_{2}}=O\left(\frac{1}{r}\right)$.
\begin{rem}
\label{rem:bnd_for_reg_event_comp} To find the event $\regevent$
we applied Lemma \ref{lem:reg_event} and Lemma \ref{lem:reg_event_conclusion}.
We may choose the parameter $A$ in Lemma \ref{lem:reg_event} to
be arbitrarily large (but fixed). For $r$ sufficiently large this
gives,
\[
\pr{\kregevent c}\le\exp\left(-C_{A}r^{4}\right),
\]
where $C_{A}$ is a constant depending on $A$ (and $B$), such that
$C_{A}\to\infty$ as $A\to\infty$.
\end{rem}
We choose the parameter $L$ as follows:
\begin{equation}
L=L\left(r;B\right)=\begin{cases}
\left(1+t\right)^{-1}\left(r-K_{0}\right) & \,\mbox{in Case 1};\\
\left(1-t\right)^{-1}\left(r+K_{0}\right) & \,\mbox{in Case 2}.
\end{cases}\label{eq:def_param_L}
\end{equation}
Now, for $\tau\in\bbr$, $\eta\ge0$, define the following set
\[
L_{\varphi,\tau,\eta}^{N}=L_{\varphi,\tau,\eta}^{N}\left(t\right)=\setd{\vec z}{\left|\frac{1}{N}\sum_{j=1}^{N}\varphi\left(z_{j}\right)-\tau\right|\ge\eta+\omega\left(\varphi;t\right)}.
\]
Put $Z_{N}=\left\{ n_{\truncscaledpoly}\left(1+t\right)\le pr^{2}\right\} \cap L_{\varphi,\tau,\eta}^{N}\subset\bbc^{N}$.
We showed in Section \ref{subsec:reduc_mod_energ_problem} that
\begin{equation}
\setd{\mu_{\vec z}^{t}}{\vec z\in Z_{N}}\subset\setd{\mu_{\vec z}^{t}}{\mu_{\vec z}^{t}\left(D\right)\le\frac{p}{\alpha}\mbox{ and }\vec z\in L_{\varphi,\tau,\lambda}^{N}}\subset\calF_{p}\cap\calL_{\varphi,\tau,\eta}\defeq\calZ,\label{eq:inclusion_for_smoothed_meas}
\end{equation}
where we used (\ref{eq:inclusion_of_meas}). The bound (\ref{eq:prob_upp_bnd_est})
gives
\begin{multline}
\log\pr{\left\{ n_{\truncscaledpoly}\left(1+t\right)\le pr^{2}\right\} \cap L_{\varphi,\tau,\eta}^{N}\cap\kregevent N}\\
\quad\quad\le-N^{2}\left[\inf_{\nu\in\calZ}\funccontparam{\alpha}{\nu}-\frac{1}{2}\log\left(\frac{N}{L^{2}}\right)+\frac{3}{4}\right]+L^{2}\log L\cdot O\left(\log L+\log\frac{1}{t}+tL^{2}\sqrt{\log L}\right)\\
=-N^{2}\left[\inf_{\nu\in\calZ}\funccontparam{\alpha}{\nu}-\frac{1}{2}\log\alpha+\frac{3}{4}\right]+O\left(r^{2}\log^{2}r\right).\quad\quad\quad\quad\quad\quad\quad\quad\quad\quad\quad\quad\quad\quad\quad\quad\,\,\label{eq:bnd_prob_fluct_and_lin_stats}
\end{multline}
We can bound the probability of the event $\pr{\left\{ n_{\truncscaledpoly}\left(1+t\right)\ge pr^{2}\right\} \cap L_{\varphi,\tau,\eta}^{N}\cap\kregevent N}$
in a similar way. 

\subsection{\label{subsec:large_fluct_in_lin_stats}Large fluctuations in the
number of zeros and linear statistics}

Let $\kappa\ge0$ be sufficiently large (depending on $r$). Recall
the event
\[
F\left(p\right)\cap L\left(p,\varphi,\kappa\right)=F\left(p;r\right)\cap L\left(p,\varphi,\kappa;r\right)=\left\{ n_{\gef}\left(r\right)\le pr^{2}\right\} \cap\left\{ \left|\linstats{\gef}{\varphi}r-r^{2}\int_{\bbc}\varphi\left(w\right)\,\dd\minmeasglob p\left(w\right)\right|\ge\kappa\right\} .
\]
By the definition of the measures $\minmeassup p{\alpha}$ and $\minmeasglob p$,
we have for $\alpha$ sufficiently large (depending on $p$ and $B$)
\[
r^{2}\int_{\bbc}\varphi\left(w\right)\,\dd\minmeasglob p\left(w\right)=r^{2}\cdot\alpha\int_{\bbc}\varphi\left(w\right)\,\dd\minmeassup p{\alpha}\left(w\right)=N\cdot\int_{\bbc}\varphi\left(w\right)\,\dd\minmeassup p{\alpha}\left(w\right).
\]
In addition, using (\ref{eq:lin_stats_poly_L_lin_stats}), we get
\[
\sum_{z\in\zeroset{\truncscaledpoly}}\varphi\left(z\cdot\left(\frac{r}{L}\right)^{-1}\right)=n_{\truncscaledpoly}\left(\varphi;\frac{r}{L}\right)=n_{\gef}\left(\varphi;r\right)+O\left(M_{0}\cdot\omega\left(\varphi;K_{0}r^{-1}\right)\right),
\]
where $\zeroset{\truncscaledpoly}=\left\{ z_{1},\dots,z_{N}\right\} $
is the zero set of the polynomial $\truncscaledpoly$. Since $\frac{r}{L}=1+O\left(r^{2-C_{2}}\right)$
by (\ref{eq:def_param_L}), and because $\varphi$ is supported inside
$\disc 0B$, we obtain
\begin{eqnarray*}
\sum_{j=1}^{N}\varphi\left(z_{j}\right) & = & \sum_{j=1}^{N}\varphi\left(z_{j}\cdot\left(\frac{r}{L}\right)^{-1}\right)+O\left(N\cdot B\cdot r^{2-C_{2}}\right)\\
 & = & n_{\gef}\left(\varphi;r\right)+O_{B}\left(r^{4-C_{2}}\log r+r^{3}\cdot\omega\left(\varphi;Cr^{2-C_{2}}\right)\right)\\
 & = & n_{\gef}\left(\varphi;r\right)+O_{B}\left(N\cdot E_{1}\left(\varphi;r\right)\right),
\end{eqnarray*}
where $N\cdot E_{1}\left(\varphi;r\right)=r^{3}\left(r^{2-C_{2}}+\omega\left(\varphi;Cr^{2-C_{2}}\right)\right)$,
using $N=O\left(r^{2}\log r\right)$. 
\begin{rem}
\label{rem:mod_of_cont_test_func} If the test function $\varphi$
depends on $r$ in such a way that its modulus of continuity satisfies
$\omega\left(\varphi,t\right)=O\left(r^{C_{3}}t^{C_{4}}\right)$,
for some numerical constants $C_{3},C_{4}>0$, then we can choose
$C_{2}$ sufficiently large to make $C_{B}\cdot N\cdot E_{1}\left(\varphi;r\right)\le C$,
for $r$ sufficiently large.
\end{rem}
We conclude that on the event $L\left(p,\varphi,\kappa;r\right)\cap\kregevent N$,
we have
\begin{equation}
\left|\frac{1}{N}\sum_{j=1}^{N}\varphi\left(z_{j}\right)-\int_{\bbc}\varphi\left(w\right)\,\dd\minmeassup p{\alpha}\left(w\right)\right|\ge\frac{\kappa}{N}-C_{B}\cdot E_{1}\left(\varphi;r\right)\defeq\frac{\kappa_{1}}{N},\label{eq:diff_emp_meas_and_min_meas}
\end{equation}
with some constant $C_{B}\ge1$, depending only on $B$. Let $\tau=\int_{\bbc}\varphi\left(w\right)\,\dd\minmeassup p{\alpha}\left(w\right)$
and assume that $\kappa$ is sufficiently large so that $\kappa_{1}>0$.
Claim \ref{claim:low_bnd_func_for_far_from_min_meas} (with $\lambda=\frac{\kappa_{1}}{N})$
implies that
\[
\funccontparam{\alpha}{\nu}\ge\funccontparam{\alpha}{\minmeassup p{\alpha}}+\frac{C}{\Dnorm{\varphi}}\cdot\left(\frac{\kappa_{1}}{N}\right)^{2},\quad\nu\in\calZ.
\]
By (\ref{eq:bnd_prob_fluct_and_lin_stats}), (\ref{eq:diff_emp_meas_and_min_meas}),
we have
\begin{eqnarray}
\log\pr{F\left(p;r\right)\cap L\left(p,\varphi,\kappa;r\right)\cap\kregevent N} & \le & \log\pr{\left\{ n_{\truncscaledpoly}\left(1+t\right)\le pr^{2}\right\} \cap L_{\varphi,\tau,\kappa_{1}}^{N}\cap\kregevent N}\nonumber \\
 & \le & -N^{2}\left[\funccontparam{\alpha}{\minmeassup p{\alpha}}+\frac{C}{\Dnorm{\varphi}}\cdot\left(\frac{\kappa_{1}}{N}\right)^{2}-\frac{1}{2}\log\alpha+\frac{3}{4}\right]+O\left(r^{2}\log^{2}r\right)\nonumber \\
 & \le & \log\pr{F\left(p;r\right)}-\frac{C}{\Dnorm{\varphi}}\cdot\kappa_{1}^{2}+O\left(r^{2}\log^{2}r\right),\label{eq:upp_bnd_lin_stat_dev}
\end{eqnarray}
where we used the bound $\frac{\kappa}{N}\ge\frac{\kappa_{1}}{N}+\omega\left(\varphi;t\right)$
in the first inequality (since $\omega\left(\varphi;t\right)\le E_{1}\left(\varphi;r\right)$),
and the bound 
\[
-N^{2}\left[\funccontparam{\alpha}{\minmeassup p{\alpha}}-\frac{1}{2}\log\alpha+\frac{3}{4}\right]=-Z_{p}r^{4}\le\log\pr{F\left(p;r\right)}+O\left(r^{2}\log^{2}r\right),
\]
from Section \ref{sec:low_bnd_for_large_flucs}, in the third inequality.

\subsubsection{Finishing the proof of Theorem \ref{thm:devs_in_emp_meas}}

Rewriting (\ref{eq:upp_bnd_lin_stat_dev}) we find that
\[
\pr{F\left(p\right)\cap L\left(p,\varphi,\kappa\right)\cap\kregevent N}\le\pr{F\left(p\right)}\exp\left(-\frac{C}{\Dnorm{\varphi}}\cdot\kappa_{1}^{2}+O\left(r^{2}\log^{2}r\right)\right),
\]
with $\kappa_{1}=\kappa-C_{B}N\cdot E_{1}\left(\varphi;r\right)$.
By Remark (\ref{rem:mod_of_cont_test_func}), if $\omega\left(\varphi,t\right)=O\left(r^{C_{3}}t^{C_{4}}\right)$
with some constants $C_{3},C_{4}>0$, we can choose $C_{2}$ sufficiently
large, so that $\kappa_{1}\ge\kappa-C$, if $r$ is sufficiently large.
We thus have,
\begin{eqnarray*}
\pr{F\left(p\right)\cap L\left(p,\varphi,\kappa\right)} & \le & \pr{\kregevent c}+\sum_{N=N_{0}}^{N_{1}}\pr{F\left(p\right)\cap L\left(p,\varphi,\kappa\right)\cap\kregevent N}\\
 & \le & \pr{\kregevent c}+\left(N_{1}-N_{0}+1\right)\pr{F\left(p\right)}\exp\left(-\frac{C}{\Dnorm{\varphi}}\cdot\kappa_{1}^{2}+O\left(r^{2}\log^{2}r\right)\right)\\
 & \le & \pr{\kregevent c}+\pr{F\left(p\right)}\exp\left(-\frac{C}{\Dnorm{\varphi}}\cdot\kappa^{2}+\frac{C\kappa}{\Dnorm{\varphi}}+O\left(r^{2}\log^{2}r\right)\right).
\end{eqnarray*}
We can assume w.l.o.g. that $\sup\setd{\varphi\left(w\right)}{w\in\bbc}=1$,
and that $\Dnorm{\varphi}>c$ for some constant $c=c\left(B\right)>0$
(since $\varphi$ is supported on $\disc 0B$). In addition, by Remark
\ref{rem:bnd_for_reg_event_comp}, we can choose $C_{5}>0$ as large
as we wish (but fixed). such that $\pr{\kregevent c}\le\exp\left(-C_{5}r^{4}\right)$
for $r$ sufficiently large. Finally, we conclude that for $\kappa\le Cr^{2}$,
we have
\[
\pr{F\left(p\right)\cap L\left(p,\varphi,\kappa\right)}\le\pr{F\left(p\right)}\exp\left(-\frac{C}{\Dnorm{\varphi}}\cdot\kappa^{2}+O\left(r^{2}\log^{2}r\right)\right).
\]
This completes the proof of Theorem \ref{thm:devs_in_emp_meas} in
the case $p\in\left[0,1\right)$. The proofs of the other cases go
along the same lines, and we leave them to the reader. We note that
in the case $p\ge e$, the constant in the statement of the theorem
may depend on $p$.
\begin{rem}
Recall that $F\left(0\right)=F\left(0;r\right)=H_{r}$ is the hole
event for $\left\{ \left|z\right|<r\right\} $. Let $\varepsilon\in\left(r^{-2},1\right)$
and $\gamma\in\left(1,2\right]$. Theorem \ref{thm:few_zeros_in_the_gap}
in Section \ref{sec:intro} follows from Theorem \ref{thm:devs_in_emp_meas}
by considering a positive (say radial) test function $\varphi=\varphi_{\varepsilon}$,
such that
\[
\varphi\left(z\right)=\varphi\left(\left|z\right|\right)=\begin{cases}
1 & \,1+\varepsilon\le\left|z\right|\le\sqrt{e}-\varepsilon;\\
0 & \,\left|z\right|\le1\mbox{ or }\left|z\right|\ge\sqrt{e}.
\end{cases}
\]
We note that one can construct such a $\varphi$ so that it satisfies
$\Dnorm{\varphi}=O\left(\varepsilon^{-1}\right)$ and $\omega\left(\varphi,t\right)=O\left(\varepsilon^{-1}t\right)$.
This implies
\begin{eqnarray*}
\condpr{n_{\gef}\left(\left\{ r\left(1+\varepsilon\right)\le\left|z\right|\le\sqrt{e}r\left(1-\varepsilon\right)\right\} \right)>r^{\gamma}}{F\left(0\right)} & \le & \condpr{\linstats{\gef}{\varphi}r>r^{\gamma}}{F\left(0\right)}\\
 & \le & \exp\left(-\frac{C}{\Dnorm{\varphi}}r^{2\gamma}\right)\\
 & \le & \exp\left(-C\varepsilon r^{2\gamma}\right),
\end{eqnarray*}
provided $r^{\gamma}>C_{5}\sqrt{\Dnorm{\varphi}}r\log r$ (which is
satisfied if $\gamma\in\left(1+\frac{1}{2}\log\frac{1}{\varepsilon}\left(\log r\right)^{-1},2\right]$,
and $r$ is sufficiently large).
\end{rem}

\subsection{Convergence of the counting measure - Proof of Theorem \ref{thm:conv_emp_meas_for_hole}}

The proof of Theorem \ref{thm:conv_emp_meas_for_hole} is straightforward
and we include it for completeness. We will need the following result,
we leave its simple proof to the reader.
\begin{claim}
\label{claim:bound_for_expec} Let $X$ be a real random variable
with finite mean, and $a\in\bbr$, $T\ge0$. We have
\[
\exabs{X-a}\le T+\int_{0}^{\infty}\pr{\left|X-a\right|>s+T}\,\dd s.
\]
We write $X_{\varphi}=X_{\varphi}\left(r\right)=\left.\linstats{\gef}{\varphi}r\right|_{F\left(p\right)}$
for the random variable $\linstats{\gef}{\varphi}r$ conditioned on
the event $F\left(p\right)$. Applying Claim \ref{claim:bound_for_expec}
to $X=X_{\varphi}$, $a=r^{2}\int_{\bbc}\varphi\left(w\right)\,\dd\minmeasglob p\left(w\right)$,
$T=Cr\log^{2}r$, and using Theorem \ref{thm:devs_in_emp_meas}, we
get
\begin{eqnarray*}
\left|\ex{X_{\varphi}}-r^{2}\int_{\bbc}\varphi\left(w\right)\,\dd\minmeasglob p\left(w\right)\right| & \le & \ex{\left|X_{\varphi}-r^{2}\int_{\bbc}\varphi\left(w\right)\,\dd\minmeasglob p\left(w\right)\right|}\\
 & \le & C\left(r\log^{2}r+\sqrt{\Dnorm{\varphi}}\right).
\end{eqnarray*}
Recall that $\cntmeas{\condzeroset r^{p}}$ is the counting measure
of the zeros of $\gef$ on the event $F\left(p;r\right)$, and denote
by $\cntmeas{\widetilde{\calZ}_{r}^{p}}=\frac{1}{r^{2}}\cntmeas{\condzeroset r^{p}}\left(\frac{\cdot}{r}\right)$
the scaled counting measure. In order to prove 
\[
\cntmeas{\widetilde{\calZ}_{r}^{p}}\xrightarrow[r\to\infty]{v}\minmeasglob p\quad\mbox{in distribution,}
\]
we have to show that for every $\phi\in C_{0}\left(\bbc\right)$ a
continuous test function with compact support, the random variable
\[
\int_{\bbc}\phi\left(w\right)\,\dd\cntmeas{\widetilde{\calZ}_{r}^{p}}\left(w\right)=\frac{1}{r^{2}}\left.\linstats{\gef}{\phi}r\right|_{F\left(p\right)}=r^{-2}X_{\phi},
\]
converges in distribution to $\int_{\bbc}\phi\left(w\right)\,\dd\minmeasglob p\left(w\right)$
(since the limit is a constant, this is the same as convergence in
probability). Suppose $\phi$ is supported on the disk $\disc 0B$,
where $B\ge1$, and let $\varphi\in C_{0}^{2}\left(\bbc\right)$ be
a smooth test function, supported on the disk $\disc 0{B+1}$, such
that $\left|\varphi-\phi\right|\le\delta$. In particular, we have
\begin{eqnarray*}
\left|X_{\phi}-X_{\varphi}\right| & \le & \delta\cdot n_{\gef}\left(B+1\right),\\
\left|\int_{\bbc}\phi\left(w\right)\,\dd\minmeasglob p\left(w\right)-\int_{\bbc}\varphi\left(w\right)\,\dd\minmeasglob p\left(w\right)\right| & \le & C\delta B.
\end{eqnarray*}
By Theorem \ref{thm:very_large_fluct}, we have
\[
\condpr{n_{\gef}\left(B+1\right)>C_{B,p}r^{2}}{F\left(p\right)}\le\exp\left(-Cr^{4}\right),
\]
for some constant $C_{B,p}$, depending only on $B$ and $p$. Therefore,
\begin{gather*}
\pr{\left|r^{-2}X_{\phi}-\int_{\bbc}\phi\left(w\right)\,\dd\minmeasglob p\left(w\right)\right|>\varepsilon}\quad\quad\quad\quad\quad\quad\quad\quad\quad\quad\quad\quad\quad\quad\quad\quad\quad\quad\quad\quad\quad\quad\quad\quad\quad\\
\le\pr{\left|r^{-2}X_{\varphi}-\int_{\bbc}\varphi\left(w\right)\,\dd\minmeasglob p\left(w\right)\right|>\varepsilon-C\delta B}+\pr{\left|X_{\phi}-X_{\varphi}\right|>\varepsilon r^{2}}\quad\quad\quad\quad\\
\le\pr{\left|r^{-2}X_{\varphi}-\int_{\bbc}\varphi\left(w\right)\,\dd\minmeasglob p\left(w\right)\right|>\varepsilon-C\delta B}+\condpr{n_{\gef}\left(B+1\right)\ge\frac{\varepsilon}{\delta}r^{2}}{F\left(p\right)}.
\end{gather*}
Choosing $\delta$ sufficiently small, depending on $\varepsilon$,
$\phi$, and $p$, and using Theorem \ref{thm:devs_in_emp_meas},
we find that
\[
\pr{\left|r^{-2}X_{\phi}-\int_{\bbc}\phi\left(w\right)\,\dd\minmeasglob p\left(w\right)\right|>\varepsilon}\le\exp\left(-C_{\varepsilon,\phi,p}r^{4}\right).
\]
This concludes the proof of Theorem \ref{thm:conv_emp_meas}.
\end{claim}

\section{\label{sec:discussion}Discussion}

In this paper, we considered the GEF,
\[
\gef\left(z\right)=\sum_{k=0}^{\infty}\xi_{k}\frac{z^{k}}{\sqrt{k!}},\quad z\in\bbc.
\]
The zero set of $\gef$ is the only translation invariant zero set
of a Gaussian entire function up to scaling (and multiplication by
a non-random entire function with no zeros, see \cite[Sec. 2.5]{hough2009zeros}).
We mention that there exist similar constructions for other domains
with transitive groups of isometries (the hyperbolic plane, the Riemann
sphere, the cylinder and the torus, see \cite[Sec. 2.3]{hough2009zeros}
for some examples).

\subsection{Asymptotic probability of large fluctuations in the number of zeros}

\subsubsection{The Hyperbolic GAF}

The hyperbolic GAF is the following Gaussian Taylor series,
\[
F_{D}\left(z\right)=\sum_{k=0}^{\infty}\xi_{k}z^{k},\quad\left|z\right|<1.
\]
It is known that its zero set is invariant with respect to the isometries
of the unit disk (\cite[Sec. 2.3]{hough2009zeros}). Peres and Virág
\cite{peres2005zeros} proved that this zero set is a determinantal
point process (see \cite[Chap. 4]{hough2009zeros}); this is the only
example of this type). Denote by $n_{F_{D}}\left(r\right)$ the number
of zeros of $F_{D}$ in $D\left(0,r\right)$ ($0<r<1$). Using the
representation of $n_{F_{D}}\left(r\right)$ as a sum of independent
Bernoulli random variables, they found the asymptotics of the hole
probability, as $r\to1$ (see \cite[Corollary 5.1.8.]{hough2009zeros}).
More recent results about the hole probability for GAFs in the unit
disk can be found in \cite{buckley2018hole,skaskiv2013probability}.

\subsubsection{Some known results for the GEF and other Gaussian entire functions}

Let us denote by $n\left(r\right)=n_{\gef}\left(r\right)$ the number
of zeros of the GEF inside the disk $\disc 0r$. As we mentioned in
Section \ref{sec:upp_bnd_for_larg_flucs}, the Edelman-Kostlan formula
implies that $\ex{n\left(r\right)}=r^{2}.$ In the paper \cite{sodin2005random},
Sodin and Tsirelson considered large fluctuations in the number of
zeros of the GEF, and proved that for every $\delta\in\left(0,\frac{1}{4}\right]$,
\[
\pr{\left|n\left(r\right)-r^{2}\right|\ge\delta r^{2}}\le\exp\left(-c\left(\delta\right)r^{4}\right),\quad\mbox{as }r\to\infty,
\]
with some unspecified positive constant $c\left(\delta\right)$. In
the case where the GEF has no zeros in the disk $\disc 0r$ (i.e.
the `hole' event) they showed $\pr{n\left(r\right)=0}\ge\exp\left(-Cr^{4}\right)$.
In the paper \cite{nishry2010asymptotics}, the second author found
that the logarithmic asymptotics of the hole probability are given
by
\[
\log\pr{n\left(r\right)=0}=-\frac{e^{2}}{4}r^{4}+o\left(r^{4}\right),\quad r\to\infty.
\]
This result was later generalized by the second author to include
entire functions represented by Gaussian Taylor series with arbitrary
coefficients (see \cite{nishry2012hole,nishry2013topics}).

\subsubsection{Large fluctuations results for the Ginibre ensemble}

Let us denote by $\cntmeas{\calG}$ the random counting measure of
the infinite Ginibre ensemble. It is known that this process is a
determinantal point process. In particular, for a compact set $K\subset\bbc$,
the random variable $\cntmeas{\calG}\left(K\right)$ can be expressed
as a sum of \textit{independent} Bernoulli random variables (\cite[Theorem 4.5.3 and Remark 4.5.4]{hough2009zeros}).
Shirai \cite{shirai2006large} proved the following result (corresponding
to our Theorem \ref{thm:very_large_fluct}):
\[
\pr{\cntmeas{\calG}\left(\disc 0r\right)=\left\lfloor pr^{2}\right\rfloor }=\exp\left(-G_{p}\cdot r^{4}+o\left(r^{4}\right)\right),\quad\mbox{as }r\to\infty,
\]
where
\[
G_{p}=\left|\int_{1}^{p}\left(1-x+x\log x\right)\,\dd x\right|.
\]
This provides a rigorous proof for some of the results of the paper
\cite{jancovici1993large} (obtained for the finite Ginibre ensemble,
in particular). The graphs of the constants $G_{p}$ and $Z_{p}$
are shown in Figure \ref{fig:GpZp}. Recently, Adhikari and Reddy
\cite{adhikari2016hole} found the asymptotics of the hole probability
for non-circular domains (for both the finite and infinite Ginibre
ensembles). We plan to consider this problem for the GEF in a future
paper.
\begin{center}
\begin{tabular}{cc}
\includegraphics[scale=0.45]{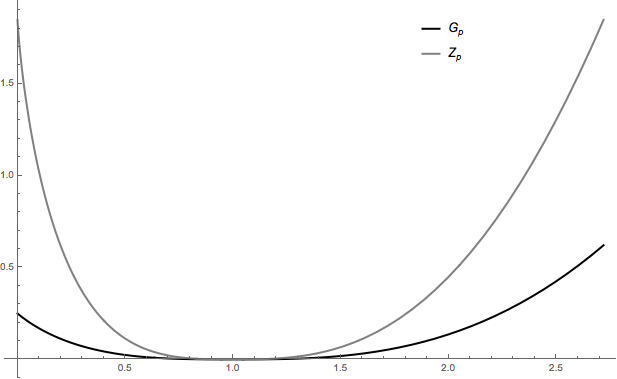} & \tabularnewline
\begin{tabular}{c}
Figure 8.1\manuallabel{fig:GpZp}{8.1} - The constants $G_{p}$ and
$Z_{p}$, $p\in\left[0,e\right]$.\tabularnewline
\end{tabular} & \tabularnewline
\end{tabular}
\par\end{center}

\subsubsection{Some related results for the Gaussian unitary ensemble (GUE)}

A problem similar to ours has been studied in the physical literature
(\cite{majumdar2011many}) in the one-dimensional setting. More precisely,
consider Hermitian Gaussian random matrices (sampled from the GUE
ensemble). Since the matrix is Hermitian, the eigenvalues are real,
and the point process of eigenvalues is one-dimensional. The problem
which is analogous to ours, is to find the limiting conditional distribution
of the eigenvalues, given that there is a ``gap'' in the (macroscopic)
interval $\left(-w\sqrt{N},w\sqrt{N}\right)$, where the dimension
of the matrix $N$ goes to infinity, and $w>0$ is a fixed number.
Considering a constrained variational problem somewhat similar to
our case, the authors are able to obtain a description of the minimizing
measure. An important feature of this minimizing measure is that it
has a density with respect to the Lebesgue measure (unlike the two-dimensional
setting, where we find the appearance of a singular component in both
the Ginibre and the GEF zero ensembles). Furthermore, there is no
forbidden region, compared to our result in the case of the GEF zero
ensemble.

\subsection{The Jancovici-Lebowitz-Manificat Law}

In the paper \cite{nazarov2008jancovici}, Nazarov, Sodin, and Volberg
studied a wider range of fluctuations in the random variable $n\left(r\right)$.
For fixed $b>\frac{1}{2}$ and any $\varepsilon>0$, they obtained
the following result
\[
-r^{\psi\left(b\right)+\varepsilon}\le\log\pr{\left|n\left(r\right)-r^{2}\right|>r^{b}}\le-r^{\psi\left(b\right)-\varepsilon},\quad r\ge r_{0}\left(b,\varepsilon\right),
\]
where
\[
\psi\left(b\right)=\begin{cases}
2b-1 & \,\frac{1}{2}<b\le1;\\
3b-2 & \,1\le b\le2;\\
2b & \,b\ge2.
\end{cases}
\]
Some of the cases were previously proved by Krishnapur in \cite{krishnapur2006overcrowding}
(and also in \cite{sodin2005random}), in particular he showed
\[
\log\pr{n\left(r\right)>r^{b}}=-\left(\frac{b}{2}-1\right)\left(1+o\left(1\right)\right)r^{2b}\log r,\quad b>2,\quad r\to\infty.
\]
These results are in agreement with a law discovered earlier by Jancovici,
Lebowitz, and Manificat in their physical paper \cite{jancovici1993large}.
This paper considered charge fluctuations of a one-component Coulomb
system of particles of one sign embedded into a uniform background
of the opposite sign (the finite Ginibre ensemble being a special
case).

Our methods allows us to also consider smaller fluctuations in $n\left(r\right)$.
For fixed constants $a,b$, with $a>0$ and $b\in\left(\frac{4}{3},2\right)$,
we have
\begin{equation}
\pr{n\left(r\right)=\left\lfloor r^{2}-ar^{b}\right\rfloor },\,\pr{n\left(r\right)=\left\lfloor r^{2}+ar^{b}\right\rfloor }=\exp\left(-\frac{2a^{3}}{3}\cdot r^{3b-2}\left(1+o\left(1\right)\right)\right),\quad r\to\infty.\label{eq:JLM_fluctuations_for_GEF}
\end{equation}
We can actually obtain the lower bound for $b\in\left(1,2\right)$.
Previously, Krishnapur (\cite{krishnapur2006overcrowding}) found
the lower bound
\[
\pr{n\left(r\right)=\left\lfloor r^{2}+ar^{b}\right\rfloor }\ge\exp\left(-a^{3}r^{3b-2}\left(1+o\left(1\right)\right)\right),\quad b\in\left(1,2\right),\quad r\to\infty.
\]
It is plausible the result (\ref{eq:JLM_fluctuations_for_GEF}) holds
in the whole range $b\in\left(1,2\right)$. 

\subsection{Large deviations and the modified weighted energy functional $\protect\funccont{\nu}$}

Let us fix $\alpha\ge e,$ and for $N\in\bbn$ write $L=\sqrt{\alpha^{-1}N}$.
Consider the following polynomials with independent standard complex
Gaussian coefficients $\xi_{k}$,
\[
\truncscaledpoly\left(z\right)=\sum_{k=0}^{N}\xi_{k}\frac{\left(Lz\right)^{k}}{\sqrt{k!}},\quad z\in\bbc.
\]
Denote by $\mu_{N}=\frac{1}{N}\sum_{j=1}^{N}\delta_{z_{j}}$ the empirical
measure of the zeros of $\truncscaledpoly$. In \cite{zeitouni2010large},
Zeitouni and Zeldich prove a large deviation principle (LDP) for the
sequence of measures $\mu_{N}$ (and for more general Gaussian polynomials).
For $\nu\in\probmeas$, a probability measure on $\bbc$, let $\logpot{\nu}z$,
$\logenerg{\nu}$ be its logarithmic potential and logarithmic energy,
respectively (see the notation section for the definitions). In addition,
let us define the following functional
\begin{equation}
\funccont{\nu}=\funccontparam{\alpha}{\nu}=2\sup_{w\in\bbc}\left\{ \logpot{\nu}w-\frac{\left|w\right|^{2}}{2\alpha}\right\} -\logenerg{\nu},\label{eq:LDP_functional}
\end{equation}
which in the terminology of large deviations is called the \textit{rate
function}. The LDP means that for a Borel subset $\calC\subset\probmeas$,
we have
\begin{equation}
-\inf_{\nu\in\calC^{\circ}}\funccont{\nu}+A_{\alpha}\le\liminf_{N\to\infty}\frac{1}{N^{2}}\log\pr{\mu_{N}\in\calC}\le\limsup_{N\to\infty}\frac{1}{N^{2}}\log\pr{\mu_{N}\in\calC}\le-\inf_{\nu\in\overline{\calC}}\funccont{\nu}+A_{\alpha},\label{eq:LDP_emp_meas_poly_zeros}
\end{equation}
where $\overline{\calC}$ (resp. $\calC^{\circ}$) is the closure
(resp. interior) of $\calC$ in the weak topology, and $A_{\alpha}=\frac{\log\alpha}{2}-\frac{3}{4}$.
\begin{rem}
This functional was introduced for the first time in the paper \cite{zeitouni2010large}.
In the papers \cite{BenArous1997,arous1998large,hiai1998maximizing}
on large deviations for (Gaussian) random matrices, the following
functional appears
\[
\funcenerg{\nu}=J_{\alpha}\left(\nu\right)=\int_{\bbc}\frac{\left|w\right|^{2}}{\alpha}\,\dd\nu\left(w\right)-\logenerg{\nu}.
\]
In potential theory, the functional $\funcenerg{\nu}$ is known as
the \textit{weighted energy functional} (see \cite{saff2013logarithmic}).
\end{rem}
Let $n_{\truncscaledpoly}\left(D\right)$ be the number of zeros of
$\truncscaledpoly$ in the unit disk. Combining the LDP with the results
of Section \ref{sec:sol_of_energy_prob}, gives (for a fixed $\alpha$)
\begin{eqnarray*}
\lim_{N\to\infty}\frac{1}{N^{2}}\log\pr{n_{\truncscaledpoly}\left(D\right)\le pL^{2}} & = & -Z_{p},\quad p\in\left(0,1\right),\\
\lim_{N\to\infty}\frac{1}{N^{2}}\log\pr{n_{\truncscaledpoly}\left(D\right)\ge pL^{2}} & = & -Z_{p},\quad p\in\left(1,\alpha\right),
\end{eqnarray*}
where $Z_{p}$ is the constant appearing in Theorem \ref{thm:very_large_fluct}.
In the case $p=0$, the LDP can only give a non-trivial upper bound,
since the set $\calF_{0}=\setd{\nu\in\probmeas}{\nu\left(D\right)=0}$
has empty interior. Theorem \ref{thm:very_large_fluct} can be seen
as an effective version of the LDP for the zeros of the GEF (for these
particular questions).

\subsection{The conditional distribution of the zeros}

In the context of large deviations theory, the convergence of the
empirical measure to a limit measure under conditioning is called
the \textit{Gibbs conditioning principle} (see \cite{dembo1996refinements},
\cite[Sec. 7.3]{dembozeitouni2010book}). This limit measure is given
by the minimizer of a rate function under the constraint. In our case,
the measure $\minmeasglob p$ is the limit of measures $\alpha\minmeassup p{\alpha}$
as $\alpha\to\infty$. Here the probability measures $\minmeassup p{\alpha}$
are the minimizers of the functional $\funccontparam{\alpha}{\nu}$
in (\ref{eq:LDP_functional}).

The paper \cite{jancovici1993large} describes in particular the limiting
conditional distribution for the finite Ginibre ensemble (i.e. the
minimizers of the functional $J_{\alpha}\left(\nu\right)$). One obtains
the following limiting measures:
\[
\dd\minmeasaltsup p{\alpha}\left(z\right)=\begin{cases}
\frac{1}{\alpha}\left[\ind{\left\{ \left|z\right|\le\sqrt{p}\right\} }z+\ind{\left\{ 1\le\left|z\right|\le\sqrt{\alpha}\right\} }z\right]\cdot\frac{\dd\lebmeas\left(z\right)}{\pi}+\frac{1-p}{\alpha}\dd\subslebmeas{\left|z\right|=1} & \,p\in\left[0,1\right);\\
\frac{1}{\alpha}\left[\ind{\left\{ \left|z\right|\le1\right\} }z+\ind{\left\{ \sqrt{p}\le\left|z\right|\le\sqrt{\alpha}\right\} }z\right]\cdot\frac{\dd\lebmeas\left(z\right)}{\pi}+\frac{p-1}{\alpha}\dd\subslebmeas{\left|z\right|=1} & \,p\in\left(1,\alpha\right).
\end{cases}
\]
We see there is no additional ``forbidden region'' for the eigenvalues.
In general a gap appears for the Ginibre ensemble (only) in the region
where there are less points than the expected value (see \cite{armstrong2014remarks}).
In the case of the disk and its complement, we showed that a gap appears
in both regions for the GEF. 

\subsubsection{Simulation of the conditional distribution}

It is possible to simulate the conditional distribution of the zeros
(say on the hole event) using a modified Metropolis-Hastings algorithm
\cite{landau2014guide}, which takes into account the constraint.
For the results of such a simulation of the Ginibre ensemble, see
for example \cite[Fig. 8]{ghosh2018point}. However, it seems like
this method is only efficient in practice for a few hundreds of the
GEF zeros.

To produce the figures in Section \ref{sec:intro} we used two different
methods. Figure \ref{fig:GEFcondzeros} (zeros conditioned on a hole)
is created using the ideas of the proof of the lower bound of Theorem
\ref{thm:very_large_fluct} in Section \ref{sec:low_bnd_for_large_flucs}.
Since standard complex Gaussians which are conditioned to be very
small in modulus are approximately uniform, we generate such random
variables, in a way that the GEF will have no zeros inside the disk
of radius $r=13$. Figure \ref{fig:GinibreCond} is created by simply
moving the eigenvalues of a large random Ginibre matrix from the disk
$\left\{ \left|z\right|<13\right\} $ to the boundary (for a more
convincing simulation see the reference above).

\subsection{Large deviations for linear statistics}

Let $\varphi\in C_{0}^{2}\left(\bbc\right)$ be an arbitrary compactly
supported test function. The following result is known as Offord\textquoteright s
estimate (see \cite[Theorem 7.1.1]{hough2009zeros}, \cite{sodin2000zeros}),
\begin{eqnarray*}
\pr{\left|\linstats{\gef}{\varphi}r-\frac{r^{2}}{\pi}\int_{\bbc}\varphi\left(w\right)\,\dd\lebmeas\left(w\right)\right|\ge\lambda} & = & \pr{\left|\int_{\bbc}\varphi\left(\frac{w}{r}\right)\,\dd n_{\gef}\left(w\right)-\frac{1}{\pi}\int_{\bbc}\varphi\left(\frac{w}{r}\right)\,\dd\lebmeas\left(w\right)\right|\ge\lambda}\\
 & \le & 3\exp\left(\frac{-\pi\lambda}{\left\Vert \Delta\varphi\right\Vert _{L^{1}\left(\lebmeas\right)}}\right),\quad\lambda>0,
\end{eqnarray*}
where we used $\left\Vert \Delta\varphi\left(\frac{\cdot}{r}\right)\right\Vert _{L^{1}\left(\lebmeas\right)}=\left\Vert \Delta\varphi\right\Vert _{L^{1}\left(\lebmeas\right)}$.
We mention that this bound is valid in general for Gaussian analytic
functions. The following bound can be derived from our proof of Theorem
\ref{thm:devs_in_emp_meas},
\[
\pr{\left|\linstats{\gef}{\varphi}r-\frac{r^{2}}{\pi}\int_{\bbc}\varphi\left(w\right)\,\dd\lebmeas\left(w\right)\right|\ge\lambda}\le\exp\left(-\frac{C}{\left\Vert \nabla\varphi\right\Vert _{L^{2}\left(\lebmeas\right)}^{2}}\cdot\lambda^{2}+O\left(r^{2}\log^{2}r\right)\right),\quad\lambda>0.
\]

\appendix

\section{\label{sec:joint_density}The Joint Distribution of The Zeros}

Let $L>0$ and $N\in\bbn^{+}$. We want to find the joint probability
density of the zeros of the polynomial
\[
P\left(z\right)=\truncscaledpoly\left(z\right)=\sum_{k=0}^{N}\xi_{k}\frac{\left(Lz\right)^{k}}{\sqrt{k!}},
\]
where $\xi_{k}$ are i.i.d. standard complex Gaussians. This requires
a change of variables, from the coefficients to the zeros (cf. the
more general \cite[Proposition 3]{zeitouni2010large}). We use the
fact that the Jacobian determinant of this transformation can be expressed
in a simple way in terms of the zeros.
\begin{lem}
\label{lem:poly_change_of_vars}Let $\vec z=\lvec zN$ be the zeros
of $\truncscaledpoly\left(z\right)$ in uniform random order. The
joint distribution of $\vec z,$ w.r.t. Lebesgue measure on $\bbc^{N}$,
is given by
\[
f\left(\vec z\right)=f\left(z_{1},\dots,z_{N}\right)=\normalconst\left|\Delta\left(\vec z\right)\right|^{2}\left(\int_{\bbc}\prod_{j=1}^{N}\left|w-z_{j}\right|^{2}\,\dd\refmeas\left(w\right)\right)^{-\left(N+1\right)},
\]
where
\[
A_{L}^{N}=\frac{N!\cdot\prod_{j=1}^{N}j!}{\pi^{N}L^{N\left(N+1\right)}}=\exp\left(\frac{1}{2}N^{2}\log\left(\frac{N}{L^{2}}\right)-\frac{3}{4}N^{2}+O\left(N\left(\log N+\log L\right)\right)\right).
\]
\end{lem}
\begin{rem}
Recall that $\left|\Delta\left(\vec z\right)\right|^{2}=\prod_{j\ne k}\left|z_{j}-z_{k}\right|$
and $\dd\refmeas\left(w\right)=\frac{L^{2}}{\pi}e^{-L^{2}\left|w\right|^{2}}\,\dd m\left(w\right)$,
where $m$ is Lebesgue measure on $\bbc$.
\end{rem}
\begin{proof}
Let $\vec{\xi}=\lvecl{\xi}0N$. The joint density of $\vec{\xi}$
w.r.t. Lebesgue measure on $\bbc^{N+1}$ is given by
\begin{equation}
g\left(\vec{\xi}\right)=\frac{1}{\pi^{N+1}}\exp\left(-\sum_{k=0}^{N}\left|\xi_{k}\right|^{2}\right),\quad\vec{\xi}\in\bbc^{N+1}.\label{eq:joint_density_xis}
\end{equation}
We now define the monic polynomials corresponding to $\truncscaledpoly\left(z\right)$,
\[
q_{\vec z}\left(z\right)=\frac{\truncscaledpoly\left(z\right)}{\xi_{N}\cdot\frac{L^{N}}{\sqrt{N!}}}=z^{N}+b_{N-1}z^{N-1}+\dots+b_{0}=\prod_{j=1}^{N}\left(z-z_{j}\right),
\]
where
\[
b_{k}=\frac{\xi_{k}}{\xi_{N}}\cdot\frac{\sqrt{N!}}{\sqrt{k!}}\cdot L^{k-N},\quad k\in\left\{ 0,\dots N-1\right\} .
\]
The Jacobian of the map $T_{1}:\vec z\mapsto\vec b$ that takes the
zeros of the polynomial to the coefficients is given by $\left|\Delta\left(\vec z\right)\right|^{2}$
(see for example \cite[Lemma 1.1.1]{hough2009zeros}). Clearly, the
Jacobian of the (complex) linear map $T_{2}:\vec b\mapsto\vec{\xi}$
is given by
\[
\prod_{k=0}^{N-1}\frac{N!}{\left|\xi_{N}\right|^{2}k!\cdot L^{2\left(N-k\right)}}=\frac{\left(N!\right)^{N+1}}{\left|\xi_{N}\right|^{2N}\prod_{k=1}^{N}k!\cdot L^{N\left(N+1\right)}}\defeq\left|\xi_{N}\right|^{-2N}\cdot A^{\prime}.
\]
Therefore, after doing the change of variables from $\lvecl{\xi}0N$
to $\left(\vec z,\xi_{N}\right)$, and using Claim \ref{claim:integ_wrt_ref_meas},
we arrive at the joint density,
\[
g^{\prime}\left(\vec z,\xi_{N}\right)=\frac{1}{\pi^{N+1}}\cdot\frac{\left|\xi_{N}\right|^{2N}}{A^{\prime}}\cdot\left|\Delta\left(\vec z\right)\right|^{2}\cdot\exp\left(-\left|\xi_{N}\right|^{2}\frac{L^{2N}}{N!}\cdot\int_{\bbc}\left|q_{\vec z}\left(w\right)\right|^{2}\,\dd\refmeas\left(w\right)\right).
\]
We now integrate out $\xi_{N}$, and use the fact
\[
\frac{1}{\pi}\int_{\bbc}\left|w\right|^{2N}e^{-B\left|w\right|^{2}}\dd\lebmeas\left(w\right)=N!\cdot B^{-\left(N+1\right)},
\]
to get
\begin{eqnarray*}
f\left(\vec z\right) & = & \frac{1}{\pi^{N}}\cdot\frac{1}{A^{\prime}}\cdot N!\left(\frac{N!}{L^{2N}}\right)^{N+1}\left|\Delta\left(\vec z\right)\right|^{2}\left[\int_{\bbc}\left|q_{\vec z}\left(w\right)\right|^{2}\,\dd\refmeas\left(w\right)\right]^{-\left(N+1\right)}\\
 & = & \frac{N!\cdot\prod_{k=1}^{N}k!}{\pi^{N}L^{N\left(N+1\right)}}\left|\Delta\left(\vec z\right)\right|^{2}\left[\int_{\bbc}\left|q_{\vec z}\left(w\right)\right|^{2}\,\dd\refmeas\left(w\right)\right]^{-\left(N+1\right)}.
\end{eqnarray*}
Stirling's approximation for the factorial shows that
\[
N!\cdot\prod_{k=1}^{N}k!=\frac{N^{2}\log N}{2}-\frac{3}{4}N^{2}+O\left(N\log N\right).
\]
\end{proof}
\begin{claim}
\label{claim:integ_wrt_ref_meas}Let $L>0$ and $k\in\bbn$. We have
\begin{equation}
\int_{\bbc}\left|w\right|^{2k}\,\dd\refmeas\left(w\right)=\frac{k!}{L^{2k}}.\label{eq:moment_ref_meas}
\end{equation}
In addition, for $N\in\bbn^{+}$ let $\truncscaledpoly\left(z\right)=\sum_{k=0}^{N}\xi_{k}\frac{\left(Lz\right)^{k}}{\sqrt{k!}}$,
and let $q_{\vec z}\left(z\right)$ be the corresponding monic polynomial.
Then,
\[
\int_{\bbc}\left|\truncscaledpoly\left(w\right)\right|^{2}\,\dd\refmeas\left(w\right)=\sum_{k=0}^{N}\left|\xi_{k}\right|^{2},
\]
and thus
\[
\int_{\bbc}\left|q_{\vec z}\left(w\right)\right|^{2}\,\dd\refmeas\left(w\right)=\left(\left|\xi_{N}\right|^{2}\frac{L^{2N}}{N!}\right)^{-1}\cdot\sum_{k=0}^{N}\left|\xi_{k}\right|^{2}.
\]
\end{claim}
\begin{proof}
Notice that for any $k\in\bbn$,
\begin{eqnarray*}
\int_{\bbc}\left|w\right|^{2k}\,\dd\refmeas\left(w\right) & = & \frac{L^{2}}{\pi}\int_{\bbc}\left|w\right|^{2k}e^{-L^{2}\left|w\right|^{2}}\,\dd\lebmeas\left(w\right)=2L^{2}\int_{0}^{\infty}t^{2k+1}e^{-L^{2}t^{2}}\,\dd t\\
 & = & L^{2}\int_{0}^{\infty}s^{k}e^{-L^{2}s}\,\dd s=\frac{k!}{L^{2k}}.
\end{eqnarray*}
Therefore, using the orthogonality of $z^{j}$ and $\overline{z}^{k}$
w.r.t. the measure $\refmeas$, we have
\[
\int_{\bbc}\left|\truncscaledpoly\left(w\right)\right|^{2}\,\dd\refmeas\left(w\right)=\sum_{k=0}^{N}\left|\xi_{k}\right|^{2}\frac{L^{2k}}{k!}\cdot\int_{\bbc}\left|w\right|^{2k}\,\dd\refmeas\left(w\right)=\sum_{k=0}^{N}\left|\xi_{k}\right|^{2}.
\]
\end{proof}
The following estimate is sometimes called the Bernstein-Markov property
of the measure $\refmeas$ (cf. \cite[pg. 3939]{zeitouni2010large}).
\begin{lem}
\label{lem:bern_markov}Let $L>0$ and let $h$ be a polynomial of
degree $N$. We have
\[
\sup_{w\in\bbc}\left\{ \left|h\left(w\right)\right|^{2}e^{-L^{2}\left|w\right|^{2}}\right\} \le\int_{\bbc}\left|h\left(z\right)\right|^{2}\,\dd\refmeas\left(z\right).
\]
\end{lem}
\begin{proof}
Introduce the reproducing kernel (with respect to $\refmeas$)
\[
\Pi_{N,L}\left(w,z\right)=\sum_{k=0}^{N}\frac{L^{2k}}{k!}\left(w\overline{z}\right)^{k},\quad z,w\in\bbc.
\]
It has the following basic properties:
\begin{eqnarray*}
1. &  & h\left(w\right)=\int_{\bbc}\Pi_{N,L}\left(w,z\right)h\left(z\right)\,\dd\refmeas\left(z\right),\\
\mathrm{2.} &  & \Pi_{N,L}\left(w,w\right)=\int_{\bbc}\left|\Pi_{N,L}\left(w,z\right)\right|^{2}\,\dd\refmeas\left(z\right),\\
3. &  & \Pi_{N,L}\left(w,w\right)\le e^{L^{2}\left|w\right|^{2}}.
\end{eqnarray*}
The first two properties follow from (\ref{eq:moment_ref_meas}) and
Property $3$ is clear from the definition. Applying the Cauchy-Schwarz
inequality, we find
\begin{eqnarray*}
\left|h\left(w\right)\right|^{2} & \le & \left(\int_{\bbc}\left|h\left(z\right)\right|^{2}\,\dd\refmeas\left(z\right)\right)\left(\int_{\bbc}\left|\Pi_{N,L}\left(w,z\right)\right|^{2}\,\dd\refmeas\left(z\right)\right)\\
 & \le & \left(\int_{\bbc}\left|h\left(z\right)\right|^{2}\,\dd\refmeas\left(z\right)\right)e^{L^{2}\left|w\right|^{2}},
\end{eqnarray*}
where we used the fact $\refmeas$ is a probability measure.
\end{proof}

\section{\label{sec:poten_theory}Some background on Logarithmic Potential
Theory}

All the required background on weighted logarithmic potential theory
can be found in the book \cite{saff2013logarithmic}, notice that
we use here the opposite sign convention for the logarithmic potential
of a measure.

Let $\nu\in\probmeas$ be a probability measure. Consider the following
weighted energy functional
\[
\funcenerg{\nu}=J_{\alpha}\left(\nu\right)=\int_{\bbc}\frac{\left|w\right|^{2}}{\alpha}\,\dd\nu\left(w\right)-\logenerg{\nu},
\]
where $\alpha>0$ is a parameter. We recall that the logarithmic potential
and logarithmic energy of $\nu$ are given by:
\[
\logpot{\nu}z=\int_{\bbc}\log\left|z-w\right|\,\dd\nu\left(w\right),\quad\logenerg{\nu}=\int_{\bbc}\logpot{\nu}z\,\dd\nu\left(z\right)=\int_{\bbc^{2}}\log\left|z-w\right|\,\dd\nu\left(z\right)\dd\nu\left(w\right).
\]
The logarithmic energy $\logenerg{\nu}$ is an upper semi-continuous
and strictly concave functional, on measures with finite logarithmic
energy and compact support (\cite[Proposition 2.2]{hiai1998maximizing}).
Since $\int_{\bbc}\left|w\right|^{2}\,\dd\nu\left(w\right)$ is a
continuous linear functional, it follows that the functional $\funcenerg{\nu}$
is lower semi-continuous and strictly convex. 

It is known (see \cite[Example IV.6.2]{saff2013logarithmic}, but
notice the different scaling) that the global minimizer of this functional
is the uniform measure on the disk $\disc 0{\sqrt{\alpha}}$ which
we denote by $\eqmeas{\alpha}$. This measure is sometimes called
the equilibrium or extremal measure. An easy calculation shows
\[
\logpot{\eqmeas{\alpha}}z=\begin{cases}
\frac{\left|z\right|^{2}}{2\alpha}+\frac{\log\alpha}{2}-\frac{1}{2} & \,\left|z\right|\le\sqrt{\alpha};\\
\log\left|z\right| & \,\left|z\right|\ge\sqrt{\alpha},
\end{cases}\qquad\logenerg{\eqmeas{\alpha}}=\frac{\log\alpha}{2}-\frac{1}{4},
\]
and
\[
F_{\alpha}\defeq\int_{\bbc}\frac{\left|w\right|^{2}}{2\alpha}\,\dd\eqmeas{\alpha}\left(w\right)-\logenerg{\eqmeas{\alpha}}=\frac{1}{4}-\left(\frac{\log\alpha}{2}-\frac{1}{4}\right)=\frac{1}{2}-\frac{\log\alpha}{2}.
\]
Let $\calH$ be the set of all subharmonic functions $g\left(z\right)$
on $\bbc$ that are harmonic for large $\left|z\right|$, and $g\left(z\right)-\log\left|z\right|$
is bounded from above near $\infty$. By Theorem I.4.1 in \cite{saff2013logarithmic},
we have 
\[
\logpot{\eqmeas{\alpha}}z+F_{\alpha}=\frac{\left|z\right|^{2}}{2\alpha}=\sup\setd{g\left(z\right)}{g\in\calH\mbox{ and}\,g\left(w\right)\le\frac{\left|w\right|^{2}}{2\alpha},\,\forall\left|w\right|\le\sqrt{\alpha}}.
\]
We summarize the implications in the following claim.
\begin{claim}
\label{claim:sup_log_pot}Let $\nu\in\probmeas$ be a probability
measure with compact support, and define
\[
B\left(\nu\right)=2\sup_{w\in\bbc}\left\{ \logpot{\nu}w-\frac{\left|w\right|^{2}}{2\alpha}\right\} .
\]
We have
\[
B\left(\nu\right)=B_{\alpha}\left(\nu\right)\defeq2\sup_{\left|w\right|\le\sqrt{\alpha}}\left\{ \logpot{\nu}w-\frac{\left|w\right|^{2}}{2\alpha}\right\} .
\]
\end{claim}
\begin{proof}
Notice $\logpot{\nu}z$ is a subharmonic function, that is harmonic
for large $\left|z\right|$, and $\logpot{\nu}w-\log\left|z\right|$
is bounded from above near $\infty$. Since,
\[
\logpot{\nu}z-\frac{B_{\alpha}\left(\nu\right)}{2}\le\frac{\left|z\right|^{2}}{2\alpha},\quad z\in\disc 0{\sqrt{\alpha}},
\]
we have
\[
\logpot{\nu}z-\frac{B_{\alpha}\left(\nu\right)}{2}\le\logpot{\eqmeas{\alpha}}z+F_{\alpha}=\frac{\left|z\right|^{2}}{2\alpha},\quad z\in\bbc,
\]
which implies $B\left(\nu\right)\le B_{\alpha}\left(\nu\right)$.
\end{proof}
\bibliographystyle{amsalpha}
\bibliography{valdistrefs_v1_0}

\end{document}